\pdfoutput=1
\documentclass[reqno,a4paper]{article}
\usepackage[colorlinks=true,breaklinks=true,linkcolor=lightblue,citecolor=lightblue]{hyperref}
\usepackage{amssymb,amsmath,epsfig,graphics,psfrag,graphicx,color,cite,fancyhdr,subfigure}

\definecolor{lightblue}{rgb}{0.22,0.45,0.70}

\numberwithin{equation}{section}
\newcommand\cero{\boldsymbol{0}}

\newcommand\ff{\boldsymbol{f}}

\newcommand\bu{\boldsymbol{u}}
\newcommand\bv{\boldsymbol{v}}
\newcommand\bx{\boldsymbol{x}}

\newcommand\bz{\boldsymbol{z}}

\newcommand\bbbeta{\boldsymbol{\beta}}

\newcommand\cP{\mathcal{P}}
\newcommand\cT{\mathcal{T}}

\newcommand\cE{{\mathcal{E}}}
\newcommand\cF{{\mathcal{F}}}

\newcommand\N{\mathbb{N}}
\newcommand\R{\mathbb{R}}

\newcommand{\norm}[1]{\left\|#1\right\|}

\newcommand{\set}[1]{\left\{#1\right\}}

\newcommand\disp{\displaystyle}

\newcommand\DO{\partial\O}

\renewcommand\O{\Omega}
\newcommand\G{\Gamma}

\renewcommand\H{\mathrm{H}}
\renewcommand\L{\mathrm{L}}
\newcommand\Q{\mathrm{Q}}
\newcommand\Z{\mathrm{Z}}
\newcommand\X{\mathrm{X}}
\newcommand\LO{\L^2(\O)}

\newcommand\vdiv{\mathop{\mathrm{div}}\nolimits}
\newcommand\vgrad{\mathop{\mathrm{grad}}\nolimits}
\newcommand\hdivO{{\H(\mathrm{div};\O)}}

\newcommand\HsO{\H^s(\O)}

\newcommand\HusO{\H^{1+s}(\O)}

\newcommand\bn{\boldsymbol{n}}

\newcommand\bomega{\boldsymbol{\omega}}
\newcommand\btheta{\boldsymbol{\theta}}

\newcommand\be{\boldsymbol{e}}
\newcommand\bpi{\boldsymbol{\Pi}}
\newcommand\bxi{\boldsymbol{\xi}}
\newcommand\betta{\boldsymbol{\eta}}

\newtheorem{remark}{Remark}[section]
\newtheorem{lemma}{Lemma}[section]
\newtheorem{theorem}{Theorem}[section]
\newtheorem{proposition}{Proposition}
\newtheorem{corollary}{Corollary}

\newcommand\curl{\mathop{\mathbf{curl}}\nolimits}

\newcommand{\jump}[1]{\left[\![#1\right]\!]}
\newcommand{\jumpT}[1]{{\left[\![#1\right]\!]}_T}
\newcommand{\jumpN}[1]{{\left[\![#1\right]\!]}_N}
\newcommand{\avg}[1]{\left\{\!\!\{#1\right\}\!\!\}}

\def\CR{{\mathcal R}}
\def\CP{{\mathcal P}}

\newenvironment{proof}{\noindent{\it Proof.}}{\hfill$\square$}

\allowdisplaybreaks

\begin{document}

\title{Analysis and approximation of a vorticity-velocity-pressure 
formulation for the Oseen equations}

\author{{\sc Ver\'onica Anaya}\thanks{GIMNAP, Departamento de Matem\'atica,
Universidad del B\'io-B\'io, Concepci\'on, Chile. E-mail:  
{\tt vanaya@ubiobio.cl}.},  \ 
{\sc Afaf Bouharguane}\thanks{Institut de Math\'ematiques de Bordeaux, 
CNRS UMR 5251, Universit\'e de Bordeaux, 33405 Talence, France. E-mail: {\tt afaf.bouharguane@math.u-bordeaux.fr}.}, \ 
\ {\sc David Mora}\thanks{GIMNAP, Departamento de Matem\'atica, Universidad
del B\'io-B\'io, Concepci\'on, Chile and
Centro de Investigaci\'on en Ingenier\'ia Matem\'atica
(CI$^2$MA), Universidad de Concepci\'on, Concepci\'on, Chile.
E-mail: {\tt dmora@ubiobio.cl}.}, \ 
\ {\sc Carlos Reales}\thanks{Departamento de Matem\'aticas
y Estad\'isticas, Universidad de C\'ordoba, Monter\'ia, Colombia. E-mail: 
{\tt creales@correo.unicordoba.edu.co}.}, \\
\ {\sc Ricardo Ruiz-Baier}\thanks{Mathematical Institute,
University of Oxford,  OX2 6GG Oxford, UK. E-mail: {\tt ruizbaier@maths.ox.ac.uk}.},  \ 
{\sc Nour Seloula}\thanks{LMNO, CNRS UMR 6139,
Universit\'e de Caen, 5186 Caen, France. E-mail: {\tt nour-elhouda.seloula@unicaen.fr}.}, 
\ and {\sc Hector Torres}\thanks{Departamento de Matem\'aticas, Universidad de La Serena, 
La Serena, Chile. E-mail: {\tt htorres@userena.cl}.}
}

\date{\today}
\maketitle

\begin{abstract}
\noindent 
We introduce a family of mixed methods and discontinuous Galerkin discretisations designed to numerically solve the Oseen equations written in terms of velocity, vorticity, and Bernoulli pressure. The unique solvability of the continuous problem is addressed by invoking a global inf-sup property in an adequate abstract setting for non-symmetric systems. The proposed finite element schemes, which produce exactly divergence-free discrete velocities, are shown to be well-defined and optimal convergence rates are derived in suitable norms. In addition, we establish optimal rates of convergence for a class of discontinuous Galerkin schemes, which employ stabilisation. A set of numerical examples serves to illustrate salient features of these methods. 
\end{abstract}



\maketitle
\section{Introduction}
The Oseen equations  
stem from linearisation of the steady (or alternatively from the backward Euler time-discretisation 
of the transient) Navier-Stokes equations. Of particular appeal to us is their formulation in terms of 
fluid velocity, vorticity vector, and pressure. A diversity of discretisation methods is available to solve 
incompressible flow problems using these three fields as principal unknowns. Some recent examples 
include spectral elements \cite{amoura07,azaiez06} as well as stabilised and least-squares 
schemes \cite{amara07,bochev97} for Navier-Stokes; also several mixed and augmented methods for 
Brinkman  \cite{AGMRb15,anaya15a,anaya17b}, and a number of other discretisations 
specifically designed for Stokes 
flows \cite{anaya13,duan03,dubois03,gatica11,salaun15}. 

Both the implementation and the analysis of numerical schemes for 
Navier-Stokes are typically based on 
the Oseen linearisation. A few related contributions  (not only restricted to the velocity-pressure 
formulation) include for instance
\cite{bochev99}, that presents 
a least-squares  method 
for Navier-Stokes equations with vorticity-based first-order reformulation, and whose 
analysis exploits  
the elliptic theory of Agmon-"Douglas-Nirenberg. 
Conforming finite element methods exhibit optimal order of accuracy for 
diverse boundary conditions. We also mention the 
non-conforming exponentially
accurate least-squares spectral  method for Oseen
equations  proposed in \cite{mohapatra16}, where 
a suitable preconditioner is also proposed. 
In \cite{tsai05} the authors introduce a velocity-vorticity-pressure 
least-squares  finite element method for Oseen and Navier-Stokes equations with velocity
boundary conditions. They derive error estimates and reported a degeneracy of 
the convergence for large Reynolds numbers. A div least-squares
minimisation problem based on the stress-velocity-pressure formulation was 
introduced in \cite{cai16}. The study shows 
that the corresponding homogeneous least-squares functional is elliptic and 
continuous in suitable norms. Several first-order 
Oseen-type systems are analysed in \cite{chang07}, also including vorticity and total pressure in the formulation.

Discontinuous Galerkin (DG) methods have also been used to solve the Oseen problem,
as for example, in \cite{Cockburn03,Cockburn05}  for the case of Dirichlet boundary conditions. 
Compared with conforming finite elements, discretisations based on DG methods have a number of 
attractive, and well-documented features. These include  
 high order accuracy, being amenable for $hp$-adaptivity, relatively simple implementation on 
 highly unstructured meshes, and superior robustness when handling rough coefficients. 
We also mention the a priori error analysis of hybridisable DG schemes introduced  
in \cite{CockburnJSC13} for the Oseen equations. The family of DG methods we propose here 
has resemblance with those schemes, but exploits a three-field formulation described below.

This paper is concerned with mixed non-symmetric variational problems 
which will be  analysed using a global inf-sup argument. 
To do this, we conveniently restrict the set of equations to the space of divergence-free velocities, and apply
results from \cite{G2014} in order to prove that the equivalent resulting non-symmetric saddle-point
problem is well-posed. For the numerical approximation, we first consider
Raviart-Thomas elements of order $k\ge0$ for the velocity field, N\'ed\'elec elements
or order $k$ for the vorticity, and piecewise polynomials of degree $k$ without 
continuity restrictions, for the Bernoulli pressure.
We prove unique solvability of the discrete problem by adapting the same tools utilised
in the analysis of the continuous problem. In addition, the proposed family of Galerkin 
finite element methods 
turns out to be optimally convergent, under the common 
assumptions of enough regularity of the exact solutions to the continuous 
 problem. The method produces exactly divergence-free approximations 
of the velocity by construction;  thus preserving, at the discrete level, an essential constraint of the governing
equations. Next, inspired by the methods presented in \cite{Cockburn_SIAM2002,Cockburn03}, we present another scheme involving the discontinuous Galerkin discretisation of the $\curl$-$\curl$ and $\mathrm{grad}$-$\mathrm{div}$ operators. We prove the well-posedness of the DG scheme and derive error estimates under some solution regularity assumptions.

We have structured the contents of the paper in the following manner. Notation-related preliminaries 
are stated in the remainder of this Section. We then present the model problem 
as well as the three-field weak formulation and its solvability analysis in Section~\ref{sec:model}. 
The finite element discretisation is constructed in Section~\ref{sec:FE}, where we also 
derive the stability and convergence bounds. In Section~\ref{sec:DG}, we present the mixed DG formulation for the model problem. The well-posedness of the method and the error analysis are established in the same section. We close in Section~\ref{sec:numer} with a set of numerical tests 
that illustrate the properties of the proposed numerical schemes in a variety of scenarios.

Let $\O$ be a  bounded domain of $\R^3$ with Lipschitz boundary $\DO$.
Moreover, we assume that $\DO$ admits a
disjoint partition $\DO=\Gamma\cup \Sigma$. For any $s\geq 0$, the symbol 
$\norm{\cdot}_{s,\O}$ denotes 
the norm of the Hilbertian Sobolev spaces $\HsO$ or
$\HsO^3$, with the usual convention $\H^0(\O):=\LO$. For $s\geq 0$, we recall 
the definition of the space 
$$\H^s(\curl;\O):=\set{\btheta\in\HsO^3:\ \curl\btheta\in\HsO^3},$$ endowed with the norm
 $\norm{\btheta}^2_{\H^s(\curl;\O)}=\norm{\btheta}_{s,\O}^2+\norm{\curl\btheta}^2_{s,\O}$, and will denote 
$\H(\curl;\O)=\H^0(\curl;\O)$. Finally, $c$ and
$C$, with subscripts, tildes, or hats, will represent a generic constant
independent of the mesh parameter $h$.

\section{Statement and solvability of the continuous problem}\label{sec:model}
\paragraph{Oseen problem in terms of velocity-vorticity-pressure.}
A standard backward Euler time-discretisation of the classical Navier-Stokes 
equations, or a linearisation of the 
steady version of the problem combined with standard curl-div identities, 
 leads to the following set of equations, known as the 
Oseen equations (see \cite{oseen,gr-1986}): 
\begin{align}\label{NS-cont-vort}\begin{split}
\sigma\bu-\nu\Delta\bu+\curl\bu\times\bbbeta +\nabla p & = 
  \ff \qquad \mbox{ in } \O, \\
  \vdiv\bu & =  0 \qquad\, \mbox{ in } \O,  
 \end{split}\end{align}
where $\nu>0$ is the kinematic fluid viscosity, 
 $\sigma>0$ is inversely proportional to the time-step, $\bbbeta$ is 
 an adequate approximation of velocity to be made precise below, and 
the vector of external forces $\ff$ also absorbs the contributions related to previous time steps, or to fixed states in 
the linearisation procedure of the steady Navier-Stokes equations. As usual in this context, in the momentum 
equation we have conveniently introduced the Bernoulli (also known as 
dynamic) pressure $p := P + \frac{1}{2}\vert\bu\vert^2$, where $P$ is the actual fluid pressure. 

The structure of \eqref{NS-cont-vort} suggests to introduce the rescaled 
vorticity vector $\bomega:=\sqrt{\nu}\curl\bu$ as a new
unknown. Furthermore, 
in this study we focus on the case of zero normal velocities and zero 
tangential vorticity trace imposed on a part of the boundary $\Gamma\subset \partial\Omega$, 
whereas a non-homogeneous tangential velocity $\bu_{\Sigma}$ and a fixed Bernoulli pressure $p_\Sigma$ are set 
on the remainder of the boundary $\Sigma = \partial\Omega\setminus\Gamma$. 
Therefore, system \eqref{NS-cont-vort} can be recast in the form 
\begin{align}
\sigma\bu+\sqrt{\nu}\curl \bomega+\nu^{-1/2}\bomega\times\bbbeta +\nabla p  = 
  \ff, \quad 
  \bomega-\sqrt{\nu}\curl\bu  =  \cero, \quad \text{and} \quad 
\vdiv\bu & =  0 & \mbox{ in } \O, \nonumber  \\ 
  \bu\cdot\bn = 0 \quad \text{ and }  \quad \bomega\times\bn & = \cero&    
\mbox{ on } \G, \label{oseen-cont}\\ 
  p =p_\Sigma \quad \text{ and }  \quad   \bu\times\bn&=\bu_{\Sigma}&\mbox{ on } \Sigma,\nonumber
 \end{align}
where  $\bn$ stands for the outward unit normal
on $\partial\Omega$. Should the boundary $\Sigma$ have zero measure, 
the additional condition $ (p,1)_{\Omega,0} = 0$ is required to 
 enforce uniqueness of the 
Bernoulli pressure.

\paragraph{Defining a weak formulation.} 
Let us introduce the following functional spaces 
$$\H:=\{\bv\in\H(\vdiv;\O): \bv\cdot\bn=0\,\,\text{on}\,\,\G\}, \quad
\Z:=\{\btheta\in\H(\curl;\O): \gamma_{t}(\btheta)=\cero\,\,\text{on}\,\,\G\}, \quad 
 \text{and } \quad 
\Q:=\LO,$$
where the operator $\gamma_{t}$ is the tangential trace
operator on $\G$, defined by: $\gamma_{t}(\btheta)=\btheta\times\bn$.
Let us endow  $\H$ and $\Q$ with their natural norms. 
For the space $\Z$ however, we consider the following
viscosity-weighted norm: 
$$\Vert\btheta\Vert_{\Z}:=\left(\Vert\btheta\Vert_{0,\O}^2+
\nu\Vert\curl\btheta\Vert_{0,\O}^2\right)^{1/2}.$$

From now on, we will assume that the data are regular enough: $\bbbeta\in \L^{\infty}(\O)^3$ and $\ff\in\L^2(\O)^3$. 
We proceed to test  \eqref{oseen-cont} against adequate
functions and to impose the boundary conditions in such a manner that 
we end up with the following formulation: 
 Find $(\bu,\bomega,p)\in\H\times\Z\times\Q$ such that 
\begin{alignat}{4}
&a(\bu,\bv)      &+&\;b_{1}(\bv,\bomega)&+b_2(\bv,p)
+c(\bomega,\bv)&=\;F(\bv)&\qquad\forall\bv\in\H,\nonumber\\
&b_{1}(\bu,\btheta)&-&\;d(\bomega,\btheta)&           
&=\;G(\btheta)&\qquad\forall\btheta\in\Z,\label{probform1}\\
&b_2(\bu,q)      & &                &           & =\;0 & \qquad\forall q\in\Q,\nonumber
\end{alignat}
where the bilinear forms
$a:\H\times\H\to\R$,
$b_1:\H\times\Z\to\R$,
$d:\Z\times\Z\to\R$,
$b_2:\H\times\Q\to\R$,
$c:\Z\times\H\to\R$, and the
linear functionals $F:\H\to\R$, and $G:\Z\to\R$ are
specified as follows 
\begin{align*} 
a(\bu,\bv)&:=\sigma\int_{\O}\bu\cdot\bv\,d\bx,\quad 
b_1(\bv,\btheta):=\sqrt{\nu}\int_{\O}\curl\btheta\cdot\bv\,d\bx,\quad 
b_2(\bv,q):=-\int_{\O}q\vdiv\bv\,d\bx,\\
d(\bomega,\btheta)&:=\int_{\O}\bomega\cdot\btheta\,d\bx,\quad
c(\btheta,\bv):=\frac{1}{\sqrt{\nu}}\int_{\O}(\btheta\times\bbbeta)\cdot\bv\,d\bx,\\ 
F(\bv)&:=\int_{\O}\ff\cdot\bv\,d\bx-\langle\bv\cdot\bn, p_\Sigma\rangle_{\Sigma},\quad 
G(\btheta):=-\sqrt{\nu}\langle \bu_{\Sigma},\btheta\rangle_{\Sigma},
\end{align*}
for all $\bu,\bv\in\H$, $\bomega,\btheta\in\Z$, and $q\in\Q$.

\paragraph{Solvability analysis.}
In order to analyse the variational formulation \eqref{probform1},
let us introduce the Kernel of the bilinear form $b_2(\cdot,\cdot)$ and 
its classical characterisation 
$$ \X := \{ \bv \in \H\,:\, b_2(\bv,q)=0,\quad \forall \, q\in Q\}\,
=\,\{ \bv \in \H\,:\, \vdiv \bv \,=\,0\;\;\text{in}\;\;\O\},$$
and let us recall that $b_2$ satisfies the inf-sup condition:
\begin{equation}\label{eq:inf-sup-b2}
\sup_{\stackrel{\scriptstyle\bv\in\H}{\bv\ne0}}\frac{\vert
  b_{2}(\bv,q)\vert}{\Vert\bv\Vert_{\H}}\ge \beta_2\Vert
  q\Vert_{0,\O}\quad\forall q\in\Q,
\end{equation}
with an inf-sup constant $\beta_2>0$ only depending on $\Omega$ (see e.g. \cite{G2014}).

We will now address the well-posedness of \eqref{probform1}.
To that end, it is enough  to study its \emph{reduced} counterpart, defined 
on $\X\times \Z$: Find $(\bu,\bomega)\in \X\times \Z$ such that 
\begin{alignat}{4}
&a(\bu,\bv)      &+&\;b_{1}(\bv,\bomega)&+c(\bomega,\bv)&=\;
F(\bv)&\qquad\forall\bv\in\X,\nonumber\\
&b_{1}(\bu,\btheta)&-&\;d(\bomega,
\btheta)&&=\;G(\btheta)&\qquad\forall\btheta\in\Z\label{eq:reduced}.
\end{alignat}
The equivalence between \eqref{probform1} and 
\eqref{eq:reduced} is established in the following result, whose proof 
follows \cite{gr-1986} and it is basically a direct consequence of the inf-sup condition \eqref{eq:inf-sup-b2}.
\begin{lemma}\label{lemma:equivalence}
If $(\bu,\bomega,p)\in \H\times\Z\times\Q$ is a solution of \eqref{probform1},
then $\bu\in \X$ and $(\bu,\bomega)\in \X\times\Z$ also solves 
\eqref{eq:reduced}. Conversely,
if $(\bu,\bomega)\in \X\times\Z$ is a solution of \eqref{eq:reduced},
then there exists a unique $p\in\Q$ such that
$(\bu,\bomega,p)\in \H\times\Z\times\Q$ solves \eqref{probform1}.
\end{lemma}

The abstract setting that will permit the analysis of \eqref{probform1} is 
stated in the following general result \cite[Theorem 1.2]{G2014}.
\begin{theorem}\label{lemma:abc}
 Let 
$\mathcal A: \mathcal X\times \mathcal X \to \R$ be a bounded bilinear form and 
$\mathcal G: \mathcal X\to\R$ a bounded functional, both defined on the Hilbert space 
$(\mathcal X,\langle\cdot,\cdot\rangle_{\mathcal X})$.
If there exists $\alpha > 0$ such that
\begin{equation}\label{cont-inf-sup-A}
\disp\,\sup_{\scriptstyle y\in {\mathcal X}\setminus\{0\}}\frac{\mathcal A(x,y)}{\|y\|_{\mathcal X}}\,\ge\,\alpha\,\|x\|_{\mathcal X}
\quad\forall\,x\in \mathcal X,
\end{equation}
and
\begin{equation}\label{cont-inf-sup-B}
\disp\,\sup_{\stackrel{\scriptstyle x\in {\mathcal X}}{y\ne0}}\mathcal A(x,y)\,>\,0\quad\forall\,y\in \mathcal X,
\end{equation}
then there exists a unique solution $x \in \mathcal X$ to the problem 
\begin{equation*}
\mathcal A(x,y)= \mathcal G(y)\quad \forall\, y\in \mathcal X.
\end{equation*}
Furthermore, there exists $C>0$ (independent of $x$) such that 
$$\|x\|_{\mathcal X}\le \frac{1}{\alpha}\|\mathcal G\|_{\mathcal X'}.$$
\end{theorem}

\begin{lemma}\label{le1:continuous-problem}
Let us assume that 
\begin{equation}\label{beta-assumption}
\frac{2\Vert\bbbeta\Vert_{\infty,\O}^2}{\nu\sigma}< 1,
\end{equation}
and let us define the bilinear form $\mathcal{A}(\cdot,\cdot)$ specified as 
$$\mathcal A((\bu,\bomega),(\bv,\btheta)):=a(\bu,\bv)+\;b_{1}(\bv,\bomega)+b_{1}(\bu,\btheta)-\;d(\bomega,\btheta)+c(\bomega,\bv).$$
Then, there exist $\alpha_1,\alpha_2>0$ such that
\begin{equation}\label{acot}
\vert\mathcal A((\bu,\bomega),(\bv,\btheta))\vert\le \alpha_1\Vert(\bu,\bomega)\Vert_{\mathcal X}\Vert(\bv,\btheta)\Vert_{\mathcal X},
\end{equation}
and
\begin{equation}\label{cont-inf-sup-Aa}
\disp\,\sup_{\stackrel{\scriptstyle (\bv,\btheta)\in {\mathcal X}}{(\bv,\btheta)\ne0}}
\frac{\mathcal A((\bu,\bomega),(\bv,\btheta))}{\|(\bv,\btheta)\|_{\mathcal X}}\,\ge\,\alpha_2\,\|(\bu,\bomega)\|_{\mathcal X}
\quad\forall\,(\bu,\bomega)\in \mathcal X,
\end{equation}
where $\mathcal X:=\X\times\Z$, endowed
with the corresponding product norm, is a Hilbert space.
\end{lemma}

\begin{proof}
As a consequence of the boundedness of $a(\cdot,\cdot)$
$b_1(\cdot,\cdot)$, $c(\cdot,\cdot)$, and $d(\cdot,\cdot)$,  the 
bilinear form $\mathcal A(\cdot,\cdot)$ is bounded and so condition \eqref{acot} readily 
follows.

Concerning the satisfaction of the inf-sup condition \eqref{cont-inf-sup-Aa}, for a 
 given $(\bu,\bomega)\in\mathcal X$, we can define 
$$\tilde{\btheta}:=-\bomega\in\Z,\quad \text{and}\quad \tilde{\bv}:=(\bu+\hat{c}\sqrt{\nu}\curl\bomega)\in\X,$$
where $\hat{c}>0$ is a constant to be chosen later. 
We can then immediately assert that 
\begin{align*}
\mathcal A((\bu,\bomega),(\tilde \bv,\tilde\btheta))
& =\sigma\int_{\O}\bu\cdot\tilde{\bv}\, d\bx+\sqrt{\nu}\int_{\O}\curl\bomega\cdot\tilde{\bv}\, d\bx
+\sqrt{\nu}\int_{\O}\curl\tilde{\btheta}\cdot\bu\, d\bx \\
&\quad -\int_{\O}\bomega\cdot\tilde{\btheta}\, d\bx+\frac{1}{\sqrt{\nu}}\int_{\O}(\bomega\times\bbbeta)\cdot\tilde{\bv}\, d\bx\\
&\geq \sigma\Vert\bu\Vert_{0,\O}^2+\hat{c}\sqrt{\nu}\sigma\int_{\O}\bu\cdot\curl\bomega\, d\bx
+\sqrt{\nu}\int_{\O}\bu\cdot\curl\bomega\, d\bx+\hat{c}\nu\Vert\curl\bomega\Vert_{0,\O}^2 \\
&\quad -\sqrt{\nu}\int_{\O}\bu\cdot\curl\bomega\, d\bx+\Vert\bomega\Vert_{0,\O}^2+\frac{1}{\sqrt{\nu}}\int_{\O}(\bomega\times\bbbeta)\cdot\bu\, d\bx
+\hat{c}\int_{\O}(\bomega\times\bbbeta)\cdot\curl\bomega\, d\bx\\
& \geq \sigma\Vert\bu\Vert_{0,\O}^2-\frac{\sigma}{4}\Vert\bu\Vert_{0,\O}^2
-\hat{c}^2\sigma\nu\Vert\curl\bomega\Vert_{0,\O}^2
+\hat{c}\nu\Vert\curl\bomega\Vert_{0,\O}^2+\Vert\bomega\Vert_{0,\O}^2
-\frac{2\sigma}{3}\Vert\bu\Vert_{0,\O}^2\\
&\quad -\frac{3\Vert\bbbeta\Vert_{\infty,\O}^2}{2\nu\sigma}\Vert\bomega\Vert_{0,\O}^2
-\frac{\Vert\bbbeta\Vert_{\infty,\O}^2}{2\nu\sigma}\Vert\bomega\Vert_{0,\O}^2
-2\hat{c}^2\sigma\nu\Vert\curl\bomega\Vert_{0,\O}^2\\
& =\frac{\sigma}{12}\Vert\bu\Vert_{0,\O}^2+\hat{c}\left(1-3\hat{c}\sigma\right)\nu\Vert\curl\bomega\Vert_{0,\O}^2
+\left(1-\frac{2\Vert\bbbeta\Vert_{\infty,\O}^2}{\nu\sigma}\right)\Vert\bomega\Vert_{0,\O}^2,
\end{align*}
where we have used the bound $\Vert\bomega\times\bbbeta\Vert_{0,\O}\le2\Vert\bbbeta\Vert_{\infty,\O}\Vert\bomega\Vert_{0,\O}$.
Choosing $\hat{c}=1/(4\sigma)$ and exploiting \eqref{beta-assumption}, we arrive at 
$$\mathcal A((\bu,\bomega),(\tilde \bv,\tilde\btheta))\ge C\Vert(\bu,\bomega)\Vert_{\mathcal X}^2,$$
with $C$ independent of $\nu$. 
On the other hand, by construction we realise that   
$\Vert{\tilde{\btheta}}\Vert_{\Z}=\Vert\bomega\Vert_{\Z}$ and 
$\Vert\tilde{\bv}\Vert_{0,\O}\le C\hat{c}(\Vert\bu\Vert_{0,\O}+\Vert\bomega\Vert_{\Z})$,
and consequently 
$$\sup_{\stackrel{\scriptstyle (\bv,\btheta)\in\mathcal X}{(\bv,\btheta)\ne0}}\frac{\mathcal A((\bu,\bomega),(\bv,\btheta))}
{\Vert(\bv,\btheta)\Vert_{\mathcal X}}\ge \frac{\mathcal A((\bu,\bomega),(\tilde \bv,\tilde\btheta))}
{\Vert(\tilde\bv,\tilde\btheta)\Vert_{\mathcal X}}
\ge \alpha_2\Vert(\bu,\bomega)\Vert_{\mathcal X}\qquad\forall(\bu,\bomega)\in\mathcal X,$$
which finishes the proof.
\end{proof}

\begin{lemma}\label{le2:continuous-problem}
Suppose that the bound \eqref{beta-assumption} is satisfied. 
Then, there exists $C>0$ such that
\begin{equation*}
\disp\,\sup_{\stackrel{\scriptstyle (\bu,\bomega)\in {\mathcal X}}{(\bu,\bomega)\ne (0,0)}}\mathcal A((\bu,\bomega),(\bv,\btheta))\,>\,0\quad\forall\,(\bv,\btheta)\in \mathcal X.
\end{equation*}
\end{lemma}

\begin{proof}
For all $(\bv,\btheta)\in \mathcal X$,
we have that:
\begin{align*}
\mathcal A((\bv,-\btheta),(\bv,\btheta))
& =\sigma\Vert\bv\Vert_{0,\O}^2+\Vert\btheta\Vert_{0,\O}^2-\frac{1}{\sqrt{\nu}}\int_{\O}(\btheta\times\bbbeta)\cdot\bv\, d\bx\\
&\ge \sigma\Vert\bv\Vert_{0,\O}^2+\Vert\btheta\Vert_{0,\O}^2-\frac{\sigma}{2}\Vert\bv\Vert_{0,\O}^2
-\frac{2\Vert\bbbeta\Vert_{\infty,\O}^2}{\nu\sigma}\Vert\btheta\Vert_{0,\O}^2\\
&\ge \frac{\sigma}{2}\Vert\bv\Vert_{0,\O}^2+\left(1-\frac{2\Vert\bbbeta\Vert_{\infty,\O}^2}{\nu\sigma}\right)\Vert\btheta\Vert_{0,\O}^2.
\end{align*}
\end{proof}

As a consequence of the previous lemmas, we have the following result.

\begin{theorem}\label{th:continuous-problem}
Let us assume \eqref{beta-assumption}. 
Then, the variational problem~\eqref{eq:reduced} admits a unique solution
$(\bu,\bomega)\in\X\times\Z$. Moreover, there
exists $C>0$ such that
\begin{equation}\label{eq:stability}
\Vert\bu\Vert_{\H}+\Vert\bomega\Vert_{\Z}\le
C(\Vert\ff\Vert_{0,\O}+\Vert p_{\Sigma}\Vert_{1/2,\Sigma}+\Vert 
\bu_{\Sigma}\Vert_{-1/2,\Sigma}).
\end{equation}
\end{theorem}
\begin{proof}
It suffices to verify the hypotheses of Theorem~\ref{lemma:abc}.
First, we define the linear functional
$$\mathcal G(\bv,\btheta):=F(\bv)+G(\btheta),$$
which is bounded on $\X\times\Z$.
Thus, the proof follows from Lemmas~\ref{le1:continuous-problem} and \ref{le2:continuous-problem}.
\end{proof}

The following result establishes the corresponding stability estimate for the Bernoulli pressure.
\begin{corollary}\label{coro1}
Let $(\bu,\bomega)\in \X\times \Z$, be the unique solution of \eqref{eq:reduced},
with $\bu$ and $\bomega$ satisfying \eqref{eq:stability}. 
In addition, let $p\in Q$ be the unique pressure provided
by Lemma~\ref{lemma:equivalence}, so that $(\bu,\bomega,p)\in \H\times \Z\times \Q$
is the unique solution of \eqref{probform1}.
Then, there exists $C>0$ such that
\begin{equation*}
\Vert p\Vert_{0,\O}\le
C(\Vert\ff\Vert_{0,\O}+\Vert p_{\Sigma}\Vert_{1/2,\Sigma}+\Vert 
\bu_{\Sigma}\Vert_{-1/2,\Sigma}).
\end{equation*}
\end{corollary}
\begin{proof}
Combining the inf-sup condition \eqref{eq:inf-sup-b2}
with the first equation in \eqref{probform1} gives the bound 
$$
\|p\|_{0,\Omega}\le \frac{1}{\beta_2}\sup_{\stackrel{\scriptstyle\bv\in\H}{\bv\ne0}}\frac{\vert
b_{2}(\bv,p)\vert}{\Vert\bv\Vert_{\H}}=\frac{1}{\beta_2}\sup_{\stackrel{
\scriptstyle\bv\in\H}{\bv\ne0}}\frac{\vert
F(\bv)-a(\bu,\bv)-b_1(\bv,\bomega)-c(\bomega,\bv)\vert}{\Vert\bv\Vert_
{ \H } } ,
$$
which together with \eqref{eq:stability}, and the boundedness
of $F$, $a$, $b_1$ and $c$, complete the proof.
\end{proof}

\begin{remark}
An alternative analysis for the nonsymmetric variational problem \eqref{probform1} can be 
carried out using a fixed-point  argument that allows a symmetrisation of the mixed structure. 
 The resulting weak form could then be analysed using classical tools for saddle-point problems, 
 for instance, following the similar treatment in \cite{caucao16}. Establishing inf-sup conditions for the 
off-diagonal bilinear forms in the original nonsymmetric formulation is, however, much more involved 
(see e.g. \cite{nicolaides82}). \end{remark}

\begin{remark}
Assumption \eqref{beta-assumption} 
holds provided one chooses $\sigma$ appropriately. As this parameter represents 
the inverse of the timestep, the aforementioned relation constitutes then a CFL-type condition. We also note that this  
bound for $\bbbeta$ coincides with the hypotheses that yield solvability of least-squares 
formulations for the Oseen problem analysed in \cite{chang07,cai16}.
\end{remark}

\section{Finite element discretisation} \label{sec:FE}
In this section we introduce a Galerkin scheme
for~\eqref{probform1} and analyse its well-posedness
by establishing suitable assumptions on the finite
element subspaces involved. Error estimates are also derived. 

\paragraph{Defining the discrete problem.}
Let $\{\cT_{h}(\O)\}_{h>0}$ be a shape-regular
family of partitions of the polyhedral region
$\bar\O$, by tetrahedrons $T$ of diameter $h_T$, with mesh size
$h:=\max\{h_T:\; T\in\cT_{h}(\O)\}$.
In what follows, given an integer $k\ge0$ and a subset
$S$ of $\R^3$, $\cP_k(S)$ will denote the space of polynomial functions
defined locally in $S$ and being of total degree $\leq k$. 

Moreover, for any $T\in\cT_{h}(\O)$, we introduce the local N\'ed\'elec space
$$\N_k(T):=\cP_{k}(T)^3\oplus R_{k+1}(T),$$ 
where $R_{k+1}(T)$ is a subspace of $\cP_{k+1}(T)^3$ composed by
homogeneous polynomials of degree $k+1$, and being orthogonal to $\bx$. With 
these tools, let us 
define the following finite element subspaces:
\begin{align}
\Z_h&:=\{\btheta_h\in\Z: \btheta_h|_T\in\N_k(T)\quad\forall T\in\cT_{h}(\O)\},\label{space1}\\
\Q_h&:=\{q_h\in\Q: q_h|_T\in\cP_{k}(T)\quad\forall T\in\cT_{h}(\O)\},\label{space2}\\
\H_h&:=\{\bv_h \in \H: \bv_h|_{T} \in RT_{k}(T)\quad\forall T\in \cT_h(\O) \},\label{space3}
\end{align}
where 
$RT_k(T)=\cP_{k}(T)^3\oplus\cP_{k}(T)\bx$ is the Raviart-Thomas space defined locally in $T\in \cT_h(\Omega)$. 

The proposed Galerkin scheme approximating \eqref{probform1} reads as follows:
 Find $(\bu_h,\bomega_h,p_h)\in\H_h\times\Z_h\times\Q_h$ such that 
\begin{alignat}{4}
&a(\bu_h,\bv_h)&+&\;b_{1}(\bv_h,\bomega_h)&+\;b_2(\bv_h,p_h)+\;c(\bomega_h,
\bv_h)&=\;
F(\bv_h)&\qquad\forall\bv_h\in\H_h,\nonumber\\
&b_{1}(\bu_h,\btheta_h)&-&\;d(\bomega_h,
\btheta_h)&&=\;G(\btheta_h)&\qquad\forall\btheta_h\in\Z_h,\label{probdics}\\
&b_2(\bu_h,q_h)&&&&=\;0&\qquad\forall q_h\in\Q_h.\nonumber
\end{alignat}

\paragraph{Solvability and stability of the discrete problem.}
The analysis of the Galerkin formulation will follow 
the same arguments exploited in the continuous setting.
Let us then consider the discrete kernel of $b_2$:
\begin{equation}\label{kernel-h}
 \X_h := \{ \bv_h \in \H_h\,:\, b_2(\bv_h,q_h)=0,\quad \forall \, q\in \Q_h\}=
 \{ \bv_h \in \H_h\,:\, \vdiv \bv_h\equiv 0\quad {\rm in}\quad \Omega\},
\end{equation}
where the characterisation is indeed possible 
thanks to the inclusion  $\vdiv \H_h\subseteq \Q_h$. Moreover,
it is well-known that the following discrete inf-sup condition holds
(see \cite{G2014}):
\begin{equation}\label{eq:inf-sup-b2-h}
\sup_{\stackrel{\scriptstyle\bv_h\in\H_h}{\bv_h\ne0}}\frac{ b_{2}(\bv_h,q_h)}{\Vert\bv_h\Vert_{\H}}\ge
\tilde\beta_2\Vert q_h\Vert_{0,\O}\quad\forall q_h\in\Q_h.
\end{equation}

We again resort to a reduced version of the problem, now defined on the product space 
$\X_h\times \Z_h$. Find $(\bu_h,\bomega_h)\in \X_h\times \Z_h$ such that 
\begin{alignat}{4}
&a(\bu_h,\bv_h)      &+&\;b_{1}(\bv_h,\bomega_h)+\;c(\bomega_h,\bv_h)\,&=\;&
F(\bv_h)\,&\qquad\forall\,\bv_h\in\X_h,\nonumber\\
&b_{1}(\bu_h,\btheta_h)&-&\;d(\bomega_h,\btheta_h)\,&=\;&G(\btheta_h)\,
&\qquad\forall\,\btheta_h\in\Z_h\label{eq:reduced-h},
\end{alignat}
and its equivalence with \eqref{probdics} is once more a direct consequence of
the inf-sup condition \eqref{eq:inf-sup-b2-h}.
\begin{lemma}\label{lemma:equivalence-h}
If $(\bu_h,\bomega_h,p_h)\in \H_h\times\Z_h\times\Q_h$ is a solution of \eqref{probdics}, then $\bu_h\in \X_h$, and $(\bu_h,\bomega_h)\in \X_h\times\Z_h$
is also a solution of \eqref{eq:reduced-h}. Conversely, if $(\bu_h,\bomega_h)\in \X_h\times\Z_h$ is a solution of \eqref{eq:reduced-h},
then there exists a unique $p_h\in\Q_h$ such that $(\bu_h,\bomega_h,p_h)\in \H_h\times\Z_h\times\Q_h$ is a solution of \eqref{probdics}.
\end{lemma}

In order to establish the well-posedness of \eqref{eq:reduced-h},
we will employ the following discrete version of Theorem~\ref{lemma:abc}.

\begin{theorem}
Assume \eqref{beta-assumption}. 
Let $k\ge0$ be an integer and let $\X_h$ and $\Z_h$
be given by \eqref{kernel-h} and \eqref{space1}, respectively.
Then, there exists a unique $(\bu_h,\bomega_h)\in\X_h\times\Z_h$
solution of the discrete scheme~\eqref{eq:reduced-h}.
Moreover, there exist positive constants $\hat{C}_1,\,\hat{C}_2>0$
independent of $h$ such that
\begin{equation}\label{eq:stability-h}
\Vert\bu_h\Vert_{\H}+\Vert\bomega_h\Vert_{\Z}\le
\hat{C}_1(\Vert\ff\Vert_{0,\O}+\Vert p_{\Sigma}\Vert_{1/2,\Sigma}+\Vert 
\bu_{\Sigma}\Vert_{-1/2,\Sigma}),
\end{equation}
and
\begin{equation}\label{ceaest}
\Vert\bu-\bu_h\Vert_{\H}+\Vert\bomega-\bomega_h\Vert_{\Z}
\le\hat{C}_2\inf_{(\bv_h,\btheta_h)\in\X_h\times\Z_h}
(\Vert\bu-\bv_h\Vert_{\H}+\Vert\bomega-\btheta_h\Vert_{\Z}),
\end{equation}
where $(\bu,\bomega)\in\X\times\Z$ is the
unique solution to \eqref{eq:reduced}.
\end{theorem}

\begin{proof}
Let us define $\mathcal X_h:=\X_h\times\Z_h$ and reuse 
the forms $\mathcal A(\cdot,\cdot)$ and $\mathcal G(\cdot)$ 
as in the proof of Lemma~\ref{th:continuous-problem}. 
The next step consists in proving that 
$\mathcal A(\cdot,\cdot)$ satisfies the discrete version
of the inf-sup conditions \eqref{cont-inf-sup-A}-\eqref{cont-inf-sup-B}, 
as in Lemmas~\ref{le1:continuous-problem} and
~\ref{le2:continuous-problem}.
In order to assert \eqref{cont-inf-sup-A}, we consider
$(\bu_h,\bomega_h)\in\mathcal X_h$, and define 
$$\tilde{\btheta}_h:=-\bomega_h\in\Z_h,\qquad \text{and }
\tilde{\bv}_h:=(\bu_h+\frac{\sqrt{\nu}}{4\sigma}\curl\bomega_h)\in\X_h.$$
Then, repeating exactly the same steps used
in the proof of Lemma~\ref{le1:continuous-problem}
the discrete version of \eqref{cont-inf-sup-A} follows. Regarding 
the discrete version of \eqref{cont-inf-sup-B}, 
we once again repeat the same arguments
given in the proof of Lemma~\ref{le2:continuous-problem}.
Finally, the C\'ea estimate follows from classical arguments.
\end{proof}

We now state the stability and an adequate approximation property of the discrete pressure.
\begin{corollary}
Let $(\bu_h,\bomega_h)\in \X_h\times \Z_h$ be the unique solution
of \eqref{eq:reduced-h}, with $\bu_h$ and $\bomega_h$ satisfying \eqref{eq:stability-h}. 
In addition, let $p_h\in Q_h$ be the unique discrete Bernoulli 
pressure provided by Lemma \ref{lemma:equivalence-h},
so that $(\bu_h,\bomega_h,p_h)\in \H_h\times \Z_h\times \Q_h$
is the unique solution of \eqref{probdics}.
Then, there exist positive constants $\bar C_1,\,\bar C_2>0$,
independent of $h$ and $\nu$, such that
\begin{equation*}
\Vert p_h\Vert_{0,\O}\le
\bar C_1(\Vert\ff\Vert_{0,\O}+\Vert p_{\Sigma}\Vert_{1/2,\Sigma}+\Vert 
\bu_{\Sigma}\Vert_{-1/2,\Sigma}),
\end{equation*}
and
\begin{equation}\label{ceaest-p}
\Vert p-p_h\Vert_{0,\O}\le\bar{C}_2\inf_{(\bv_h,\btheta_h,q_h)\in\H_h\times\Z_h\times Q_h}
(\Vert\bu-\bv_h\Vert_{\H}+\Vert\bomega-\btheta_h\Vert_{\Z}+\Vert p-q_h\Vert_{0,\O}).
\end{equation}
\end{corollary}
\begin{proof}
The result follows using the same arguments
considered in the proof of Corollary~\ref{coro1},
but using the discrete inf-sup condition \eqref{eq:inf-sup-b2-h}.
We omit further details.
\end{proof}

\paragraph{A priori error estimates.}
Let us 
introduce for a given $s>1/2$, the N\'edel\'ec global interpolation operator
$\CR_{h}:\H^s(\curl;\O)\cap\Z\to\Z_h$. From \cite{Alo-Valli} we known that 
for all $\btheta\in\H^{s}(\curl;\O)$ with $s>1/2$, 
there exists $C>0$ independent of $h$, such that
\begin{equation}\label{prop2N}
\Vert \btheta-\CR_{h}\btheta\Vert_{\Z}\le
Ch^{\min\{s,k+1\}}\Vert\btheta\Vert_{\H^{s}(\curl;\O)}.
\end{equation}
On the other hand, for the Raviart-Thomas interpolation 
$\Pi_h:\HsO^{3}\cap\H\to\H_h$, with $s>0$, we recall (see e.g. \cite{G2014}) that 
there exists $C>0$, independent of $h$, such that for all $s>0$:
\begin{equation}\label{prop2rt}
\Vert \bv-\Pi_h\bv\Vert_{\H}\le
Ch^{\min\{s,k+1\}}\Vert\bv\Vert_{\H^s(\vdiv;\O)}\quad\forall \bv\in\H^s(\vdiv;\O)\cap\H.
\end{equation}
Finally we recall that the orthogonal projection from $\LO$ onto the finite element subspace $\Q_h$, 
here denoted $\CP_h$, satisfies the following error estimate for all $s>0$:
\begin{equation}\label{cotar}
\Vert q-\CP_h q\Vert_{0,\O}\le Ch^{\min\{s,k+1\}}\Vert q\Vert_{s,\O}\quad\forall q\in\HsO.
\end{equation}
These operators fulfil the following commuting diagram 
\begin{equation*}
\vdiv\CR_h\bv=\CP_h(\vdiv\bv)\quad\forall \bv\in\HsO^{3}\cap\hdivO.
\end{equation*}

The following result summarises the error analysis for our mixed
finite element scheme \eqref{probdics}.
\begin{theorem}\label{theo:conv}
Let $k\ge 0$ be an integer and let $\H_h,\Z_h$ and $\Q_h$
be given by \eqref{space1}, \eqref{space2}, and \eqref{space3}.
Let $(\bu,\bomega,p)\in\H\times\Z\times\Q$ and
$(\bu_h,\bomega_h,p_h)\in\H_h\times\Z_h\times\Q_h$ be the unique solutions
to the continuous and discrete problems \eqref{probform1} and
\eqref{probdics}, respectively.  Assume that
$\bu\in\HsO^3$, $\vdiv\bu\in\HsO$, $\bomega\in\HusO^3$ and $p\in\HsO$, for some
$s>1/2$. Then, there exists $\hat{C}>0$ independent of $h$ such
  that
$${\Vert\bu-\bu_h\Vert_{\H}+\Vert\bomega-\bomega_h\Vert_{\Z}+\Vert
  p-p_h\Vert_{0,\O} \leq\hat{C}h^{\min\{s,k+1\}}(\Vert\bu\Vert_{\H^s(\vdiv;\O)}+\|\bomega\|_{\H^s(\curl;\O)}+\Vert p\Vert_{s,\O}).}$$
\end{theorem}

\begin{proof}
The proof follows from \eqref{ceaest}, \eqref{ceaest-p}, and standard interpolation 
estimates satisfied by the operators $\CR_h$, $\Pi_h$ and $\CP_h$ (see \eqref{prop2N},
\eqref{prop2rt} and \eqref{cotar}, respectively).
\end{proof}\medskip

\section{Discontinuous Galerkin method} \label{sec:DG}
In this section, we propose and analyse a DG method for \eqref{oseen-cont}. We provide
solvability and stability of the discrete scheme by introducing suitable numerical fluxes. 
A priori error estimates are also derived. 

\paragraph{Preliminaries.} 
Apart from the definitions laid out at the beginning of Section~\ref{sec:FE}, let us denote by $\cE_h$ the
set of internal faces, by $\cF_h^{\Sigma}$ the set of
external faces on $\Sigma$ and by $\cF_h^{\Gamma}$ the
set of external faces on $\Gamma$. We set
$\cF_h=\cE_h \cup \cF_h^{\Sigma}\cup \cF_h^{\Gamma}$. 
We denote by $h_e$ the diameter of each face $e$.
Let $T^+$ and $T^-$ be two adjacent elements of $\cT_h$
and let $\bn^+$ (respectively $\bn^-$) be the outward
unit normal vector on $\partial T^+$ (respectively $\partial T^-$).
For a vector field $\bu$, we denote by $\bu^\pm$ the trace
of $\bu$ from the interior of $T^\pm$. We define jumps 
$$\jumpT{\bv} := \bv^+ \times \bn^+ +\bv^- \times \bn^- , \qquad \jumpN{\bv} := \bv^+ \cdot \bn^+ +\bv^- \cdot \bn^{-}, \quad \jump{q}:=q^+ \bn^+ + q^- \bn^- ,$$ 
and averages 
$$ \avg{\bv}:=\frac{1}{2} (\bv^+ +\bv^-),\qquad\avg{ q}:=\frac{1}{2} (q^+ +q^-),$$
and adopt the convention that for boundary faces $e\in\cF_h^\Sigma\cup\cF_h^\Gamma$,
we set $\jumpT{\bv}=\bv\times \bn$, $\jumpN{\bv}=\bv\cdot\bn$,
$\jump{q}=q\bn$, $\avg{\bv}=\bv$ and $\avg{q}=q$.   

Suitable finite dimensional spaces for vorticity
and velocity that remove the restriction of continuity are defined by:
\begin{align*}
\tilde\Z_h&:=\{\btheta_h\in\L^2(\O)^3: \btheta_h|_T\in\cP_{k}(T)^3\quad\forall T\in\cT_{h}\},\\
\tilde\H_h&:=\{\bv_h \in \L^2(\O)^3: \bv_h|_{T} \in \cP_{k+1}(T)^3\quad\forall T\in \cT_h\},
\end{align*}
and we remark that the space for pressure approximation will 
 coincide with the one used in Section~\ref{sec:FE}, that is $\tilde\Q_h := \Q_h$.   

\paragraph{Discrete formulation and solvability analysis.} 
Multiplying each equation in  \eqref{oseen-cont}
by suitable functions, the resulting DG scheme consists in finding 
$(\bu_h,\, \bomega_h,\,p_h)\in \tilde\H_h \times \tilde Z_h\times \tilde\Q_h$,
such that for any test functions
$(\bv_h,\, \btheta_h,\,q_h)\in \tilde\H_h \times \tilde Z_h\times \tilde\Q_h$
and for all elements $T$ in the partition $\cT_h$ 
 \begin{eqnarray}
  &{}&\sigma \int_{T} \bu_h\cdot\bv_h \,d\bx +\sqrt{\nu}\int_{T} \bomega_h\cdot \curl \bv_h \,d\bx +\sqrt{\nu}\int_{\partial T} \widehat\bomega_h\cdot(\bv_h\times \bn) \,ds +\frac{1}{\sqrt{\nu}} \int_{T} (\bomega_h\times \bbbeta) \cdot \bv_h \,d \bx\nonumber\\&-&\int_{T} p_h\vdiv \bv_h \, d\bx+ \int_{\partial T} \widehat p_h \bv_h\cdot\bn \,ds= \int_{T} \ff \cdot \bv_h\, d\bx,\label{eq1 DG with flux}\\
    &{}& \int_{T} \bomega_h\cdot \btheta_h \,d\bx =\sqrt{\nu}\int_{T}\bu_h\cdot \curl \btheta_h\, d\bx +\sqrt{\nu}\int_{\partial T} \widehat \bu_h^{\omega} \cdot(\btheta_h\times \bn) \,ds,\label{eq2 DG with flux}\\
        &{}& -\int_{T} \bu_h\cdot \nabla q_h d\bx + \int_{\partial T} \widehat\bu_h^p\cdot \bn \,q\, ds=0,\label{eq3 DG with flux}
 \end{eqnarray}
where $\widehat \bu_h^w,\, \widehat \bu_h^p, \,\widehat\bomega_h$ and $\widehat p_h$ are \textit{numerical fluxes}, 
  which approximate the traces of $\bu_h$, $\bomega_h$ and $p_h$ on the boundary. The fluxes $\widehat\bomega_h$ and $\widehat \bu_h^w$ are related to the $\curl$-$\curl$ operator and are defined by 
 \begin{equation}\label{flux vort}
   \widehat\bomega_h  :=\begin{cases}
   \avg{\bomega_h}+C_{11} \jumpT{\bu},\ & \mathrm{if} \,\,e\in \cE_h,\\ \bomega_h^{+}+C_{11} (\bu_h^{+}\times \bn^{+}- \bu_{\Sigma}),\ & \mathrm{if} \,\,e\in \cF^{\Sigma}_{h},\\ 
    \widehat\bomega_h\times\bn =\bf{0},\ & \mathrm{if} \,\,e\in \cF^{\Gamma}_{h},\\
  \end{cases}\qquad\quad
   \widehat\bu_h^{w}:=\begin{cases}
   \avg{\bu_h}, \ & \mathrm{if} \,\,e\in \cE_h,\\
   \bn\times \bu_{\Sigma},\ & \mathrm{if} \,\,e\in \cF^{\Sigma}_{h},\\ 
   \bu_{h}^{+},\ & \mathrm{if} \,\,e\in \cF^{\Gamma}_{h},\\
  \end{cases}
 \end{equation}
 whereas the fluxes $\widehat \bu_h^p$ and $\widehat p_h$ are associated with the $\vgrad$-$\vdiv$ operator and  
 defined by
  \begin{equation}
   \widehat\bu_h^{p}:=\begin{cases}
   \avg{\bu_h}+D_{11} \jump{p_h},\ & \mathrm{if} \,\,e\in \cE_h,\\ 
   \bu_h^{+}+D_{11} (p_h^{+}\bn^+ +p_{\Sigma} \bn^-),\ & \mathrm{if} \,\,e\in \cF^{\Sigma}_{h},\\ 
   \widehat\bu_h^{p}\cdot \bn=0,\ & \mathrm{if} \,\,e\in \cF^{\Gamma}_{h},\\
  \end{cases}\qquad 
  \label{flux press}
   \widehat p_h  :=\begin{cases}
   \avg{p_h}+ A_{11} \jumpN{\bu_h},\ & \mathrm{if} \,\,e\in \cE_h,\\ 
   p_\Sigma,\ & \mathrm{if} \,\,e\in \cF^{\Sigma}_{h},\\ 
   p^+ + A_{11}\bu^+ \cdot \bn^+,\ & \mathrm{if} \,\,e\in \cF^{\Gamma}_{h}.
  \end{cases}
 \end{equation}
 The parameters $C_{11}$, $A_{11}$ and $D_{11}$ are positive stabilisation parameters, and 
 following \cite{Cockburn_SIAM2002}  we choose 
 \begin{equation}\label{choice param stab C11}
   C_{11}(\bx):=\begin{cases}
   c_{11} \max\{h_{T^{+}}^{-1},\,h_{T^{-}}^{-1}\},\ & \mathrm{if} \,\,\bx\,\in \partial T^{+}\cup \partial T^{-},\\ 
  c_{11}h_{T}^{-1},\ & \mathrm{if} \,\,\bx\,\in \partial T\cap \Sigma,
  \end{cases}
  \end{equation}
 \begin{equation}\label{choice param stab A11}
   A_{11}(\bx):=\begin{cases}
   a_{11} \max\{h_{T^{+}}^{-1},\,h_{T^{-}}^{-1}\},\ & \mathrm{if} \,\,\bx\in \partial T^{+}\cup \partial T^{-},\\ 
  a_{11}h_{T}^{-1},\ & \mathrm{if} \,\,\bx\in \partial T\cap \Gamma,
  \end{cases}
  \end{equation}
    \begin{equation}\label{choice param stab D11}
   D_{11}(\bx):=\begin{cases}
   d_{11} \max\{h_{T^{+}},\,h_{T^{-}}\},\ & \mathrm{if} \,\,\bx\in \partial T^{+}\cup \partial T^{-},\\ 
  d_{11}h_{T},\ & \mathrm{if} \,\,\bx\in \partial T\cap \Sigma,
  \end{cases}
  \end{equation}
where $c_{11},d_{11},a_{11}>0$. Moreover,
we suppose that $C_{11}$ (respectively $D_{11}$ and $A_{11}$)
have a uniform positive bound above and below denoted by
$\overline{C_{11}}$ and $\underline{C_{11}}$ (respectively
$\overline{D_{11}}$, $\underline{D_{11}}$ and $\overline{A_{11}}$,
$\underline{A_{11}}$ ). 
 
We then proceed to integrate by parts  equations \eqref{eq1 DG with flux}
and \eqref{eq3 DG with flux}, and then summing up over all $T\in \cT_h$,
we obtain the following DG scheme:
 Find $(\bu_h,\bomega_h,p_h)\in\widetilde\H_h\times\widetilde\Z_h\times\widetilde\Q_h$
such that 
\begin{alignat}{4}
&a(\bu_h,\bv_h)&+&\;\tilde b_{1}(\bv_h,\bomega_h)&+\;\tilde b_2(\bv_h,p_h)+\;c(\bomega_h,
\bv_h)+j(\bu_h,\bv_h)&=\;
\widetilde F(\bv_h)&,\qquad\forall\bv_h\in\widetilde\H_h,\nonumber\\
&d(\bomega_h,
\btheta_h)&-&\;\tilde b_{1}(\bu_h,\btheta_h)&&=\;\tilde G(\btheta_h)&,\qquad\forall\btheta_h\in\widetilde\Z_h,\label{Mixed DG scheme}\\
&e(p_h,q_h)&-&\;\widetilde b_2(\bu_h,q_h)&&=\;\tilde L (q_h),&\qquad\forall q_h\in\widetilde\Q_h,\nonumber
\end{alignat} 
where the forms $a$, $c$ and $d$ are the
same in \eqref{probform1}, while $\tilde b_1$, $\tilde b_2$, $j$
and $e$ are defined, respectively, by:
\begin{align*} 
 \tilde b_1(\bu_h,\btheta_h)& :=\sqrt{\nu}\sum_{T\in\cT_h}\int_{T}\curl\btheta_h\cdot\bu_h \,d\bx+\sqrt{\nu}\sum_{e\in \cE_{h}\cup\cF_h^{\Gamma}}\int_{e} \avg{\bu_h}\cdot\jumpT{\btheta_h}\,ds  ,\\
\tilde b_2(\bv_h,p_h)&:=-\sum_{T\in\cT_h}\int_{T} p_h\vdiv\bv_h\, d\bx + \sum_{e\in  \cE_{h}\cup\cF_h^{\Gamma}}\int_{e}\avg{p_h}\cdot\jumpN{\bv_h}\,ds,\\
j(\bu_h,\bv_h)&:=\sqrt{\nu}\sum_{e\in\cE_{h}\cup\cF_h^{\Sigma} }\int_{e}C_{11} \jumpT{\bu_h}\cdot\jumpT{\bv_h}\,ds  + \sum_{e\in\cE_{h}\cup\cF_h^{\Gamma} }\int_{e}A_{11} \jumpN{\bu_h}\jumpN{\bv_h}\,ds,\\
e(p_h,q_h)&:=\sum_{e\in \cE_{h}\cup\cF_h^{\Sigma}}\int_{e}D_{11} \jump{p_h}\cdot\jump{q_h}\,ds. 
\end{align*}
In addition, the linear functionals $\widetilde F$, $\widetilde G$
and $\widetilde L$ associated with the source terms are defined as: 
\begin{align*} 
 \widetilde F(\bv_h):=\int_{\O}\ff\cdot\bv_h-\sum_{e\in\cF_h^{\Sigma}}\Big(\int_{e} p_{\Sigma}(\bv\cdot\bn)\,ds-\sqrt{\nu}\int_{e} C_{11} \bu_{\Sigma} \cdot(\bv_{h}\times\bn)\,ds\Big),\\
\widetilde G(\btheta_h):=-\sqrt{\nu}\sum_{e\in\cF_h^{\Sigma}}\int_{e}\bu_{\Sigma}\cdot\btheta_h\, ds\qquad \mathrm{and}\qquad\widetilde L(q_h):=\sum_{e\in\cF_h^{\Sigma}}\int_{e} D_{11}\, (p_{\Sigma}\bn)\cdot\,q_h\bn\,ds.
\end{align*}
By integration by parts and as a consequence of the identity:
\begin{equation*}
 \sum_{T\in\cT_h} \int_{\partial T} \bu\cdot (\btheta\times \bn)\, ds
 = -\sum_{e\in\cE_h} \int_{e} \jumpT{\bu}\cdot \avg{\btheta}\, ds
 +\sum_{e\in\cF_h} \int_{e} \avg{\bu}\cdot \jumpT{\btheta}\, ds,
\end{equation*}
it follows that the form $\tilde b_{1}$ can be written as:
\begin{equation}\label{form b1 after IPP}
 \tilde b_1(\bu_h,\btheta_h)=\sqrt{\nu}\sum_{T\in\cT_h}\int_{T}\curl\bu_h\cdot\btheta_h \,d\bx+\sqrt{\nu}\sum_{e\in \cE_{h}\cup\cF_h^{\Sigma}}\int_{e} \jumpT{\bu_h}\cdot\avg{\btheta_h}\,ds.
 \end{equation}
Similarly, using the following identity: 
\begin{equation*}
 \sum_{T\in\cT_h} \int_{\partial T} p(\bv\cdot\bn)\, ds =
 \sum_{e\in\cE_h} \int_{e} \avg{\bv}\cdot \jump{p}\, ds
 +\sum_{e\in\cF_h} \int_{e} \jumpN{\bv}\avg{p}\, ds,
\end{equation*}
the form  $\tilde b_{2}$ can be recast, after integration by parts, as follows
\begin{equation} \label{form b2 after IPP}
 \tilde b_2(\bv_h,p_h)=\!\!\!\sum_{T\in\cT_h}\int_{T} \bv_h \cdot \nabla p_h\, d\bx
 -\!\!\!\sum_{e\in \cE_{h}\cup\cF_h^{\Sigma}}\int_{e} \avg{\bv_h}\cdot\jump{p_h}\,ds.
\end{equation}
Note that, differently from the conforming method introduced in Section \ref{sec:FE}, the discrete 
velocity generated by scheme \eqref{Mixed DG scheme} is not necessarily divergence-free. 

To simplify the exposition of the analysis of the method, we will write the mixed scheme 
\eqref{Mixed DG scheme} in the following equivalent form:
Find $(\bu_h,\bomega_h,p_h)\in\widetilde\H_h\times\widetilde\Z_h\times\widetilde\Q_h$ such that 
\begin{equation}\label{compact DG scheme}
 \mathcal{A}(\bu_h,\bomega_h,p_h;\bv_h,\btheta_h,q_h)=\mathcal{F}(\bv_h,\btheta_h, q_h),\quad \forall (\bv_h,\btheta_h,q_h)\in\widetilde\H_h\times\widetilde\Z_h\times\widetilde\Q_h,
\end{equation}
where 
\begin{eqnarray*}
 \mathcal{A}(\bu_h,\bomega_h,p_h;\bv_h,\btheta_h,q_h)&:=&a(\bu_h,\bv_h)
 +\tilde b_{1}(\bv_h,\bomega_h)-\tilde b_{1}(\bu_h,\btheta_h)
 +\tilde b_2(\bv_h,p_h)-\widetilde b_2(\bu_h,q_h)\\&&+\,c(\bomega_h,
\bv_h)+j(\bu_h,\bv_h)+d(\bomega_h,\btheta_h)+e(p_h,q_h),
\end{eqnarray*}
and 
\begin{equation*}
 \mathcal{F}(\bv_h,\btheta_h, q_h):=\widetilde F(\bv_h)+\widetilde G(\btheta_h)+\widetilde L(q_h).
\end{equation*}

Let us now show the existence and uniqueness of
solution to formulation \eqref{Mixed DG scheme}. 

\begin{proposition}
 The DG method \eqref{Mixed DG scheme} with the numerical
 fluxes given by \eqref{flux vort}-\eqref{flux press}
 defines a unique approximate solution
 $(\bu_h,\bomega_h,p_h)\in\widetilde\H_h\times\widetilde\Z_h\times\widetilde\Q_h$
 provided that 
 \begin{equation}\label{beta-assumption-DG}
  \frac{2\Vert\bbbeta\Vert_{\infty}^2}{\nu\sigma}< 1.
 \end{equation}
\end{proposition}

\begin{proof}
Since the problem is linear and finite dimensional,
it suffices to show that if $\ff=\textbf{0}$, $p_{\Sigma}=0$
and $\bu_{\Sigma}=\textbf{0}$, then $(\bu_h,\bomega_h,p_h)=(\textbf{0},\textbf{0},0)$.
To this end, take $\bv_h=\bu_h$, $\btheta_h=\bomega_h$
and $q_h=p_h$ in \eqref{Mixed DG scheme}, summing up the three equations, we obtain:
  \begin{equation*}
  a(\bu_h,\bu_h)+c(\bomega_h,\bu_h)+d(\bomega_h,\bomega_h)+j(\bu_h,\bu_h)+e(p_h,p_h)=0.
 \end{equation*}
It follows that
\begin{eqnarray*}
 \sigma\Vert\bu_h\Vert^{2}_{0,\O}+\Vert\bomega_h\Vert_{0,\O}^{2}+\vert\bu\vert_{j}^{2}+\vert p\vert_{e}^{2} =-\frac{1}{\sqrt{\nu}}\int_{\O}(\bomega_h\times \bbbeta)\cdot\bu_h\,d\bx,
  \end{eqnarray*}
  where we define 
  \begin{equation*}
 \vert\bu\vert_{j}^{2}:=\sqrt{\nu}\sum_{e\in\cE_{h}\cup\cF_h^{\Sigma}}\int_{e}C_{11} \jumpT{\bu}^{2}\,ds+\sum_{e\in\cE_{h}\cup\cF_h^{\Gamma}}\int_{e}A_{11} \jumpN{\bu}^{2}\,ds , \qquad 
\vert p\vert_{e}^{2}:=\sum_{e\in \cE_{h}\cup\cF_h^{\Sigma}}\int_{e}D_{11} \jump{p}^{2}\,ds.
\end{equation*}
Using Young's inequality, we  can assert that 
  \begin{align*}
 \sigma\Vert\bu_h\Vert^{2}_{0,\O}+\Vert\bomega_h\Vert_{0,\O}^{2}+\vert\bu\vert_{j}^{2}+\vert p\vert_{e}^{2}
&\leq \frac{2}{\sqrt{\nu}}\Vert\bbbeta\Vert_{\infty,\O}\Vert\bomega_h\Vert_{0,\O}\Vert\bu_h\Vert_{0,\O}\\&\leq \frac{2}{\nu\sigma}\Vert\bbbeta\Vert_{\infty,\O}^{2}\Vert\bomega_h\Vert_{0,\O}^{2} +\frac{\sigma}{2}\Vert\bu_h\Vert_{0,\O}^{2}.
 \end{align*}
Therefore, in particular we have that:
\begin{equation*}
 \frac{\sigma}{2}\Vert\bu_h\Vert_{0,\O}^{2}+(1-\frac{2\Vert\bbbeta\Vert_{\infty}^2}{\nu\sigma})\Vert\bomega_h\Vert_{0,\O}^{2}+\vert p\vert_{e}^{2}\leq 0,
\end{equation*}
which, owing to the assumption \eqref{beta-assumption-DG}, implies that $\bu_h=\textbf{0}$, $\bomega_h=\textbf{0}$ and $\jump{p_h}=0$ on $\cE_h$, $p_h=0$ on $\cF_h^{\Sigma}$. The first equation in \eqref{Mixed DG scheme} then becomes 
 $$\sum_{T\in\cT_h}\int_{T}\bv_h\cdot\nabla p_h\,d\bx=0,\quad \forall \bv_h\in \widetilde \H_h,$$
 and then $\nabla p_h=\textbf{0}$.
 Employing this result, together with $\jump{p_h}=0$
 on $\cE_h$, $p_h=0$ on $\cF_h^{\Sigma}$ or the fact
 that $p_h$ has zero mean value if $\Sigma$ have zero measure,
 we conclude that $p_h=0$. 
 \end{proof}

\paragraph{A priori error bounds.} 
 Let us now present and discuss a priori error bounds for the proposed DG method.
 The proof involves two steps. The first one consists in establishing an error
 estimate in the natural semi-norm. In the second step, we prove the error
 estimate for the pressure in the $\L^2$-norm. For this, we introduce the
 following semi-norm $\vert\cdot\vert_{\mathcal{A}}$:
  \begin{equation}\label{energy DG norm}
  \vert(\bu,\bomega,p)\vert_{\mathcal{A}}^{2}
  :=\sigma\Vert\bu\Vert_{0,\O}^{2}+\Vert\bomega\Vert_{0,\O}^{2}+\vert\bu\vert_{j}^{2}+\vert p\vert_{e}^{2}.
 \end{equation}
We shall suppose that the exact solution $(\bu,p)$ satisfies the following regularity 
\begin{equation}\label{regularity for DG error}
 \bu\in\HusO^3,\quad \mathrm{and}\quad p\in\H^{s}(\O), \,\,s\geq 1.
\end{equation}
And we will also employ the following norm 
\begin{equation*}
 \Vert(\bu,p)\Vert_{s}:=\sqrt{\nu}\Vert\bu\Vert_{s+1,\O}+\frac{1}{\sqrt{\nu}}\Vert p\Vert_{s,\O}.
\end{equation*}
We define $\be_{\bu}=\bu-\bu_h$, $\be_{\bomega}=\bomega-\bomega_h$
and $e_p=p-p_h$. Let us denote by $\bpi_{\widetilde\H}$ 
(respectively $\bpi_{\widetilde\Z}$ and $\Pi_{\widetilde\Q}$) 
the $\L^{2}$-projection onto $\widetilde\H$ (respectively $\widetilde\Z$ and $\widetilde\Q$), 
and let us split the errors in the following manner 
  $$\be_{\bu} =\bxi_{\bu}+\betta_{\bu},\quad\be_{\bomega}=\bxi_{\bomega}+\betta_{\bomega}\quad \mathrm{and}\quad e_p=\xi_{p}+\eta_{p},$$
  where the numerical and approximation errors are defined by:
  \begin{align*}
  \bxi_{\bu}&=\bu-\bpi_{\widetilde\H}\bu,\quad  \bxi_{\bomega}=\bomega-\bpi_{\widetilde\Z}\,\bomega,\quad \xi_{p}=p-\Pi_{\widetilde\Q} p,\\
  \betta_{\bu}&=\bpi_{\widetilde\H}\bu-\bu_h,\quad  \betta_{\bomega}=\bpi_{\widetilde\Z}\bomega-\bomega_h,\quad \eta_{p}=\Pi_{\widetilde\Q} p-p_h.
 \end{align*}
 We recall the following standard approximation properties (see for instance \cite{Ciarlet}).
 \begin{lemma}\label{L2 projection}
 Let $\bv\in \H^{1+r}(\O)$, $r\geq0$. Let $\Pi$ the projection operator
 such that $\Pi \bv=\bv$ for all $\bv\in\mathcal{P}^{k}(T)$, $k\geq0$. Then we have
  \begin{align}
  \Vert \bv-\Pi \bv\Vert_{0,T}+h_{T}\vert \bv-\Pi \bv\vert_{1,T}&\leq
  C h_{T}^{\min\{r,k\}+1}\Vert\bv\Vert_{r+1,T},\label{eq1: L2 projection}\\
  \Vert \bv-\Pi \bv\Vert_{0,\partial T}&\leq C h_{T}^{\min\{r,k\}+1/2}\Vert\bv\Vert_{r+1,T}\label{eq2: L2 projection}.
  \end{align}
    \end{lemma}
And as a consequence, we have the following result.
\begin{lemma}\label{numerical-error bounds DG}
 Suppose that the analytical solution $(\bu,p)$ of \eqref{NS-cont-vort}
 satisfies \eqref{regularity for DG error} and we set $\bomega=\sqrt{\nu}\curl \bu$. Then, we have:
  \begin{align*}
 \Vert\bxi_{\bu}\Vert_{0,\O}&\leq  C_{a} h^{\min\{s,\,k+1\}+1}\Vert(\bu,0)\Vert_{s},\\
    \Vert\bxi_{\bomega}\Vert_{0,\O}&\leq  C_{d} h^{\min\{s,\,k\}+1}\Vert(\bu,0)\Vert_{s},\\
    \vert\bxi_{\bu}\vert_{j}&\leq C_{j} h^{\min\{s,\,k+1\}}\Vert(\bu,0)\Vert_{s},\\
 \vert\xi_{p}\vert_{e}&\leq  C_{e} h^{\min\{s,\,k\}+1}\Vert(\mathbf{0},p)\Vert_{s},
   \end{align*}
where $C_{a}$, $C_{d}$, $C_{j}$ and $C_{e}$ are positive constants independent of the meshsize.
\end{lemma}
\begin{proof}
The two first estimates are a simple consequence
of \eqref{eq1: L2 projection}. Next we can state that 
\begin{equation*}
 \vert\bxi_{\bu}\vert_{j} \leq 2\big(\sqrt{\nu}\sum_{T\in\cT_h}\underline{C_{11}} \Vert\bxi_{\bu}\Vert_{0,\partial T}^{2}\big)^{1/2}+2\big(\sum_{T\in\cT_h}\underline{A_{11}} \Vert\bxi_{\bu}\Vert_{0,\partial T}^{2}\big)^{1/2}
 \end{equation*}
 and 
 \begin{equation*}
 \vert\xi_{p}\vert_{e} \leq 2\big(\sum_{T\in\cT_h}\underline{D_{11}} \Vert\xi_{p}\Vert_{0,\partial T}^{2}\big)^{1/2}.
\end{equation*}
Recalling that $C_{11}$, $A_{11}$ and $D_{11}$ are as in \eqref{choice param stab C11}-\eqref{choice param stab D11} and using \eqref{eq2: L2 projection}, we end up with the bound
 \begin{equation*}
 \vert\bxi_{\bu}\vert_{j} \leq C_j\big(\sum_{T\in\cT_h}h_T^{2\min \{s,k+1\}}
 \Vert{\bu}\Vert_{s+1,T}^{2}\big)^{1/2}
 \leq C_j h^{\min\{s,k+1\}}\Vert{\bu}\Vert_{s+1,\O}, 
\end{equation*}
and similarly we obtain
\begin{equation*}
  \vert \xi_{p}\vert_{e} \leq C_e\big(\sum_{T\in\cT_h}h_T^{2\min \{s,k\}+2}\Vert{p}\Vert_{s,T}^{2}\big)^{1/2}
  \leq C_e h^{\min\{s,k\}+1}\Vert{p}\Vert_{s,\O},
\end{equation*}
where $C_j$ (respectively $C_e$) depends on $c_{11} $, $a_{11}$ and $\nu$ (respectively $d_{11}$).
\end{proof}

\noindent We next concentrate on obtaining bounds for the forms $\tilde b_{1}$, $\tilde b_{2}$, $c$ and $j$. 

\begin{lemma}\label{estimates bilinear forms}
Let us assume that the solution of \eqref{NS-cont-vort}
  satisfies \eqref{regularity for DG error}. Then one has 
    \begin{align*}
    \vert\tilde b_{1}(\bxi_{\bu},\btheta_h)\vert &\leq  C_{b_{1}} h^{\min\{s,k+1\}}\Vert(\bu,0)\Vert_{s}\Vert\btheta_h\Vert_{0,\O} ,\quad \quad\forall\btheta_h\in \widetilde Z_h,\\
      \vert\tilde b_{1}(\bv_h,\bxi_{\bomega})\vert &\leq  C_{b_{1}} h^{\min\{s,k\}+1}\Vert(\bu,0)\Vert_{s}\vert\bv_h\vert_{j},\qquad\,\quad \forall\bv_h\in \widetilde H_h,\\
       \vert\tilde b_{2}(\bv_h,\xi_{p})\vert &\leq  C_{b_{2}} h^{\min\{s,k\}+1}\Vert(\mathbf{0},p)\Vert_{s}\vert\bv_h\vert_{j}, \qquad\,\,\quad\forall\bv_h\in \widetilde H_h,\\
        \vert\tilde b_{2}(\bxi_{\bu}, q_{h})\vert &\leq  C_{b_{2}} h^{\min\{s,k+1\}}\Vert(\bu,0)\Vert_{s}\vert q_h\vert_{e},\qquad\,\,\,\quad \forall q_h\in \widetilde Q_h,\\
        \vert j(\bxi_{\bu},\bv_h)\vert &\leq C_{j} h^{\min\{s,\,k+1\}}\Vert(\bu,0)\Vert_{s} \vert \bv_h\vert_{j} ,\quad\,\,\,\,\qquad\forall\bv_h\in \widetilde H_h,
        \\ \vert e(\xi_{p},q_h)\vert& \leq  C_{e} h^{\min\{s,\,k\}+1}\Vert(\mathbf{0},p)\Vert_{s}\vert q_h\vert_{e},\qquad\,\,\,\,\,\,\quad\forall q_h\in \widetilde Q_h,\\
     \vert c(\bxi_{\bomega},\bv_h)\vert &\leq C_{\bomega} h^{\min\{s,\,k\}+1}\Vert(\bu,0)\Vert_{s} \Vert \bv_h\Vert_{0,\O},\quad\,\,\quad \forall \bv_h\in \widetilde \H_h,
  \end{align*}
where $C_{b_{1}}$, $C_{b_{2}}$, $C_j$, $C_e$ and $C_{\bomega}$
are positive constants independent of the meshsize. 
\end{lemma}
\begin{proof}
 The bounds associated with the form $\tilde b_{2}$ can be proved exactly with the same arguments as
 \cite[Section 3.3]{Cockburn_SIAM2002}. Let us now deal with the term
 $\tilde b_{1}$. We observe that due to the properties of the $\L^2$-projection, we can write 
 \begin{equation*}
  \int_{T} \bxi_{\bu}\cdot \curl \btheta_{h}\, d\bx =0.
 \end{equation*}
Using Cauchy-Schwarz's inequality, we then readily obtain
\begin{equation*}
  \vert\tilde b_{1}(\bxi_{\bu},\btheta_h)\vert \leq  C\Big (\sum_{T\in\cT_h} \nu h_T^{-1}\Vert\bxi_{\bu}\Vert^{2}_{0,\partial T} \Big)^{1/2} \Big (\sum_{T\in\cT_h} h_T\Vert\btheta_{h}\Vert^{2}_{0,\partial T} \Big)^{1/2},
\end{equation*}
and  the desired estimate follows from the inverse inequality and
Lemma \ref{numerical-error bounds DG}.

\noindent Similarly, using \eqref{form b1 after IPP}
and again the properties of the $\L^{2}$-projections we have
\begin{equation*}
\int_{T} \curl \bv_h \cdot \bxi_{\bomega}\, d\bx=0,
\end{equation*}
thus 
 \begin{align*}
  \vert\tilde b_{1}(\bv_h,\bxi_{\bomega})\vert
  &\leq  C\Big (\sum_{T\in\cT_h} \frac{\nu}{\underline{C_{11}}}\Vert\bxi_{\bomega}\Vert^{2}_{0,\partial T} \Big)^{1/2} \Big(\sqrt{\nu}\sum_{e\in\cE_{h}}\int_{e} C_{11} \jumpT{\bv_h}^{2}\,ds +\sqrt{\nu} \sum_{e\in\cF_h^{\Sigma}}\int_{e}C_{11}\bv_h^{2} \,ds \Big )^{1/2} \\
  & \leq \Big (\sum_{T\in\cT_h} \frac{\nu}{\underline{C_{11}}}\Vert\bxi_{\bomega}\Vert^{2}_{0,\partial T} \Big)^{1/2} \vert\bv_{h}\vert_{j}
\end{align*}
and then, we simply have to use \eqref{choice param stab C11} and Lemma \ref{L2 projection}. 

The estimates for the forms $j$ and $e$ are obtained in a similar way.
Indeed, using once again Cauchy-Schwarz's inequality and Lemma~\ref{numerical-error bounds DG}, it follows that  
\begin{align*}
 \vert j(\bxi_{u},\bv_h)\vert& = \Big{\vert}\sqrt{\nu}\sum_{e\in\cE_{h}\cup\cF_h^{\Sigma}}\int_{e}C_{11} \jumpT{\bxi_{\bu}}\cdot\jumpT{\bv_h}\,ds+ \sum_{e\in\cE_{h}\cup\cF_h^{\Gamma}}\int_{e}A_{11} \jumpN{\bxi_{\bu}}\jumpN{\bv_h}\,ds \Big{\vert}  \\
 &\leq  \vert \bv_h\vert_j \vert \bxi_{\bu}\vert_j\leq C_{j} h^{\min\{s,\,k+1\}}\Vert(\bu,0)\Vert_{s} \vert \bv_h\vert_{j},
\end{align*}
and proceeding analogously as before, we get 
$$
 \vert e(\xi_{p},q_h)\vert \leq  \vert\xi_{p} \vert_{e}\vert q_h \vert_{e}\leq  C_{e} h^{\min\{s,\,k\}+1}\Vert(\mathbf{0},p)\Vert_{s}\vert q_h\vert_{e}.
$$
Finally, concerning the term $c(\bxi_{\bomega},\bv_h)$ we can assert that 
\begin{equation*}
 c(\bxi_{\bomega},\bv_h)\leq \frac{2}{\sqrt{\nu}} \Vert\bbbeta\Vert_{\infty}\Vert\bxi_{\bomega}\Vert_{0,\O}\Vert \bv_h\Vert_{0,\O},
\end{equation*}
and exploiting the previous bounds, we obtain the corresponding estimate with $C_{\bomega}=\frac{2C_{d}}{\sqrt{\nu}}\Vert\bbbeta\Vert_{\infty}$.
\end{proof}

\begin{theorem}\label{th: convergence DG}
Let the exact solution $(\bu,p)$ satisfy the regularity assumption
in \eqref{regularity for DG error}. Then, the mixed DG approximation
$(\bu_h,\bomega_h,p_h)$ defined by \eqref{Mixed DG scheme}, satisfies the following a priori error bounds
\begin{align}\label{energy error bound DG}
 \vert(\be_{\bu} ,\be_{\bomega},e_p)\vert_{\mathcal{A}} & \leq C_{\mathcal{A}} h^{\min\{s,\,k+1\}}\Vert(\bu,p)\Vert_{s},\\
\label{pressue error bound DG}
\Vert e_{p}\Vert_{0,\O} & \leq  C h^{\min\{s,\,k+1\}}\Vert(\bu,p)\Vert_{s},
\end{align}
where $C_{\mathcal{A}}$ and $C$ are positive constants independent of the meshsize. 
 \end{theorem}

\begin{proof}
We begin with the estimate \eqref{energy error bound DG}. 
 A direct application of the definition of the $\mathcal{A}-$seminorm in combination with Lemma~\ref{numerical-error bounds DG} gives
\begin{equation}\label{energy numerical-error bound DG}
 \vert(\bxi_{\bu} ,\bxi_{\bomega},\xi_p)\vert_{\mathcal{A}} \leq C h^{\min\{s,\,k+1\}}\Vert(\bu,p)\Vert_{s}.
\end{equation}
Concentrating on the projection of the errors, we can exploit the Galerkin orthogonality to obtain 
\begin{align}\label{projection error equation}
 \vert(\betta_{\bu} ,\betta_{\bomega},\eta_p)\vert_{\mathcal{A}}^{2}&=\mathcal{A}(\betta_{\bu},\betta_{\bomega},\eta_p;\betta_{\bu},\betta_{\bomega},\eta_p)-c(\betta_{\bomega},\betta_{\bu})\\&=\mathcal{A}(\betta_{\bu},\betta_{\bomega},\eta_p;\bxi_{\bu},\bxi_{\bomega},\xi_p)-c(\betta_{\bomega},\betta_{\bu})\nonumber.
\end{align}
Due to the orthogonality of the $\L^{2}$-projections, we have that $a(\bxi_{\bu},\betta_{\bu})=0$ and $d(\bxi_{\bomega},\betta_{\bomega})=0$. Then, from the the definition of the form $\mathcal{A}$ it follows that 
\begin{equation*}
 \mathcal{A}(\betta_{\bu},\betta_{\bomega},\eta_p;\bxi_{\bu},\bxi_{\bomega},\xi_p)=\tilde b_{1}(\bxi_{\bu},\betta_{\bomega})-\tilde b_{1}(\betta_{\bu},\bxi_{\bomega}) +\tilde b_2(\bxi_{\bu},\eta_{p})-\widetilde b_2(\betta_{\bu},\xi_{p})+c(\bxi_{\bomega},\betta_{\bu})+j(\bxi_{\bu},\betta_{\bu})+e(\xi_{p},\eta_{p}).
\end{equation*}
Note that all these terms can be controlled using Lemma \ref{estimates bilinear forms}. Indeed we have 
\begin{align}\label{estim A eta xi}
 \mathcal{A}(\betta_{\bu},\betta_{\bomega},\eta_p;\bxi_{\bu},\bxi_{\bomega},\xi_p)&\leq
 \Big(C_{b_{1}}\Vert\betta_{\bomega}\Vert_{0,\O}+(C_{b_{1}}+C_{b_{2}}+C_{j})\vert\betta_{\bu}\vert_{j}
 +(C_{b_{2}}+C_{e})\vert\eta_{p}\vert_{e}\nonumber\\
 &\qquad +C_{\bomega}\Vert\betta_{\bu}\Vert_{0,\O}\Big) h^{\min\{s,\,k+1\}}\Vert(\bu,p)\Vert_{s},
 \nonumber\\
 &\leq  C_{1} \vert(\betta_{\bu} ,\betta_{\bomega},\eta_p)\vert_{\mathcal{A}} \,h^{\min\{s,\,k+1\}}\Vert(\bu,p)\Vert_{s},
\end{align}
where $C_{1}=C_{b_{1}}+(C_{b_{1}}+C_{b_{2}}+C_{j})+(C_{b_{2}}+C_{e})+\frac{C_{\bomega}}{\sigma}$.
Next we only need to estimate  $c(\betta_{\bomega},\betta_{\bu})$
in \eqref{projection error equation}. To do so we use Young's inequality
\begin{equation}\label{estim c eta w u}
 c(\betta_{\bomega},\betta_{\bu})\leq \frac{2}{\sqrt{\nu}}\Vert\bbbeta\Vert_{\infty,\O}
 \Vert\betta_{\bomega}\Vert_{0,\O}\Vert\betta_{\bu}\Vert_{0,\O} \leq
 \frac{2}{\nu\sigma}\Vert\bbbeta\Vert_{\infty,\O}^{2}\Vert\betta_{\bomega}\Vert_{0,\O}^{2}
 +\frac{\sigma}{2}\Vert\betta_{\bu}\Vert_{0,\O}^{2}.
 \end{equation}
Substituting \eqref{estim A eta xi} and \eqref{estim c eta w u} back into \eqref{projection error equation}, we obtain the bound 
\begin{equation*}
 \frac{\sigma}{2}\Vert\betta_{\bu}\Vert_{0,\O}^{2}+(1-\frac{2\Vert\bbbeta\Vert_{\infty}^2}{\nu\sigma})\Vert\betta_{\bomega}\Vert_{0,\O}^{2}
 +\vert\betta_{\bu}\vert_{j}^{2}+\vert \eta_{p}\vert_{e}^{2}\leq
 C_{1} \vert(\betta_{\bu} ,\betta_{\bomega},\eta_p)\vert_{\mathcal{A}} \,h^{\min\{s,\,k+1\}}\Vert(\bu,p)\Vert_{s},
\end{equation*}
and thanks to assumption \eqref{beta-assumption-DG}, we can arrive at 
\begin{equation}\label{energy projection-error bound}
 \vert(\betta_{\bu} ,\betta_{\bomega},\eta_p)\vert_{\mathcal{A}}\leq
 C'\,h^{\min\{s,\,k+1\}}\Vert(\bu,p)\Vert_{s},
\end{equation}
where $C'=C_{1}(\min\{\frac{1}{2},1-\frac{2\Vert\bbbeta\Vert_{\infty}^2}{\nu\sigma}\})^{-1}$.
The error estimate in \eqref{energy error bound DG} is then obtained
by combining estimates \eqref{energy numerical-error bound DG}, \eqref{energy projection-error bound} and using triangle inequality.\medskip

We now turn to the estimate of the $\L^{2}$-norm of error in the pressure
\eqref{pressue error bound DG}. Since $e_{p}\in\L^{2}_{0}(\Omega)$,
we can find $\bz\in \H^{1}_{0}(\O)^3$ such that (see for instance \cite[Chapter I, Corollary 2.4]{gr-1986})
\begin{equation}\label{inf sup with ep}
 -\int_{\O} e_p \,\mathrm{div} \bz\,d\bx \geq \kappa \Vert e_p\Vert_{0,\O}^{2}, \qquad \Vert\bz\Vert_{1,\O}\leq \Vert e_p\Vert_{0,\O}.
\end{equation}
Therefore, we infer from \eqref{compact DG scheme} that 
\begin{align*}
 \kappa \Vert e_p\Vert_{0,\O}^{2} &\leq \tilde b_{2}(\bz, e_p)\nonumber\\
 &=\big(\tilde b_{2}(\bz, e_p)+a(e_{\bu},\bz)+\tilde b_{1}(\bz, e_{\bomega})+c(e_{\bomega},\bz)\big)-a(e_{\bu},\bz)-\tilde b_{1}(\bz, e_{\bomega})-c(e_{\bomega},\bz)\nonumber\\
 &=\mathcal{A}(e_{\bu}, e_{\bomega},e_{p}; \bz,\textbf{0},0)-a(e_{\bu},\bz)-\tilde b_{1}(\bz, e_{\bomega})-c(e_{\bomega},\bz),
\end{align*}
where we have used  that $j(e_{\bu},\bz)=0$ for $\bz\in \H^{1}_{0}(\O)^3$.
Using the Galerkin orthogonality, we obtain  
$$\mathcal{A}(e_{\bu}, e_{\bomega},e_{p}; \bz,\textbf{0},0)= \mathcal{A}(e_{\bu}, e_{\bomega},e_{p}; \bxi_{\bz},\textbf{0},0)=\mathcal{A}(\bxi_{\bu}, \bxi_{\bomega},\xi_{p}; \bxi_{\bz},\textbf{0},0)+\mathcal{A}(\betta_{\bu}, \betta_{\bomega},\eta_{p}; \bxi_{\bz},\textbf{0},0).$$
Therefore 
$$
 \kappa \Vert e_p\Vert_{0,\O}^{2} \leq \vert\mathcal{A}(\bxi_{\bu}, \bxi_{\bomega},\xi_{p}; \bxi_{\bz},\textbf{0},0)\vert+\vert\mathcal{A}(\betta_{\bu}, \betta_{\bomega},\eta_{p}; \bxi_{\bz},\textbf{0},0)\vert+\vert a(e_{\bu},\bz)\vert+\vert\tilde b_{1}(\bz, e_{\bomega})\vert+\vert c(e_{\bomega},\bz)\vert.
$$
Next, by the definition of the form $\mathcal{A}$, we have the relation 
$$
 \vert\mathcal{A}(\betta_{\bu}, \betta_{\bomega},\eta_{p}; \bxi_{\bz},\textbf{0},0)\vert \leq \vert \tilde{b}_{1}(\bxi_{\bz},\betta_{\bomega})\vert +\vert \tilde{b}_{2}(\bxi_{\bz},\eta_{p})\vert + \vert c(\betta_{\bomega}, \bxi_{\bz})\vert +\vert j(\betta_{\bu},\bxi_{\bz})\vert 
 := T_1 + T_2 + T_3 + T_4, 
$$
 and then we can write 
\begin{equation}\label{terms to estim for error on p}
 \kappa \Vert e_p\Vert_{0,\O}^{2} \leq \vert\mathcal{A}(\bxi_{\bu}, \bxi_{\bomega},\xi_{p}; \bxi_{\bz},\textbf{0},0)\vert+T_1 + T_2 + T_3 + T_4 + T_5 + T_6 + T_7.
 \end{equation}
 
  The first term in the right-hand side of the above inequality
  can be easily estimated by using Lemma \ref{estimates bilinear forms}.
  Indeed, it follows by choosing $\bv_h=\bxi_{\bz}$ that 
 \begin{equation*}
  \vert\mathcal{A}(\bxi_{\bu}, \bxi_{\bomega},\xi_{p}; \bxi_{\bz},\textbf{0},0)\vert\leq C h^{\min\{s,k+1\}} \Vert(\bu,p)\Vert_{s}.
 \end{equation*}
 
Let us now estimate each of the terms $T_i$, $i=1,\ldots,7$ in \eqref{terms to estim for error on p}.
Using the properties of the $\L^2-$projection, Cauchy-Schwarz's inequality and the inverse inequality, we obtain 
the bounds 
\begin{align*}
 T_1&\leq C\sqrt{\nu} \Big(\sum_{T\in \cT_h}h_{T}^{-1}\Vert\bxi_{\bz}\Vert^{2}_{0,\partial T}\Big)^{1/2} \Big(\sum_{T\in \cT_h}h_{T}\Vert\betta_{\bomega}\Vert^{2}_{0,\partial T}\Big)^{1/2}\leq C \sqrt{\nu}\Vert\bz\Vert_{1,\O}\Vert\betta_{\bomega}\Vert_{0,\O}.\\
 &\leq  C \sqrt{\nu}\Vert\bz\Vert_{1,\O}\vert(\betta_{\bu} ,\betta_{\bomega},\eta_p)\vert_{\mathcal{A}}
 \end{align*}
Then, using \eqref{energy projection-error bound} and \eqref{inf sup with ep}, we can deduce that 
\begin{equation*}
T_1\leq C \sqrt{\nu} h^{\min\{s,\,k+1\}}\Vert(\bu,p)\Vert_{s}\Vert e_p\Vert_{0,\O}.
\end{equation*}

Furthermore, since $\int_{T} \bxi_{\bz}\cdot \nabla \eta_p\, d\bx=0$, we get from \eqref{form b2 after IPP} the following 
estimates 
\begin{align*}
 T_2 & = \Big{\vert}\sum_{e\in\cE_{h}}\int_{e} \avg{\bxi_{\bz}}\cdot\jump{\eta_{p}}\,ds+\sum_{e\in\cF_h^{\Sigma} }\int_{e} (\bxi_{\bz}\cdot\bn)\eta_{p}\,ds\Big{\vert}\\
 &\leq  \Big(\sum_{e\in\cE_{h}}\int_{e}\frac{1}{D_{11}}\avg{\bxi_{\bz}}^{2}\,ds + \sum_{e\in\cF_h^{\Sigma}}\int_{e}\frac{1}{D_{11}}(\bxi_{\bz}\cdot\bn)^{2} \,ds \Big )^{1/2}\vert \eta_{p}\vert_e   \\
 &\leq  C\Big(\sum_{T\in\cT_h}\frac{1}{\underline{D_{11}}}\Vert\bxi_{\bz}\Vert_{0,\partial T}^{2} \Big )^{1/2}\vert \eta_{p}\vert_e \\&\leq C \Vert \bz\Vert_{1,\O}\vert(\betta_{\bu} ,\betta_{\bomega},\eta_p)\vert_{\mathcal{A}}. 
\end{align*}
Then, using \eqref{energy projection-error bound} and \eqref{inf sup with ep}, we can infer that 
\begin{equation*}
T_2\leq C h^{\min\{s,\,k+1\}}\Vert(\bu,p)\Vert_{s}\Vert e_p\Vert_{0,\O}.
\end{equation*}
Next, using Cauchy-Schwarz's inequality together with Lemma \ref{numerical-error bounds DG}
and \eqref{energy projection-error bound}, we get
$$
  T_3=\vert c(\betta_{\bomega}, \bxi_{\bz})\vert \leq \frac{1}{\sqrt{\nu}}C
  h \Vert\bbbeta\Vert_{\infty}\Vert\betta_{\bomega}\Vert_{0,\O}\Vert\bz\Vert_{1,\O}
  \leq C\mu_{h} h^{\min\{s,k+1\}}\Vert(\bu,p)\Vert_{s}\Vert e_p\Vert_{0,\O},
$$
where $\mu_{h}=\frac{h\Vert\bbbeta\Vert_{\infty}}{\sqrt{\nu}}$.
Similarly, we have:
\begin{equation*}
  T_4=\vert j(\betta_{\bu},\bxi_{\bz})\vert \leq \vert\betta_{\bu} \vert_{j} \vert\bxi_{\bz} \vert_{j}\leq C \vert(\betta_{\bu},\betta_{\bomega},\eta_{p})\vert_{\mathcal{A}}\Vert\bz\Vert_{1,\O}
\leq  C h^{\min\{s,\,k+1\}}\Vert(\bu,0)\Vert_{s}\Vert e_p\Vert_{0,\O}.
\end{equation*}

The terms $T_5$, $T_6$ and $T_7$ can be readily estimated, much in the same way as before, 
 using the error bound in \eqref{energy error bound DG} and the fact that $\bz\in \H^{1}_{0}(\O)^3$.

Finally, the pressure estimate follows after putting all individual bounds back into \eqref{terms to estim for error on p}.
\end{proof}

\section{Numerical tests}\label{sec:numer}

\begin{figure}[h!]
\begin{center}
\subfigure[]{\includegraphics[width=0.325\textwidth]{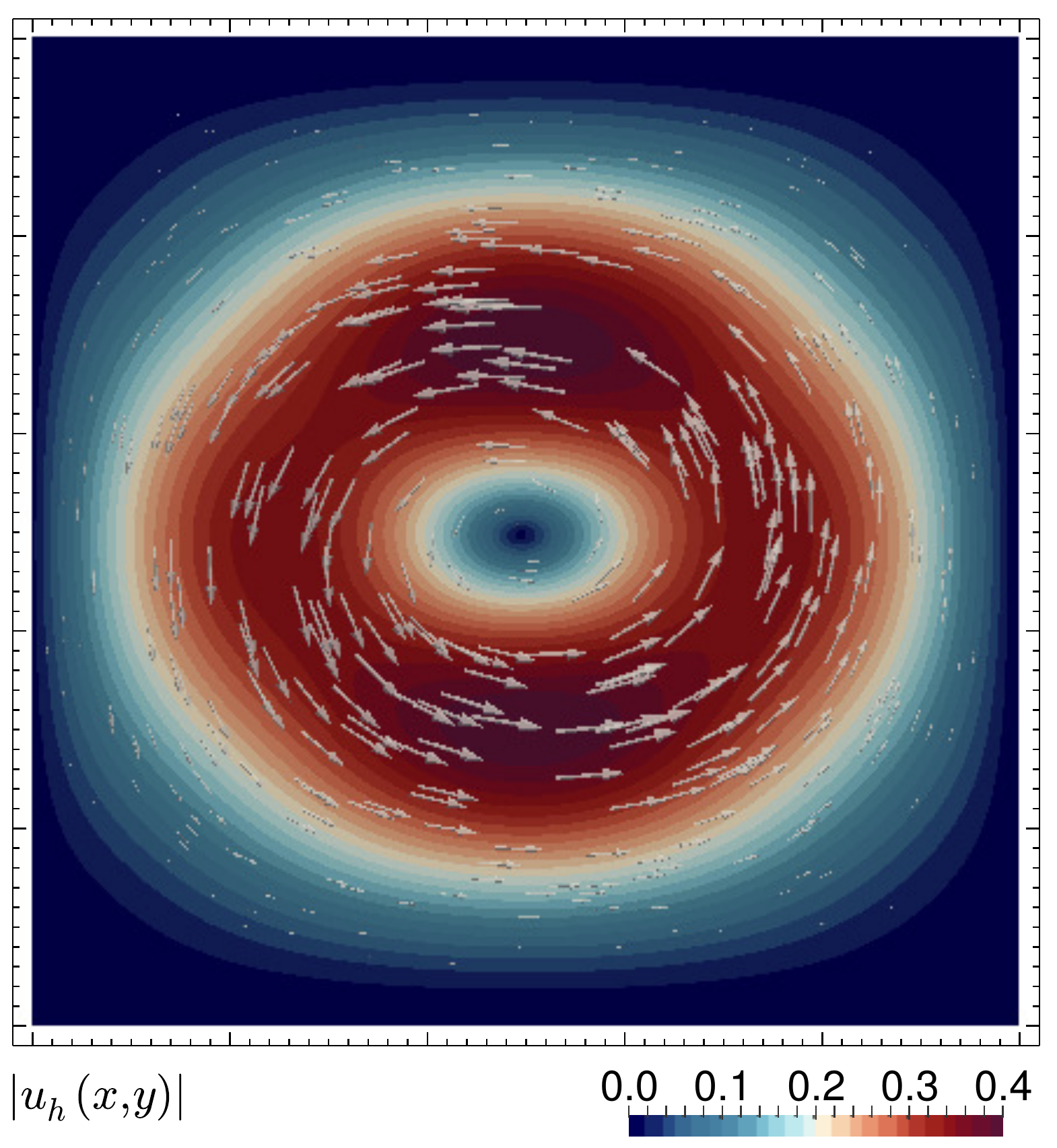}}
\subfigure[]{\includegraphics[width=0.325\textwidth]{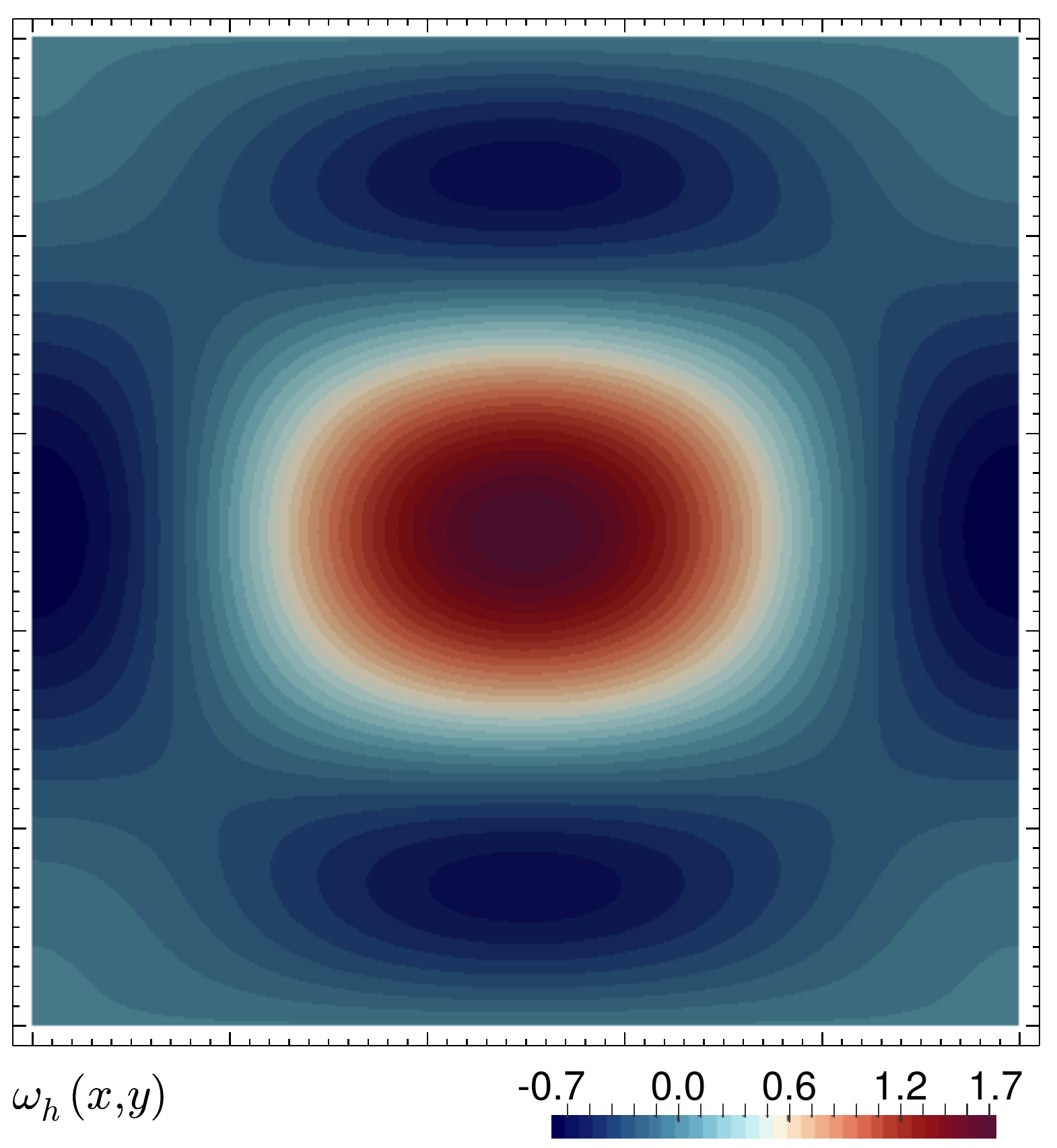}}
\subfigure[]{\includegraphics[width=0.325\textwidth]{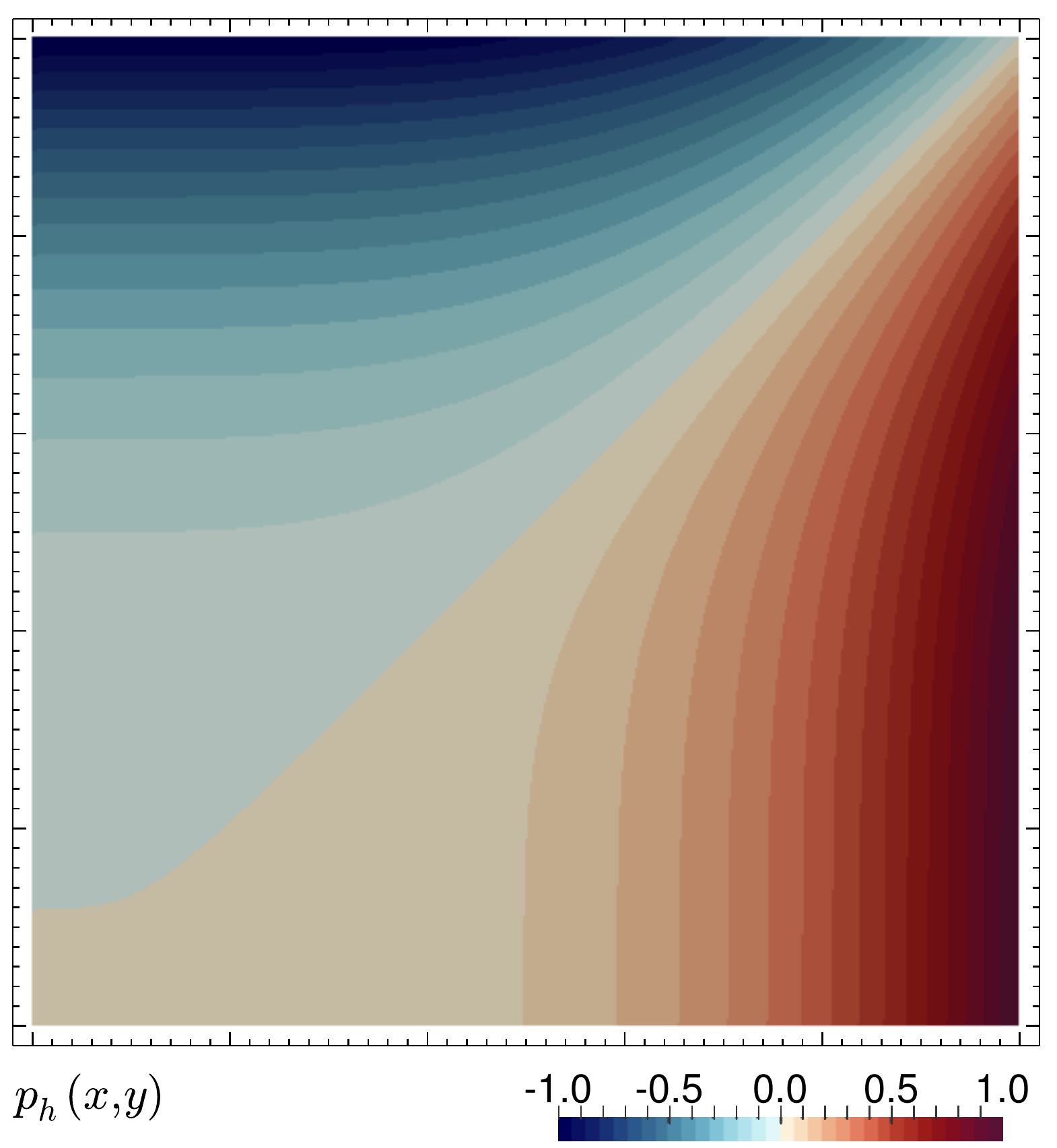}}
\end{center}

\vspace{-0.5cm}
\caption{Experimental convergence in 2D. Lowest-order mixed finite element approximation of velocity 
magnitude (a), vorticity (b), and Bernoulli pressure (c) on the unit square. }
\label{fig:ex01}
\end{figure}

We present a set of examples to confirm numerically the
convergence rates anticipated in Theorem~\ref{theo:conv} and Theorem~\ref{th: convergence DG} . We stress that
whenever $\Gamma=\partial\Omega$, the
zero-mean condition enforcing the uniqueness of the Bernoulli pressure is
implemented using a real Lagrange multiplier (which amounts to add 
one row and one column to the corresponding matrix system). Linear solves 
are performed with the direct method SuperLU.

\paragraph{Test 1: Experimental convergence in 2D.}  For our first 
examples we produce the error history associated with the proposed mixed finite element and mixed 
DG approximations. Let us consider the following closed-form solutions to the Oseen equations 
defined on the unit square domain $\Omega=(0,1)^2$:
\begin{align*}
\bu(x,y) =\! \begin{pmatrix}
\sin(\pi x)^2\sin(\pi y)^2\cos(\pi y)\\
-\frac{1}{3}\sin(2\pi x)\sin(\pi y)^3\end{pmatrix}, \quad 
\bomega (x,y)= \sqrt{\nu} \curl\bu,\quad 
p(x,y) = x^4 - y^4.\end{align*}
The exact velocity has zero normal component on the whole boundary, and the exact 
vorticity is employed to impose a non-homogeneous vorticity trace. In this 
example we are assuming 
that $\Gamma=\partial\Omega$, and 
the exact Bernoulli 
pressure fulfils the null-average condition. We consider the model parameters 
$\nu = 0.1$ and $\sigma = 10$, and the convecting velocity $\bbbeta$ is taken as the exact 
velocity solution, which in particular satisfies the bound \eqref{beta-assumption}.  
On a sequence of uniformly refined meshes we compute 
errors between the exact and approximate solutions,  measured in the
norms involved in the convergence analysis of Section~\ref{sec:FE}. 
The obtained error history is reported in
Table~\ref{table01}, where the rightmost column displays the $\ell^\infty-$norm 
of the nodal values of the velocity divergence projected to the space $\Q_h$, 
all approaching machine precision. 
The asymptotic $O(h^{k+1})$ decay of the error
observed for each field variable 
confirms the overall optimal convergence 
predicted by Theorem~\ref{theo:conv}. Sample approximate solutions
generated with the lowest order method on a coarse mesh are portrayed in
Figure~\ref{fig:ex01}.

An analogous test is now carried out to confirm numerically the convergence 
rates of the DG methods defined by \eqref{Mixed DG scheme}. The same model 
parameters and closed-form solutions are used, and the stabilisation constants 
in \eqref{choice param stab C11}-\eqref{choice param stab D11} take the values 
$a_{11}=c_{11}=\sigma$ and $d_{11}=\nu$. The results collected in Table~\ref{table01b} 
indicate that the DG scheme converges optimally when we measure errors in the energy 
$\mathcal{A}-$seminorm \eqref{energy DG norm} and in the $\L^2-$norm of the pressure. 
Here however, we do not expect divergence-free approximate velocities.

\begin{table}[t]
\begin{center}
{\small\begin{tabular}{|l|c|c|c|c|c|c|c|}
\hline
DoF  &    $\|\bu-\bu_h\|_{\H}$  &  \texttt{rate} &  $\|\bomega-\bomega_h\|_{\Z}$  
&   \texttt{rate}  &  $\|p-p_h\|_{0,\Omega}$  &  \texttt{rate} & $\|\vdiv\bu_h\|_{\ell^\infty}$ \\  
\hline
\multicolumn{8}{|c|}{$k=0$}\\
\hline
34  & 0.1357 &    -- & 1.2943 &     -- & 0.2002 &    --  & 1.0e-14\\
114 & 0.1129 & 0.2654 & 1.0072 & 0.3612 & 0.1219 & 0.7161 & 1.3e-14\\
418 & 0.0619 & 0.8655 & 0.5623 & 0.8405 & 0.0572 & 1.0910 & 1.3e-15\\
1602  & 0.0315 & 0.9763 & 0.2869 & 0.9707 & 0.0280 & 1.0311 &1.2e-15 \\
6274 & 0.0158 & 0.9952 & 0.1441 & 0.9937 & 0.0139 & 1.0072 & 1.0e-14\\
24834 & 0.0079 & 0.9989 & 0.0721 & 0.9985 & 0.0069 & 1.0022 & 1.3e-15\\
98818 & 0.0039 & 0.9997 & 0.0361 & 0.9996 & 0.0035 &    1.0000& 1.7e-13\\
\hline
\multicolumn{8}{|c|}{$k=1$}\\
\hline
98 & 0.0980 &   -- & 0.8753 &    -- & 0.0624 &   -- & 2.4e-15\\
354 & 0.0337 & 1.5384 & 0.3448 & 1.3441 & 0.0173 & 1.8502 &4.3e-14 \\
1346 & 0.0094 & 1.8472 & 0.0979 & 1.8173 & 0.0038 & 2.1804 & 3.8e-13\\
5250 & 0.0024 & 1.9514 & 0.0255 & 1.9403 & 8.3e-04 & 2.1952 & 8.0e-14\\
20738 & 6.4e-04 & 1.9873 & 0.0064 & 1.9835 & 1.9e-04 & 2.0791 & 1.2e-15\\
82434 & 2.2e-04 & 1.9973 & 0.0016 & 1.9960 & 4.8e-05 & 2.0233 & 6.2e-15\\
328706 & 3.8e-05 & 1.9992 & 4.1e-04 & 1.9992 & 1.2e-05 & 2.0064&1.4e-14\\
\hline
\multicolumn{8}{|c|}{$k=2$}\\
\hline
194 & 0.0563 &   -- & 0.5138 &   -- & 0.0237 &   -- & 1.7e-13\\
722 & 0.0078 & 2.8844 & 0.0893 & 2.7524 & 0.0024 & 3.2137 &5.8e-14 \\
2786  & 0.0011 & 2.9302 & 0.0121 & 2.9886 & 1.8e-04 & 3.1729 & 4.6e-14\\
10946 & 1.3e-04 & 2.9872 & 0.0015 & 2.9873 & 1.4e-05 & 3.2661& 4.3e-15\\
43394 & 1.6e-05 & 2.9992 & 1.9e-04 &    3.0002 & 1.4e-06 & 3.2349 & 2.7e-15\\
172802  & 2.1e-06 & 2.9981 & 2.3e-05 & 3.0014 & 1.6e-07 & 3.1252 & 4.9e-15\\
689666 & 5.3e-07 & 2.9840 & 8.2e-06 & 3.0070 & 3.7e-08 & 2.9206 &5.0e-14\\
\hline
\end{tabular}}\end{center}
\caption{Test 1A: Error history (errors on a sequence of successively refined grids, convergence rates, and 
divergence norms) associated with the mixed finite element method \eqref{probdics} using different 
polynomial degrees.} \label{table01}
\end{table}

\begin{table}[h]
\begin{center}
{\small \begin{tabular}{|l|c|c|c|c|c|}
\hline
DoF  & $h$ &   $|(\bu-\bu_h,\bomega-\bomega_h,p-p_h)|_{\mathcal{A}}$  &  \texttt{rate}  &  $\|p-p_h\|_{0,\Omega}$  &  \texttt{rate}  \\  
\hline
\multicolumn{6}{|c|}{$k=0$}\\
\hline
65    & 0.7071 & 0.8031 &  --    & 0.4623 &   -- \\
257 & 0.3536 & 0.4321 & 0.8893 & 0.2298 & 0.8384\\
1025 & 0.1768 & 0.2343 & 0.8827 & 0.1276 & 0.8438\\
4097 & 0.0884 & 0.1217 & 0.9448 & 0.0652 & 0.9668\\
16385 & 0.0442 & 0.0616 & 0.9822 & 0.0326 & 1.0010\\
65537 & 0.0221 & 0.0302 & 1.0231 & 0.0163 & 1.0005\\
\hline
\multicolumn{6}{|c|}{$k=1$}\\
\hline
145 & 0.7071 & 0.4787 &   -- & 0.1847 &    -- \\
577 & 0.3536 & 0.1496 & 1.6760 & 0.0529 & 1.8033 \\
2305 & 0.1768 & 0.0366 & 2.0297 & 0.0133 & 1.9914 \\
9217 & 0.0884 & 0.0089 & 2.0413 & 0.0033 & 2.0220 \\
36865 & 0.0442 & 0.0021 & 2.0320 & 0.0008 & 2.0361 \\
134696 & 0.0221 & 0.0005 & 2.0034 & 0.0002 & 2.0049 \\
\hline
\multicolumn{6}{|c|}{$k=2$}\\
\hline
257 & 0.7071 & 0.1882 &    -- & 0.0534 &    -- \\
1025 & 0.3536 & 0.0319 & 2.5864 & 0.0111 & 2.2722\\
4097 & 0.1768 & 0.0042 & 2.9180 & 0.0013 & 3.0093\\
16385 & 0.0884 & 0.0005 & 3.0362 & 0.0002 & 3.1105\\
65537 & 0.0442 & 6.16e-5 & 3.0640 & 1.81e-5 & 3.1510\\
268049 & 0.0221 & 1.03e-5 & 2.9973 & 2.78e-6 & 3.0076\\
\hline
\end{tabular}}\end{center}
\caption{Test 1B: Error history associated to the DG method defined in \eqref{Mixed DG scheme} 
 using increasing polynomial degree.} \label{table01b}
\end{table}

\begin{figure}[h!]
\begin{center}
\subfigure[]{\includegraphics[width=0.325\textwidth]{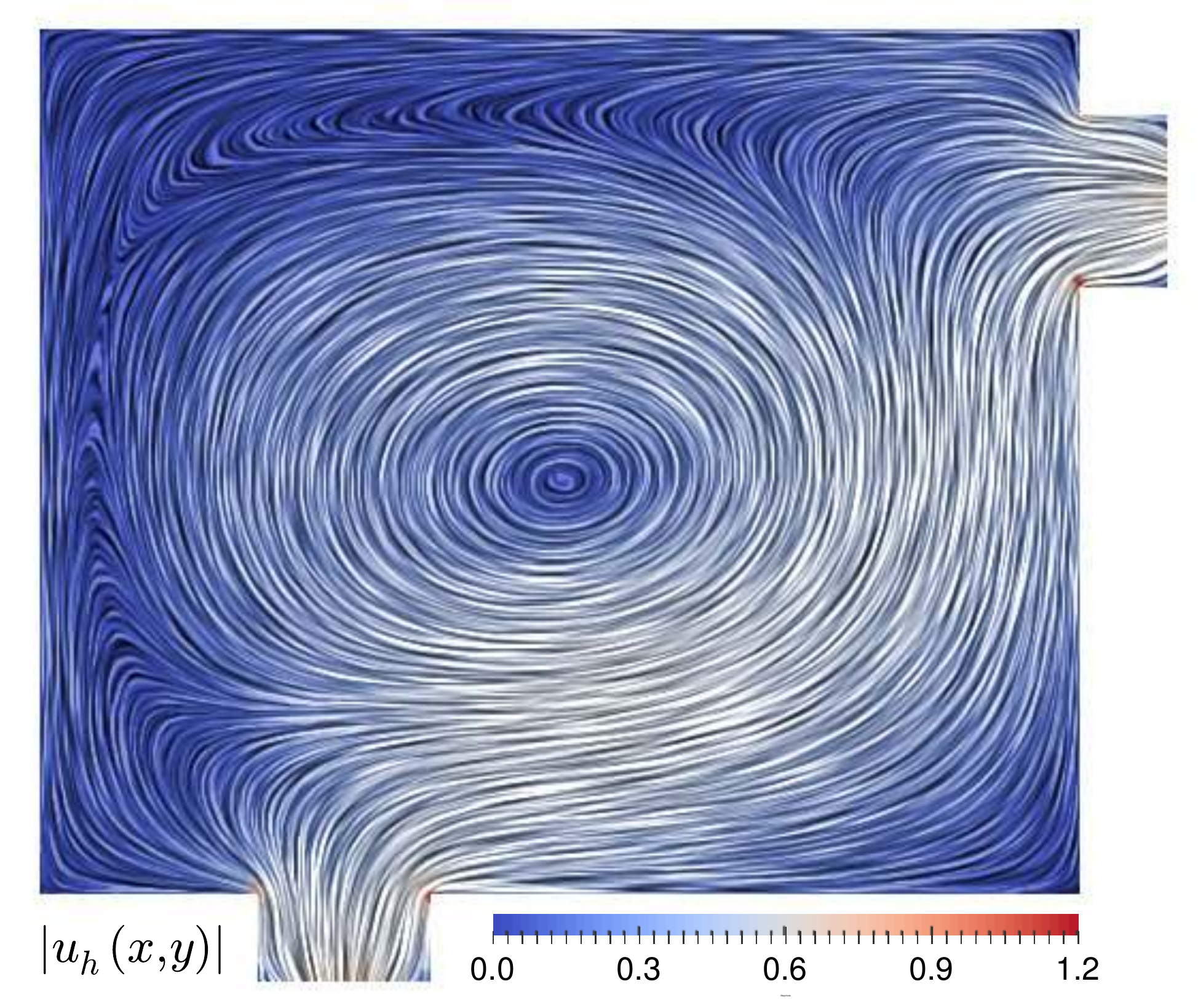}}
\subfigure[]{\includegraphics[width=0.325\textwidth]{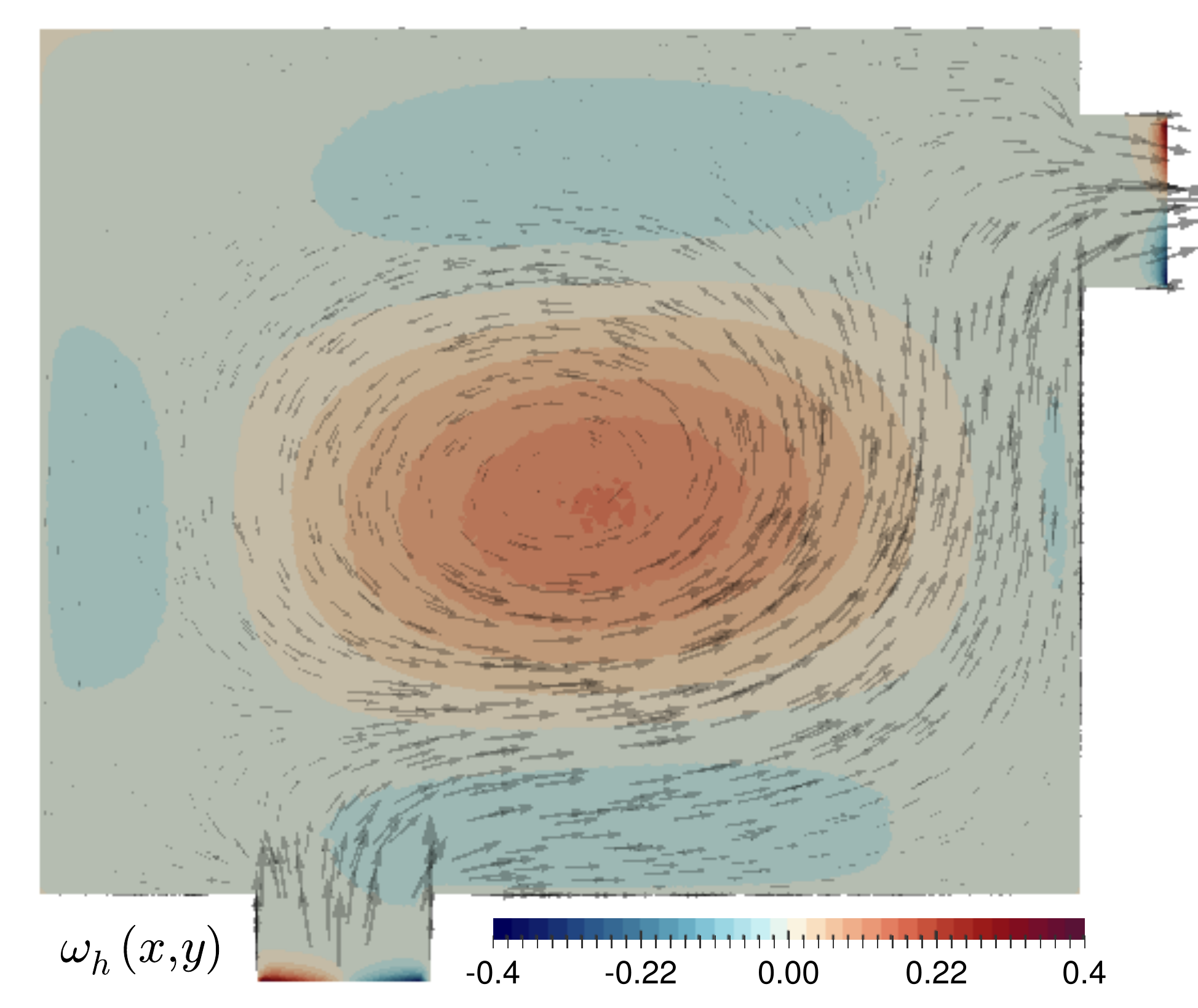}}
\subfigure[]{\includegraphics[width=0.325\textwidth]{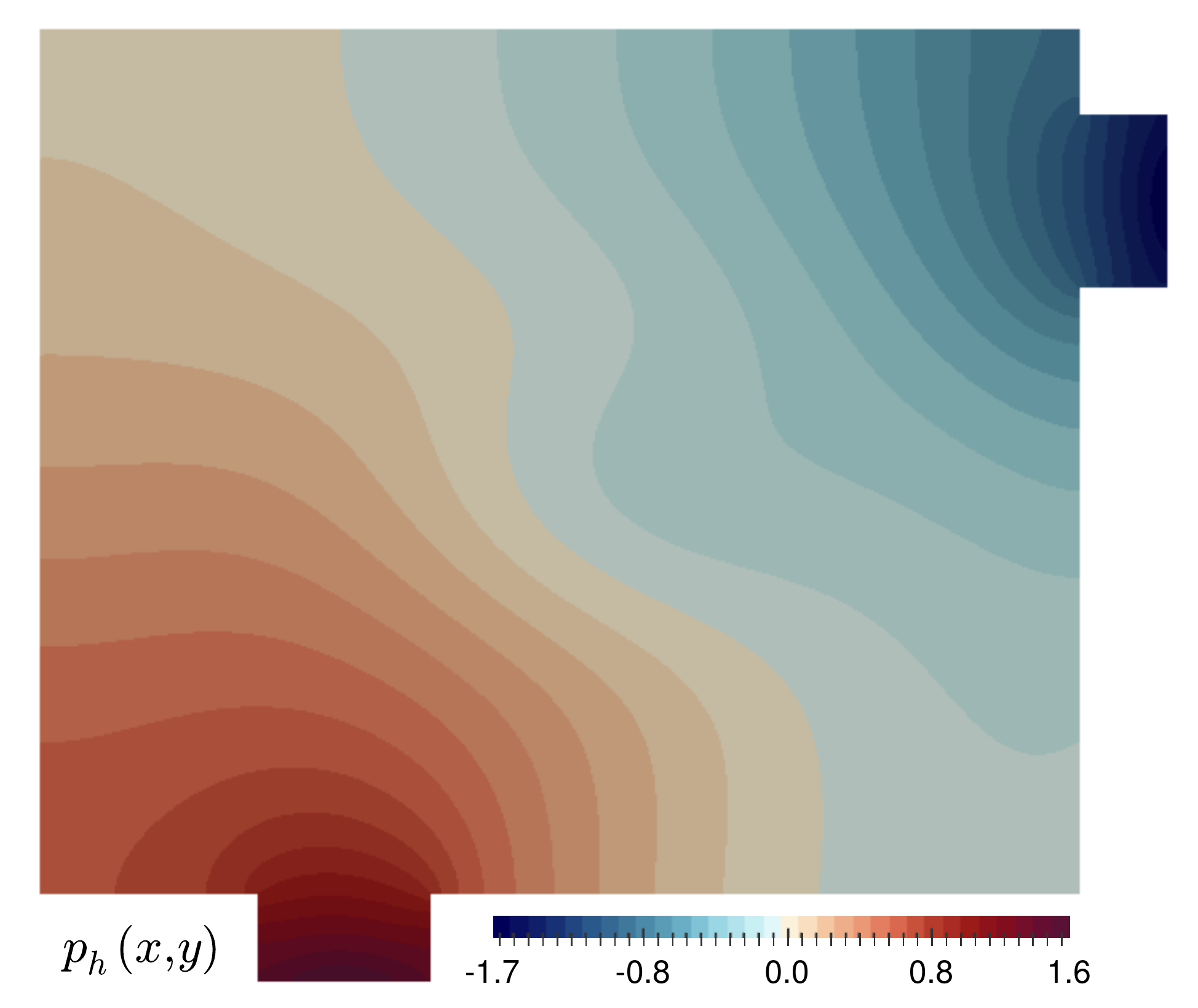}}\\
\subfigure[]{\includegraphics[width=0.325\textwidth]{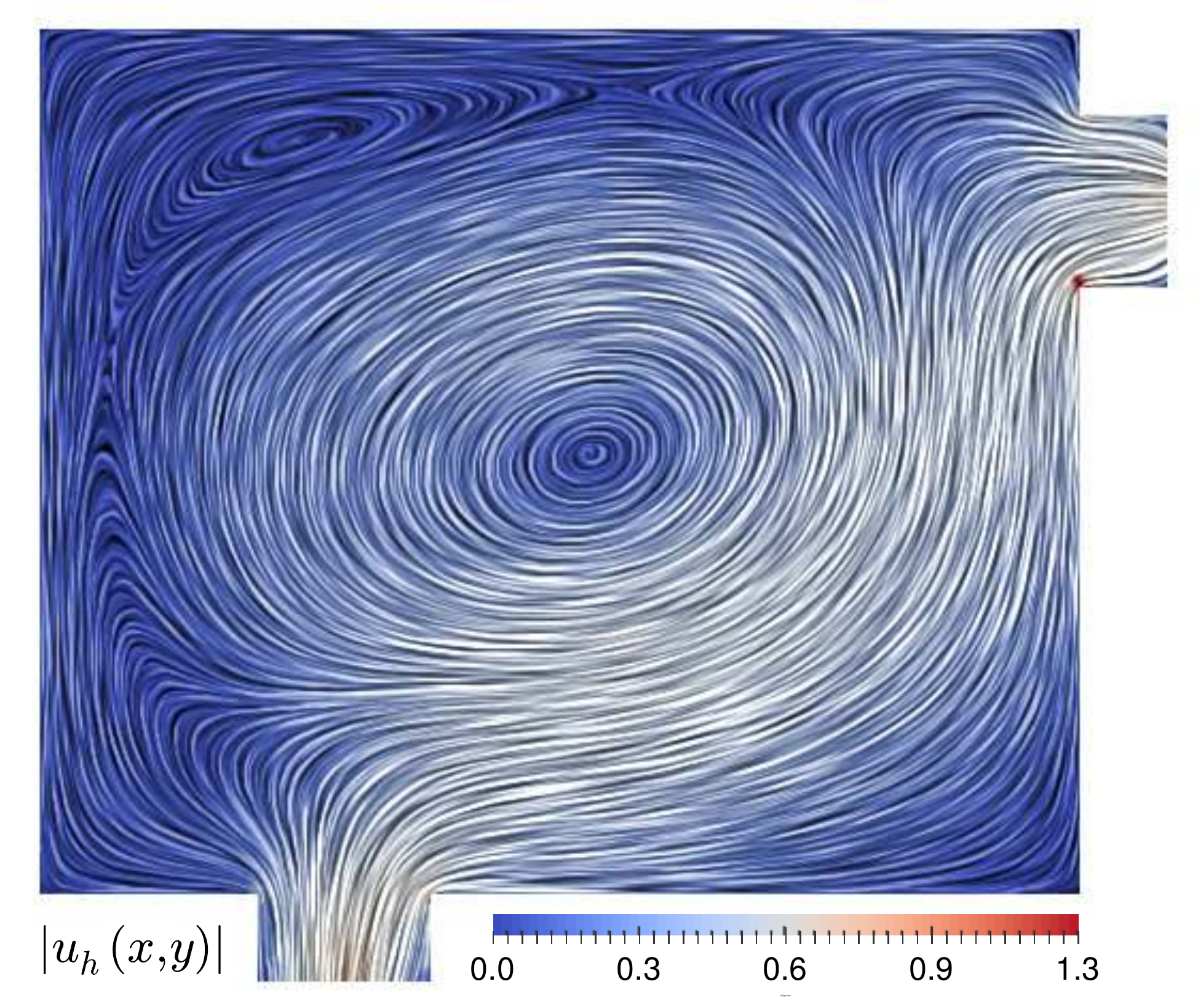}}
\subfigure[]{\includegraphics[width=0.325\textwidth]{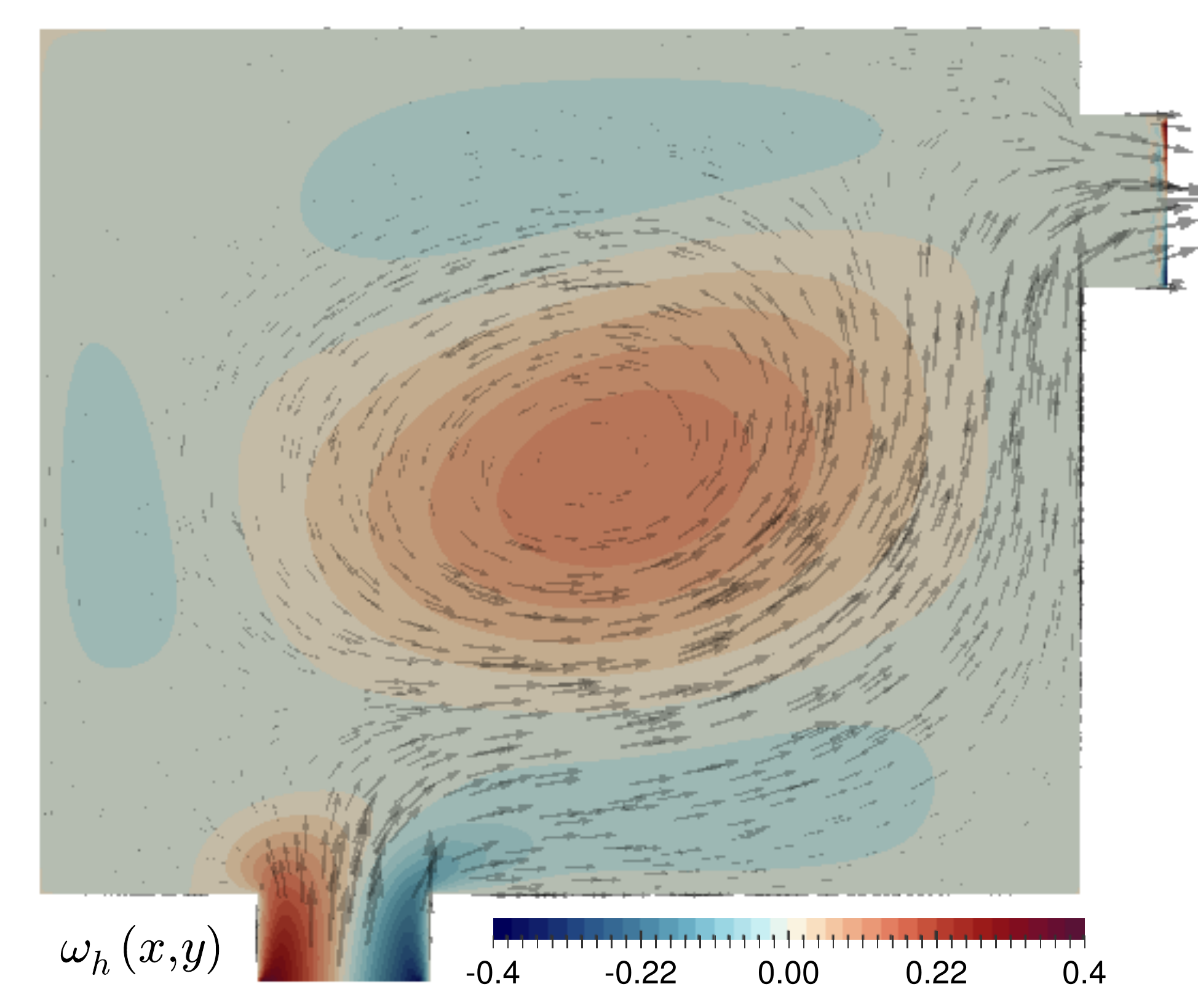}}
\subfigure[]{\includegraphics[width=0.325\textwidth]{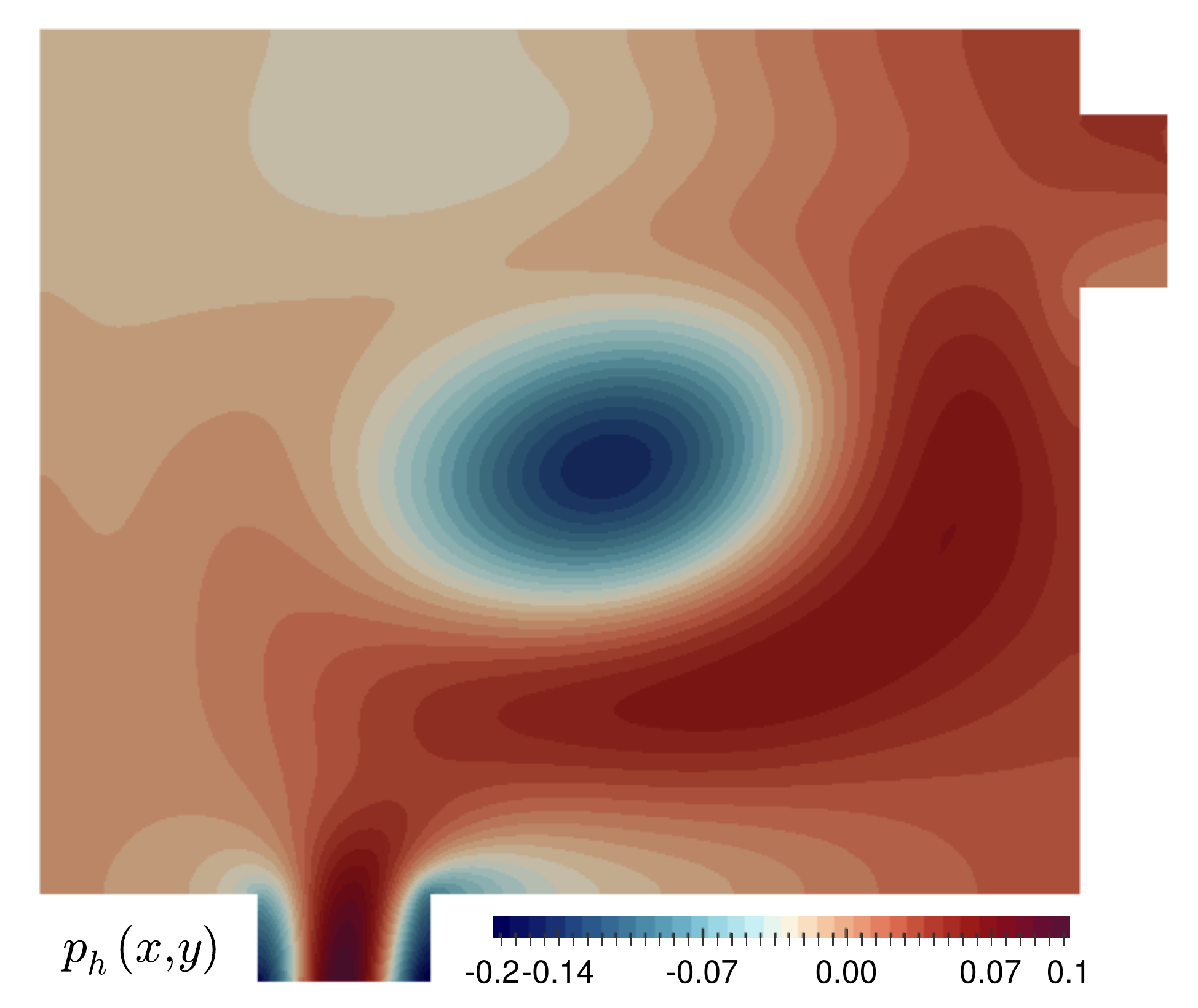}}
\end{center}

\vspace{-0.5cm}
\caption{Second-order DG approximation of the transient flow patterns 
in an open cavity after one timestep (a,b,c) and after four time steps (d,e,f), using $\Delta t = 0.1$ and 
$\nu = 0.001$. }
\label{fig:ex02}
\end{figure}

\paragraph{Test 2: Transient flow in an open cavity.}  In this example we illustrate a more complex 
problem involving the non-stationary behaviour of the flow in an open 2D cavity. The main compartment  
of the domain consists on a rectangle $(0,1.2)\times(0,1)$ whereas smaller rectangles $(0.25,0.45)\times (-0.1,0)$ 
and $(1.2,1.3)\times (0.7,0.9)$ play the role of inlet and outlet channels. The domain is discretised into 
an unstructured mesh of 35433 triangular elements. 
Normal velocities and a compatible 
trace vorticity are imposed on the whole boundary $\Gamma=\partial\Omega$ according to  
$$\bu\cdot\bn = \begin{cases}
-75(x-0.25)(0.45-x) & \text{on $y=-0.1$},\\
 75(y-0.7)(0.9-y) & \text{on $x=1.3$},\\
0 & \text{otherwise},\end{cases}\quad \bomega\times\bn = \begin{cases}
75\sqrt{\nu}(0.7-2x)& \text{on $y=-0.1$},\\
-75\sqrt{\nu}(1.6-2y)& \text{on $x=1.3$},\\
0 & \text{otherwise},\end{cases}$$
that is, parabolic inlet and outlet profiles together with slip velocities elsewhere on $\partial\Omega$. 
We set a fluid viscosity of $\nu = 0.001$ and use as initial velocity the solution adapted from the previous test 
$\bu_0(x,y) =[\sin(\pi/1.3 x)^2\sin(\pi/1.1 (y+0.1))^2\cos(\pi/1.1 (y+0.1)),
-\frac{1}{3}\sin(2/1.3\pi x)\sin(\pi/1.1 (y+0.1))^3]^T$. 
The parameter $\sigma = 10$ 
indicates a timestep of $\Delta t = 0.1$, and a backward Euler discretisation implies that we take 
$\ff = \sigma \hat{\bu}$, where $\hat{\bu}$ denotes the velocity approximation at the previous iteration. 
The convective velocity $\bbbeta = \hat{\bu}$ therefore needs to be updated at each iteration.  At least 
for the initial solution we have that the convecting velocity satisfies the assumption \eqref{beta-assumption}. 
The simulation 
is run until $T_{\text{final}} = 4\Delta t$ and we present in Figure~\ref{fig:ex02} two snapshots of the 
numerical solutions at $t = \Delta t$ and $t = T_{\text{final}}$, computed with a second-order DG scheme, and 
using the stabilisation parameters $a_{11}=c_{11}=\sigma$ and $d_{11}=\nu$. From the velocity plots (including a line 
integral convolution visualisation), we can evidence the formation of a main vortex on the centre of the domain 
plus smaller recirculation areas that emerge on the top left and bottom left corners, together with a preferential path joining the inlet and outlet 
boundaries.

\begin{figure}[t!]
\begin{center}
\subfigure[]{\raisebox{0.06\height}{\includegraphics[width=0.29\textwidth]{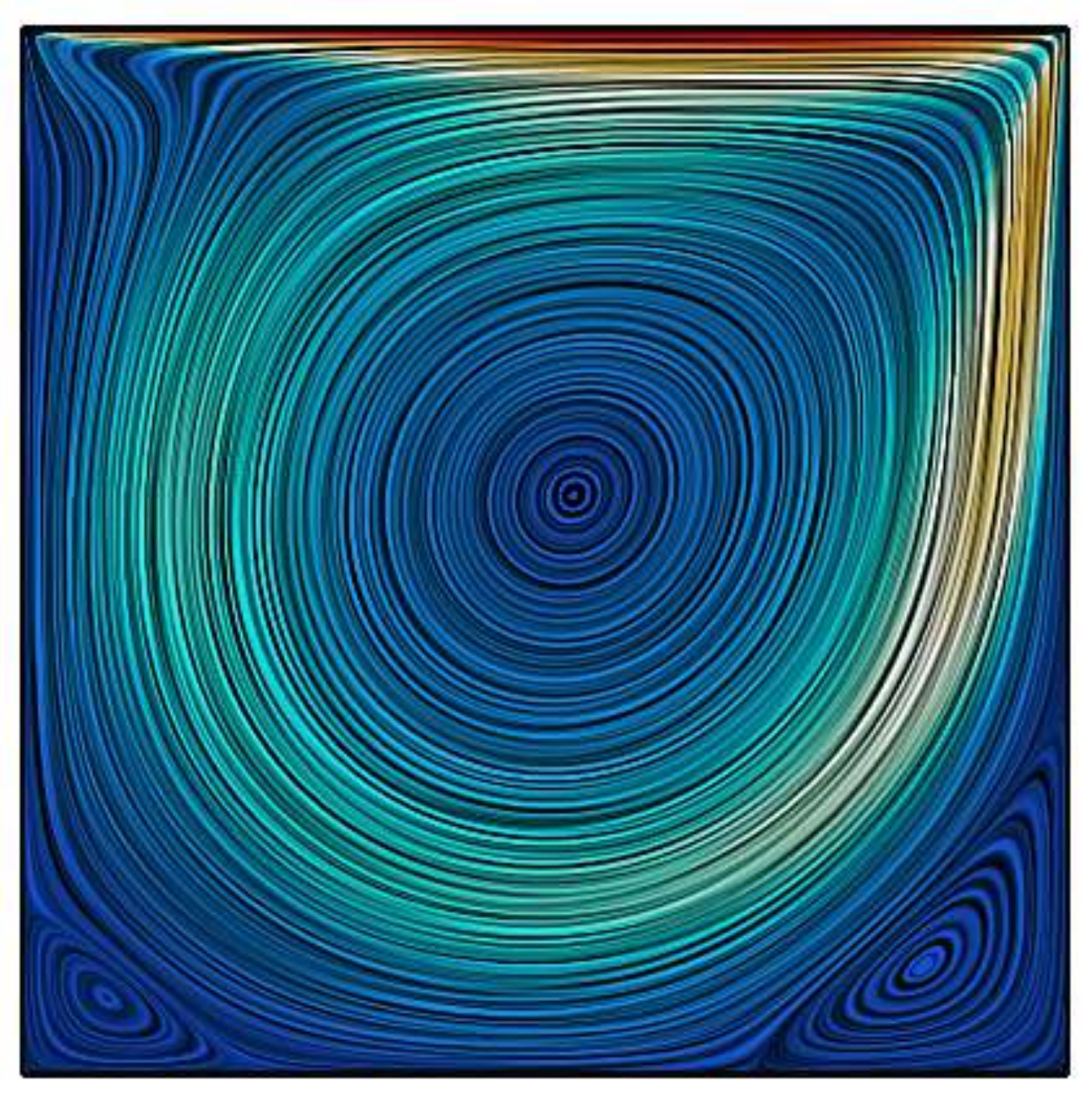}}\label{fig:2d-vortex}}\qquad 
\subfigure[]{\includegraphics[width=0.32\textwidth]{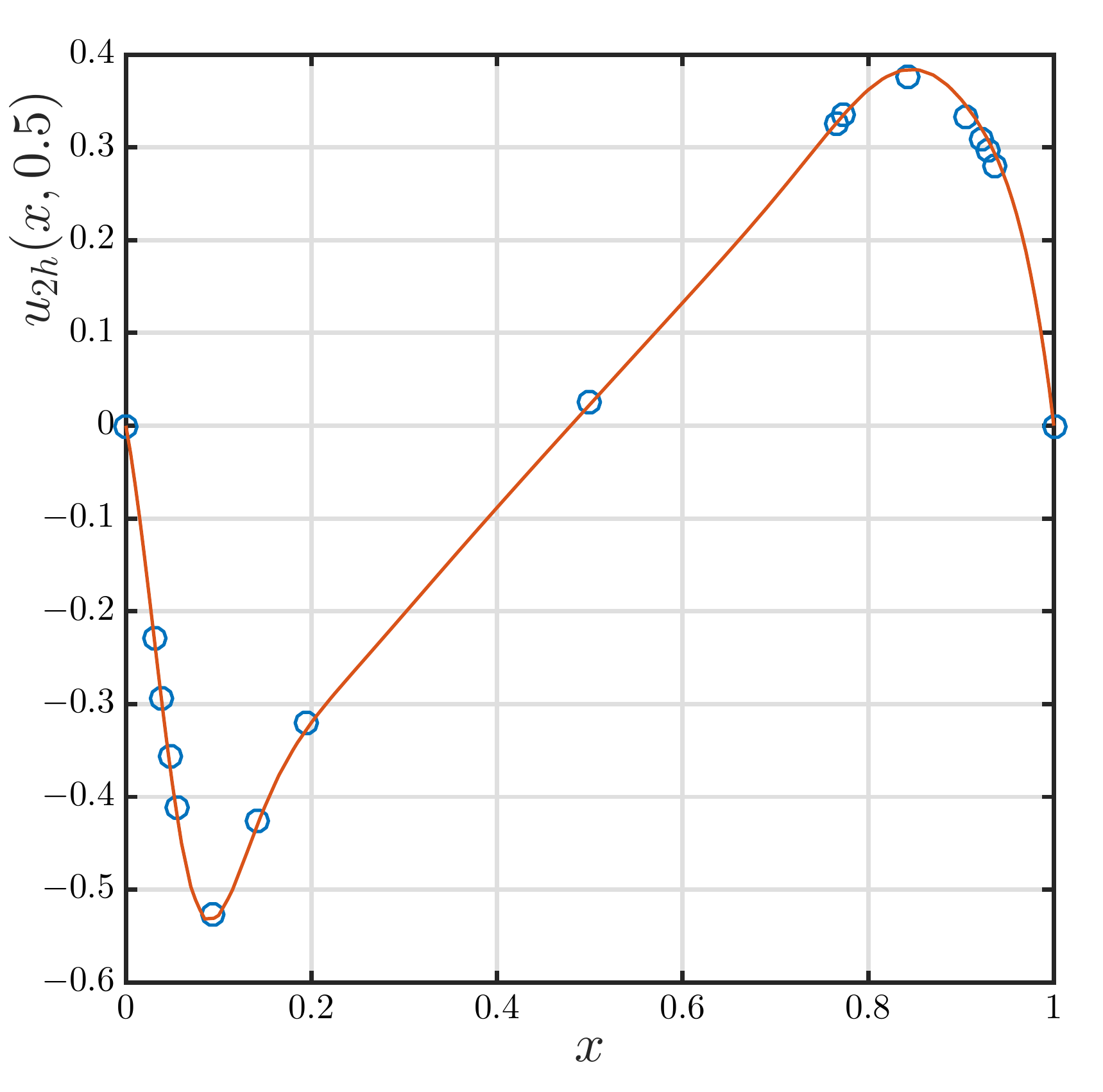} \label{fig:2d-cut1}}\\[-1ex]
\subfigure[]{\includegraphics[width=0.32\textwidth]{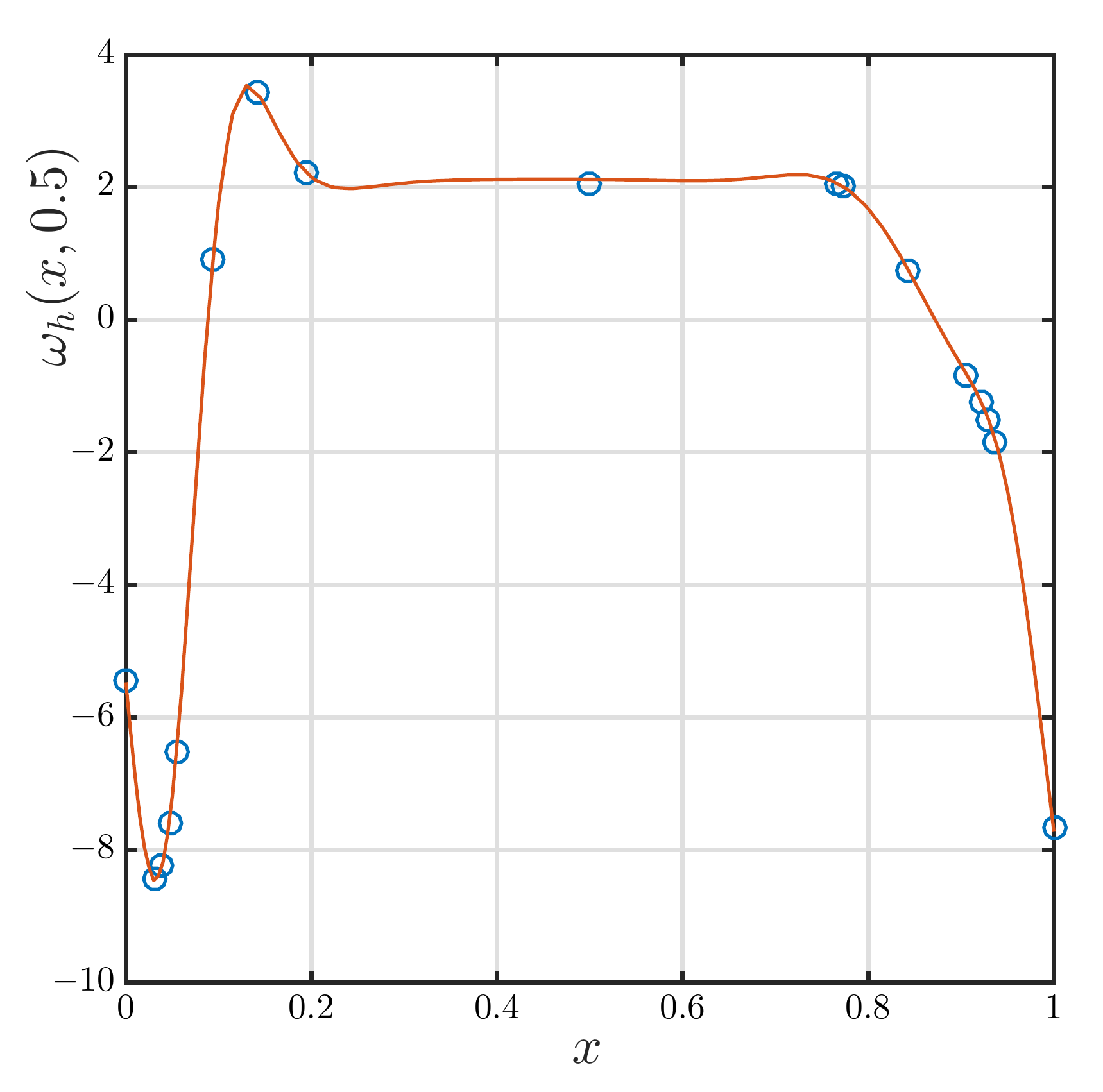}}
\subfigure[]{\includegraphics[width=0.32\textwidth]{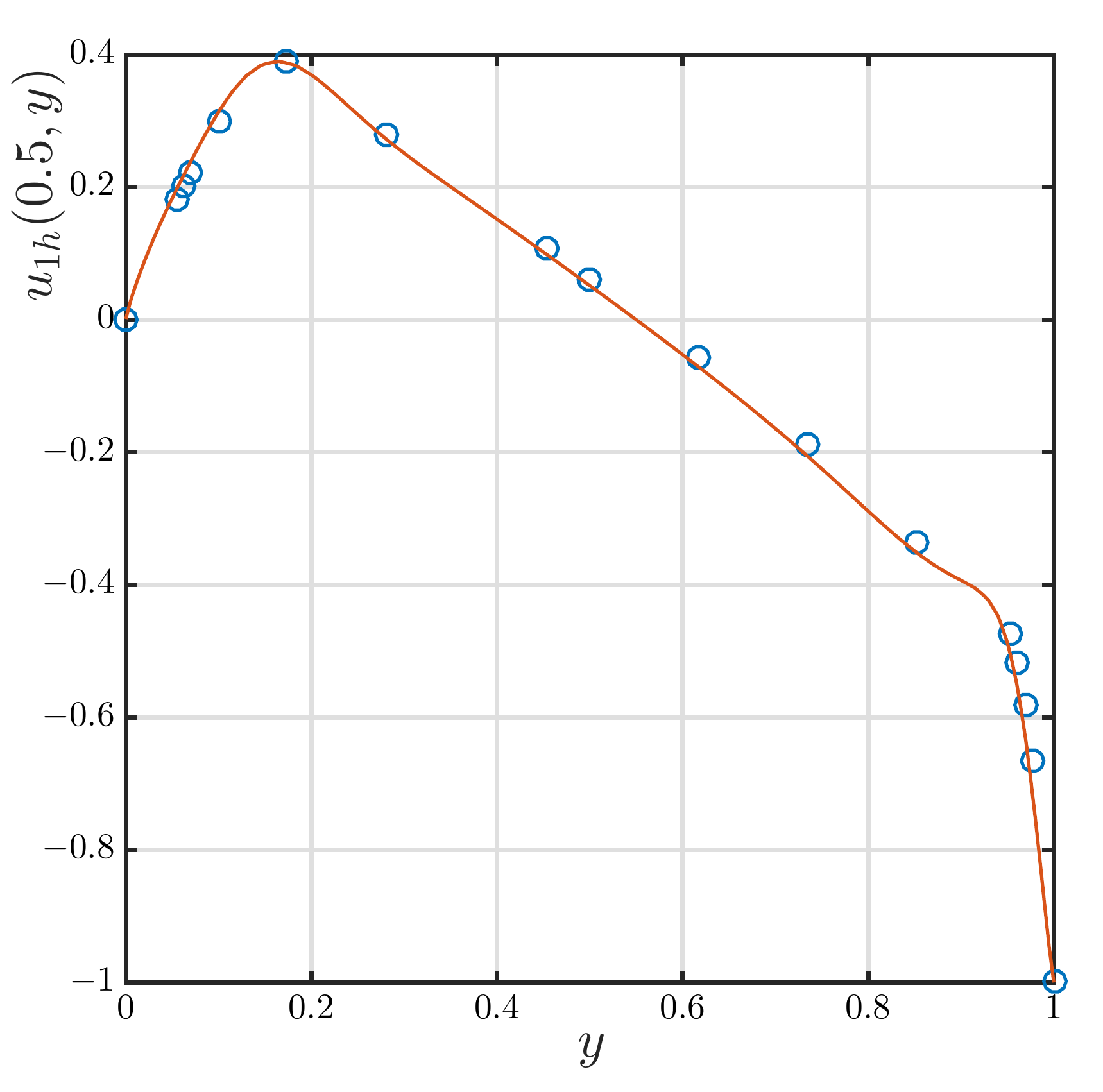}}
\subfigure[]{\includegraphics[width=0.32\textwidth]{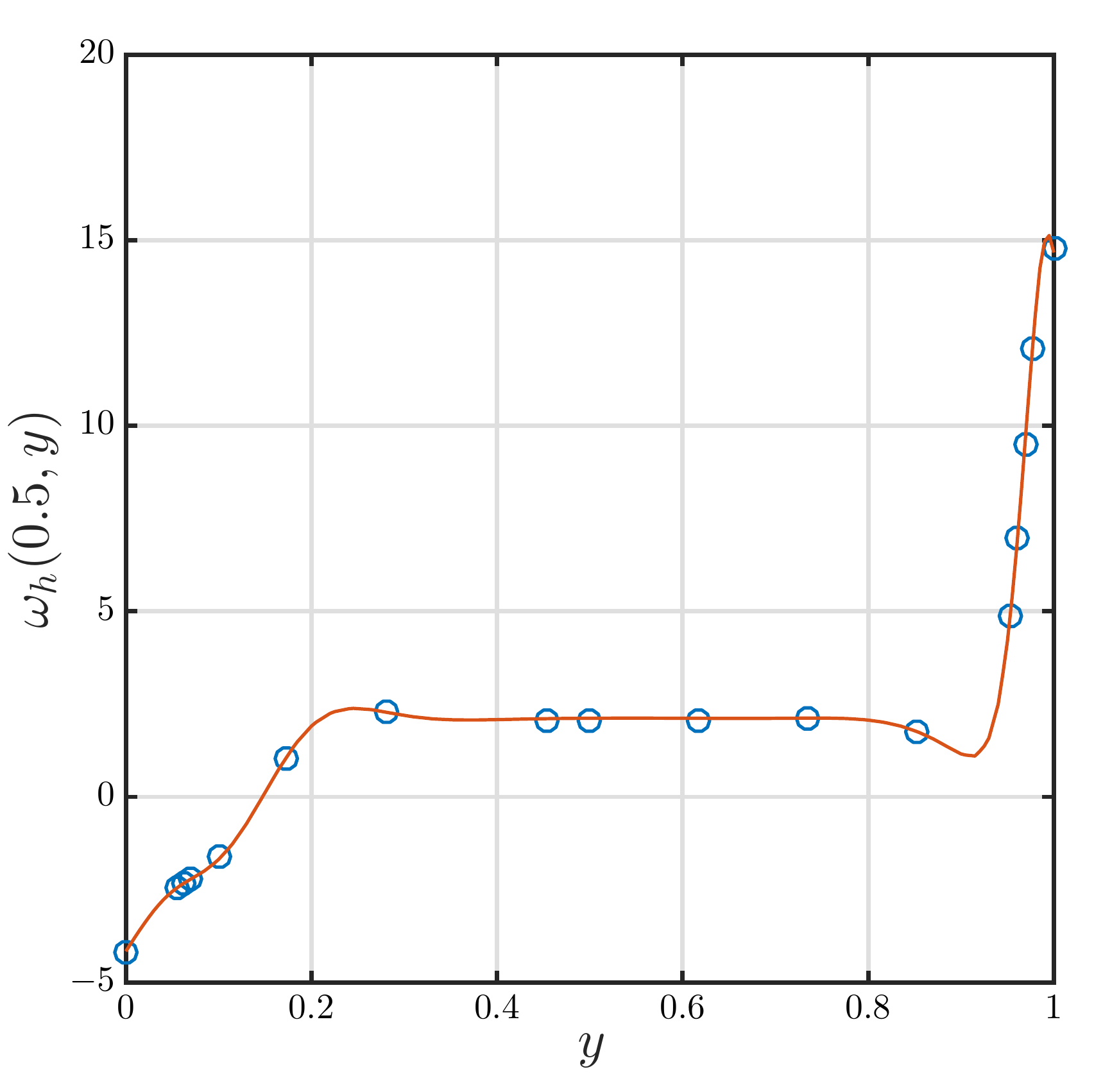}\label{fig:2d-cut4}}\\
\subfigure[]{\includegraphics[width=0.325\textwidth]{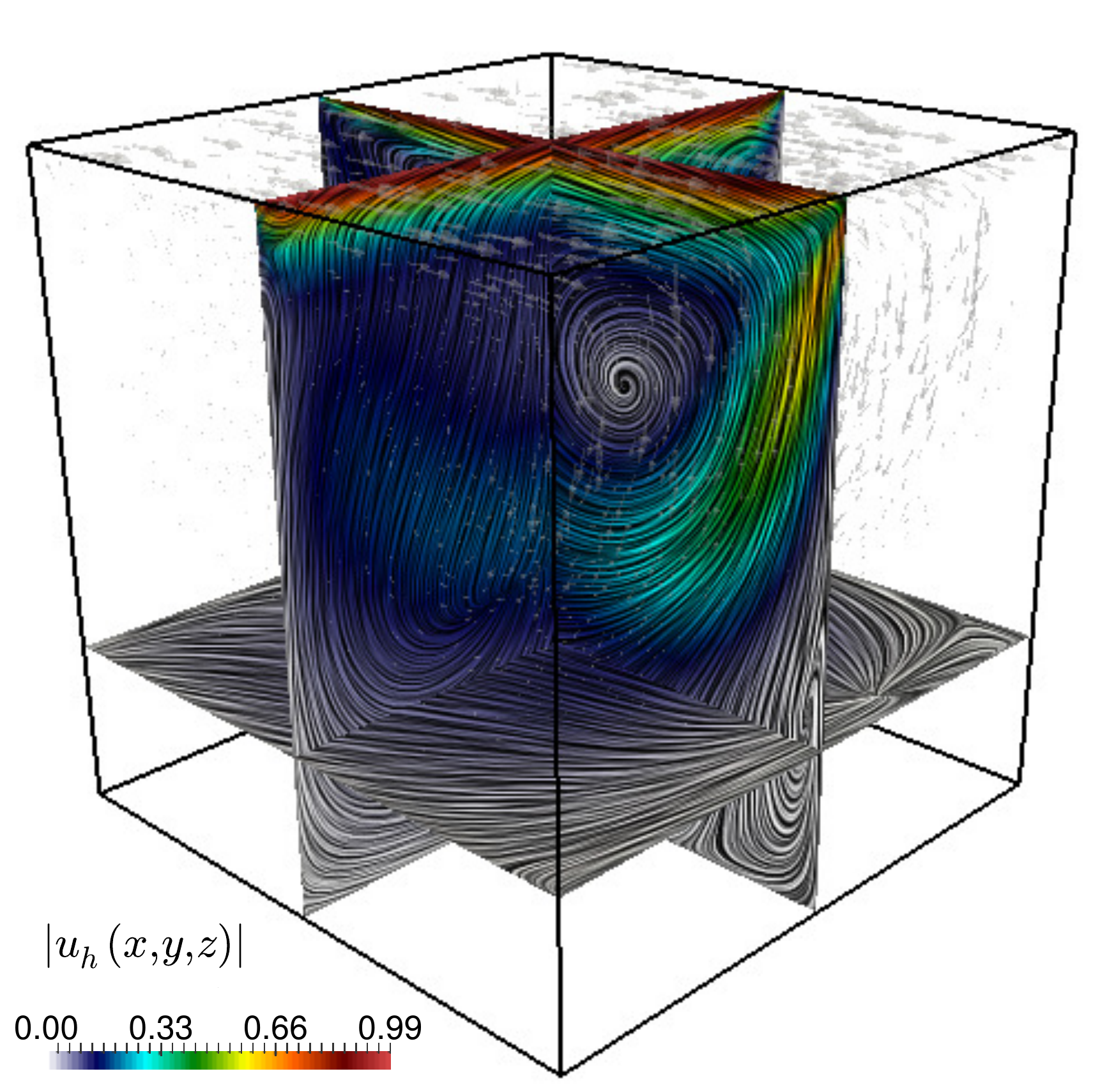}\label{fig:ex03a}}
\subfigure[]{\includegraphics[width=0.325\textwidth]{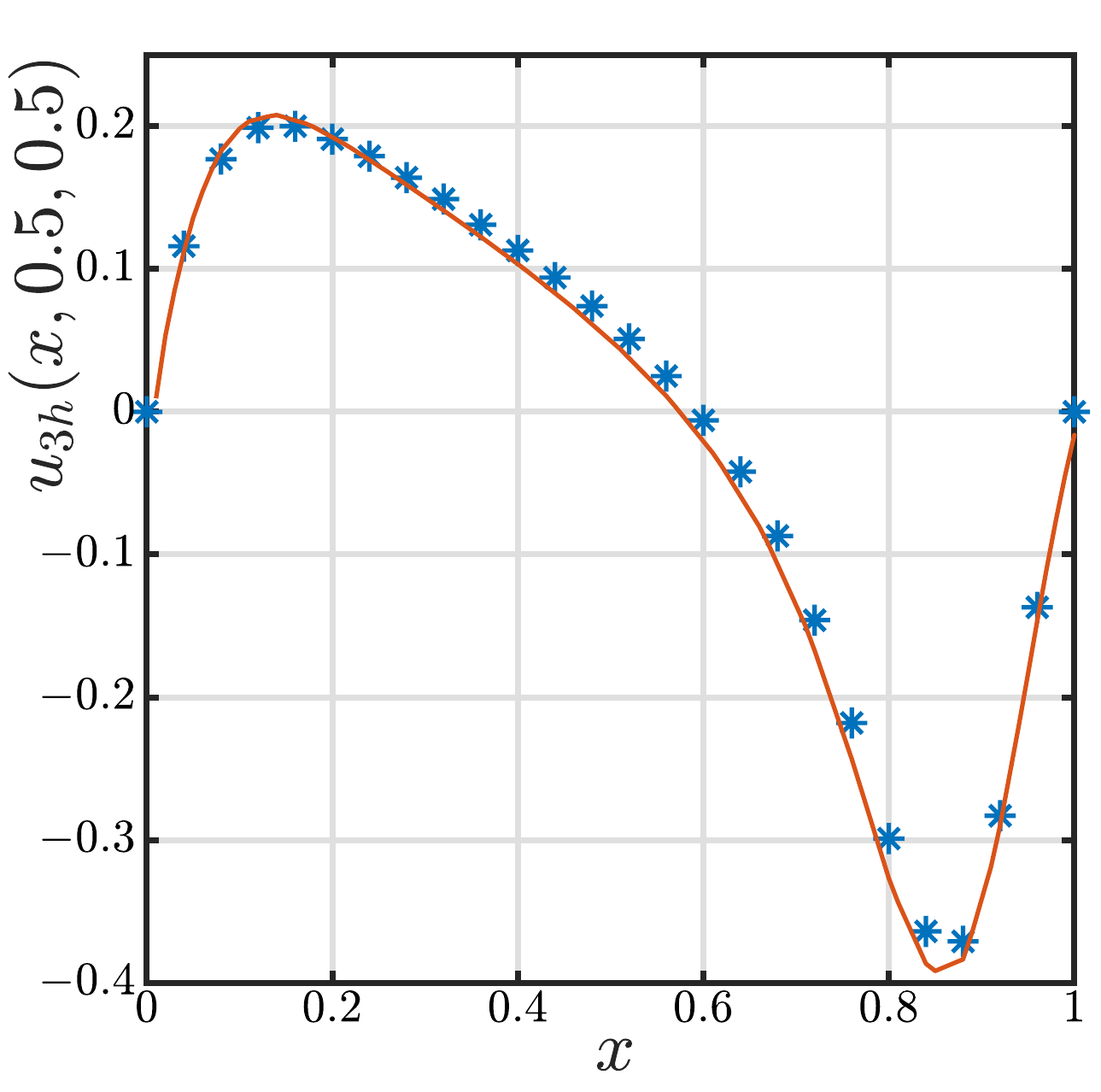}\label{fig:ex03cutA}}
\subfigure[]{\includegraphics[width=0.325\textwidth]{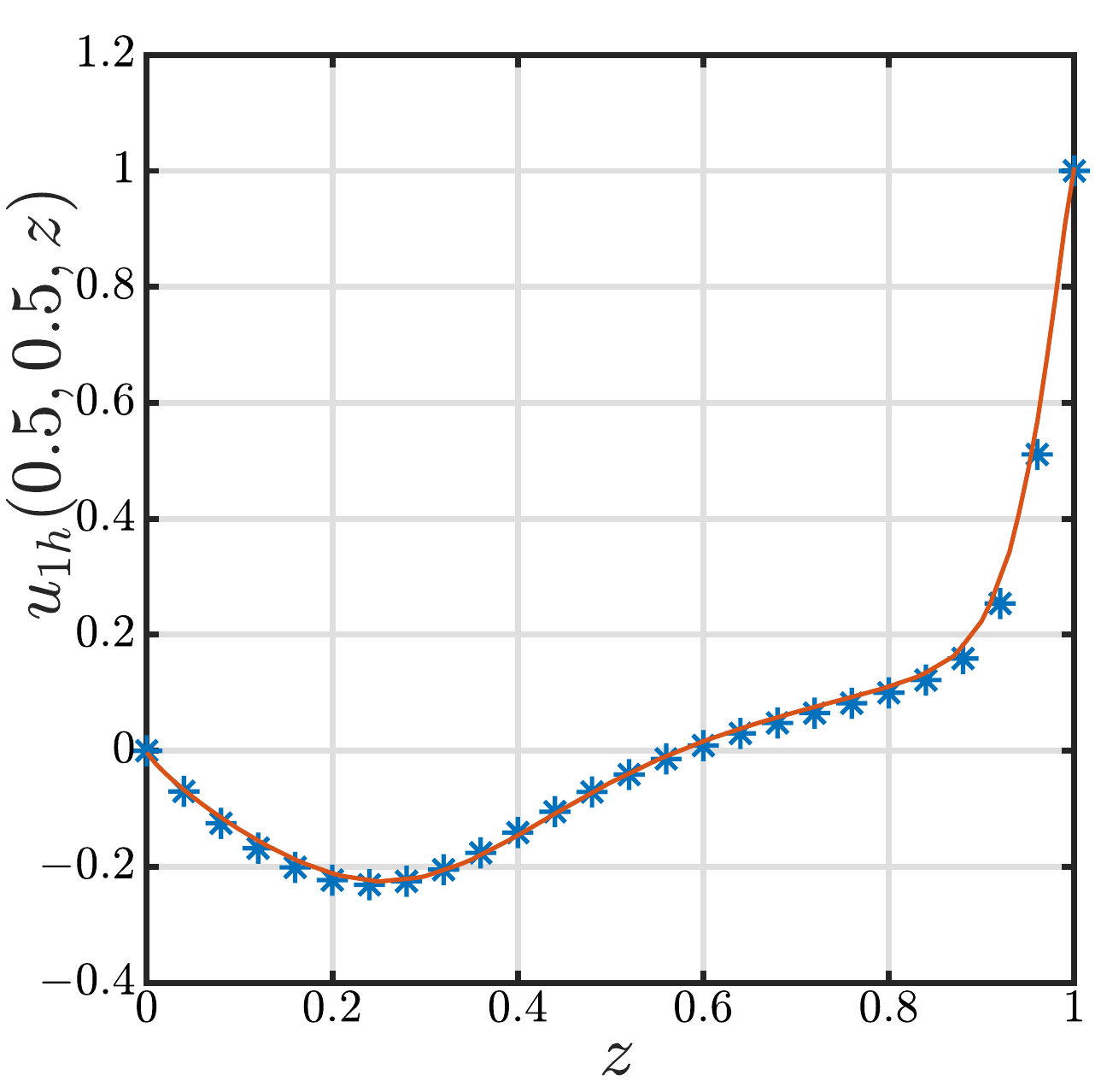}\label{fig:ex03cutB}}
\end{center}

\vspace{-0.5cm}
\caption{Test~3. Line integral convolution visualisation of velocity for 
the 2D cavity flow benchmark with $\nu =0.001$ (a), cuts of the vertical velocity and vorticity on the line $y=0.5$ (b,c); 
and horizontal velocity and vorticity on the line $x=0.5$ (d,e). The circle markers 
indicate benchmark values from \cite{botella98}.
Lowest-order approximation of velocity (f)  
for the 3D case with $\nu =0.0025$, shown at $t=20$; and profile of the horizontal velocity on the line $y=0.5,z=0.5$ (g), 
and of the vertical velocity on the line $x=0.5,y=0.5$ (h).  The asterisks  
indicate benchmark values from \cite{ding06} and all numerical approximations for this test were 
obtained with the lowest-order mixed finite element method.}
\end{figure}

\paragraph{Test 3: Lid driven cavity flow.} For this classical 
benchmark problem we consider zero external forces and concentrate on 
the case where flow recirculation occurs by Dirichlet conditions only. 
First we consider the two-dimensional case, where on the top lid of the 
unit square (at $y=1$) we set a unidirectional  velocity of 
unit magnitude, whereas no-slip conditions and zero tangential vorticity 
are imposed on the remaining sides of the boundary.  We set the 
parameters $\nu =0.001,\sigma =50$ 
and employ a structured mesh of 4096 elements. The initial velocity is 
computed from a Stokes solution (setting both $\sigma$ and $\bbbeta$ to zero). 
 We compare the results obtained with  
our lowest-order FE scheme against the benchmark data from \cite{botella98} (produced with a 
spectral method applied to a vorticity-based formulation). Figure~\ref{fig:2d-vortex} 
shows the generated velocity profile at $t=20$, having all the flow features expected 
for this regime. The solid lines in Figures~\ref{fig:2d-cut1}-\ref{fig:2d-cut4} portray cuts of the 
solution on the mid-lines of the domain, whereas the circle markers indicate benchmark data. 
The approximate vorticity has been rescaled with 
$\nu^{-1/2}$ to reflect the overall agreement with the results reported in \cite{botella98}.   

We also test the 3D implementation and formulation by conducting the same 
benchmark on the unit cube $\Omega = (0,1)^3$. Again, boundary $\Sigma$ is the 
top plate (defined by $z=1$), where we set tangential velocity of magnitude one, and on 
$\Gamma=\partial\Omega\setminus\Sigma$ we consider no-slip velocity and 
zero tangential vorticity. The fluid viscosity is now  
$\nu = 0.0025$ and a structured mesh of 58752 tetrahedral elements is employed. 
Once again we focus on a non-stationary regime with 
a backward Euler scheme, now using $\sigma = 10$, and proceed to update the convective 
velocity and the right-hand side using the velocity approximation at the 
previous time iteration. The velocity field for a converged solution after 200 time steps is 
shown in Figure~\ref{fig:ex03a}. We can observe the expected asymmetric 
vortex forming parallel to the $xz$ plane (also the generation of high pressure 
near the corners where the Dirichlet velocity datum has 
a discontinuity). We have also compared our results with the benchmark values obtained 
in \cite{ding06} (using  multiquadric differential quadratures) 
for a Reynolds number of 400: the solid lines in Figures~\ref{fig:ex03cutA}-\ref{fig:ex03cutB} 
show velocity profiles captured on the plane $y=0.5$, concentrating on the vertical and 
horizontal centrelines, where we also include the data from \cite{ding06} (in asterisks) 
showing a reasonable match (we have rotated the data, as in their tests the 
unit velocity is imposed on the face $y=1$).

\begin{figure}[h!]
\begin{center}
\subfigure[]{\includegraphics[width=0.325\textwidth]{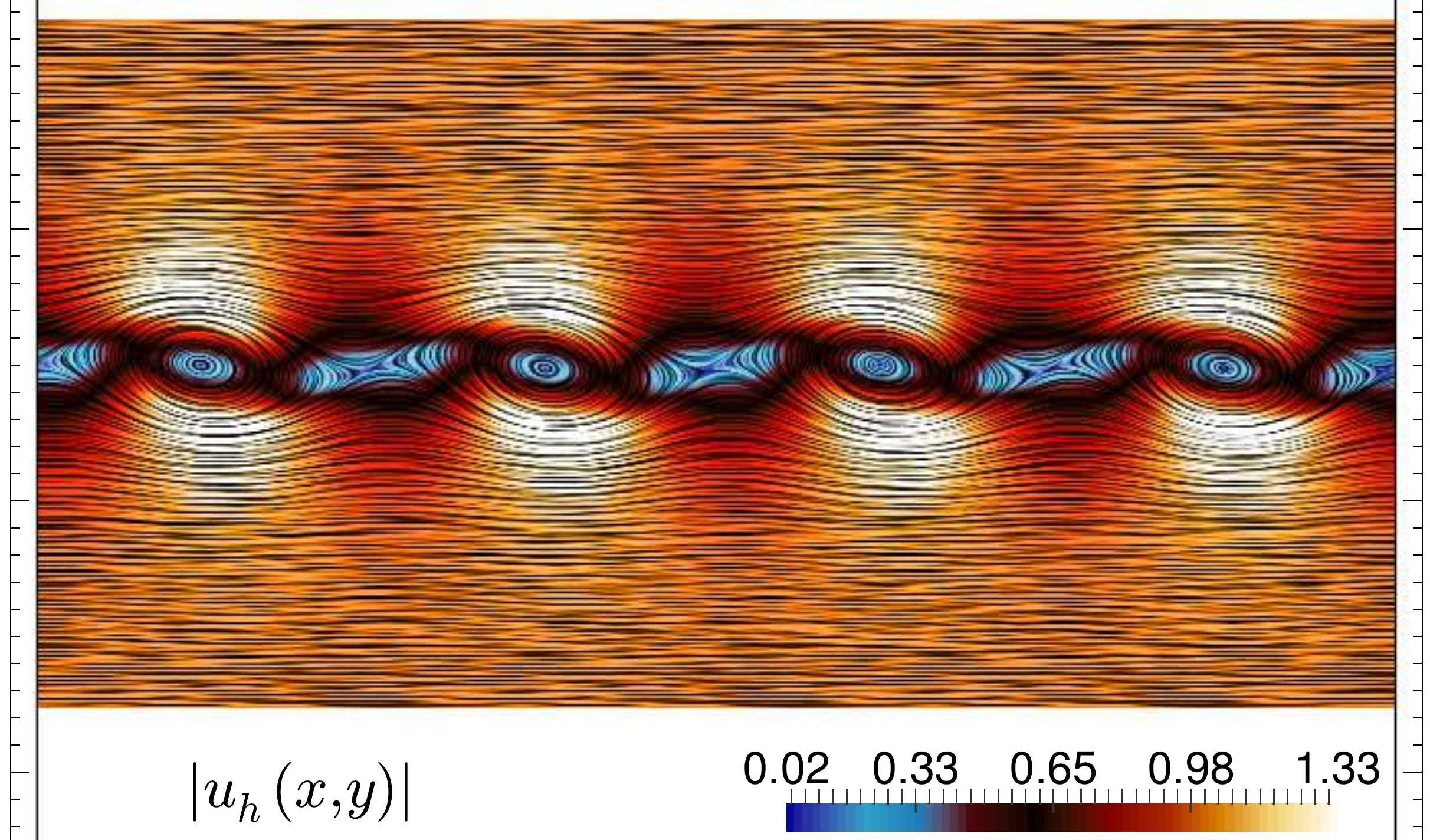}}
\subfigure[]{\includegraphics[width=0.325\textwidth]{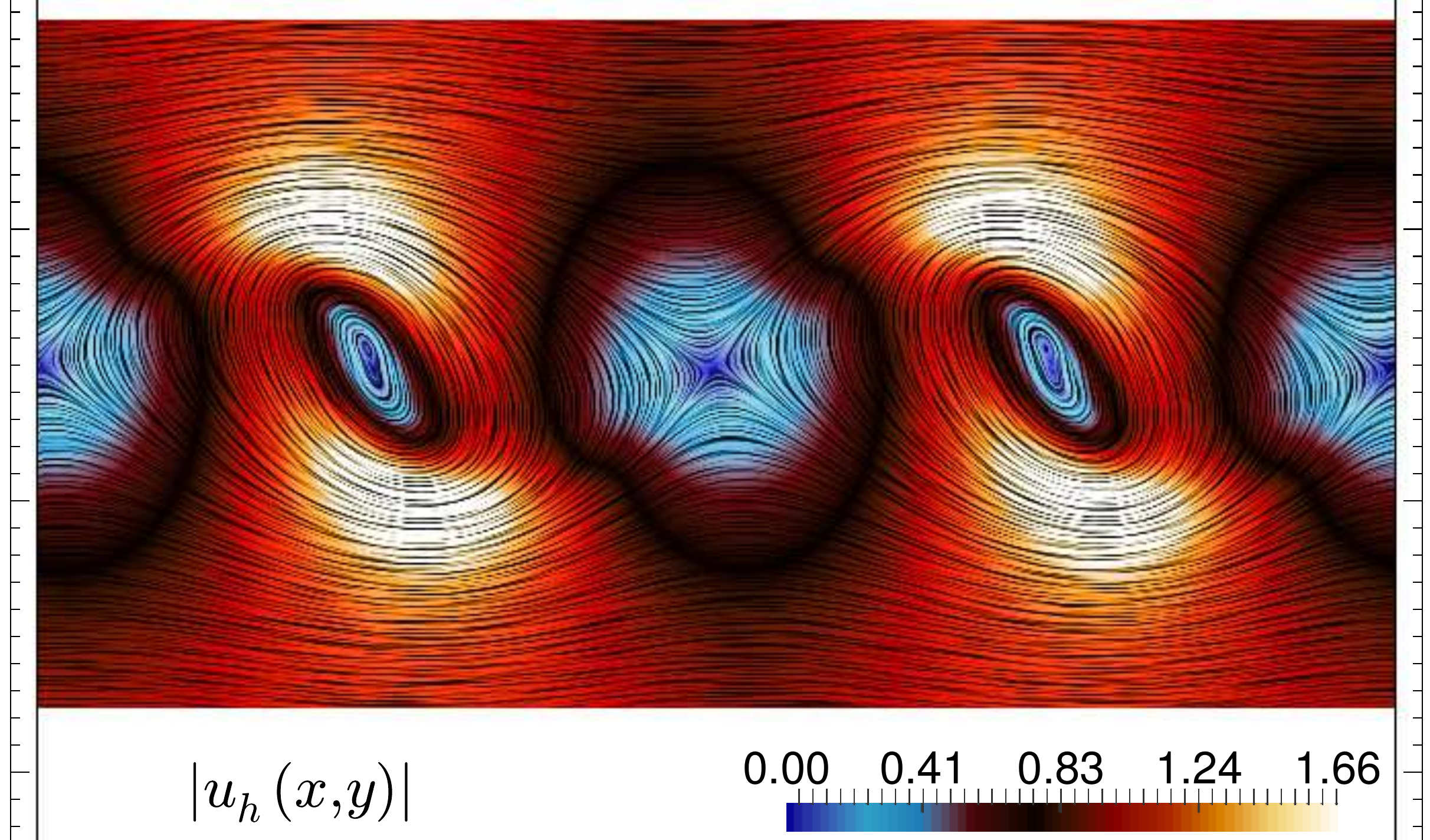}}
\subfigure[]{\includegraphics[width=0.325\textwidth]{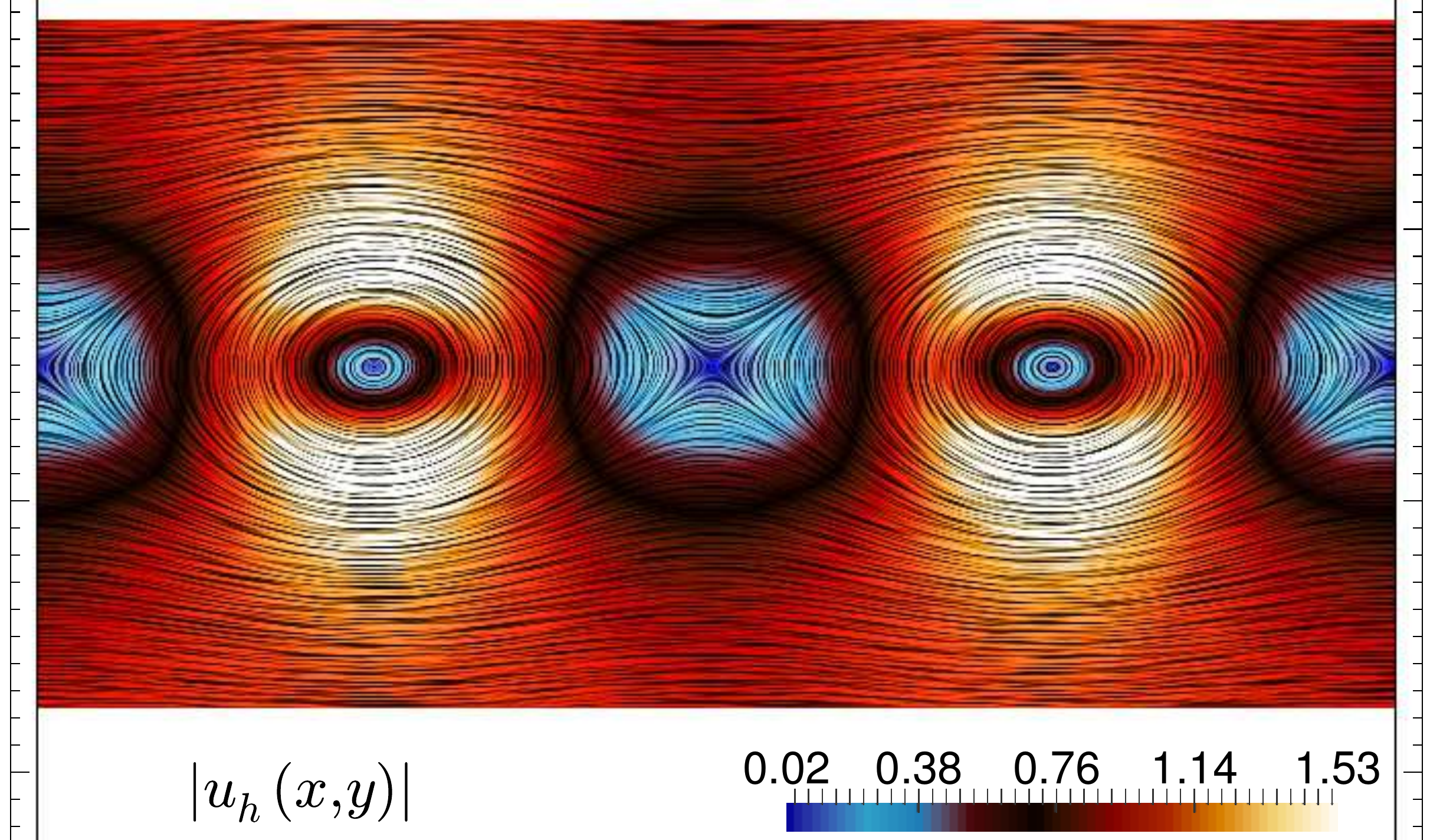}}\\
\subfigure[]{\includegraphics[width=0.325\textwidth]{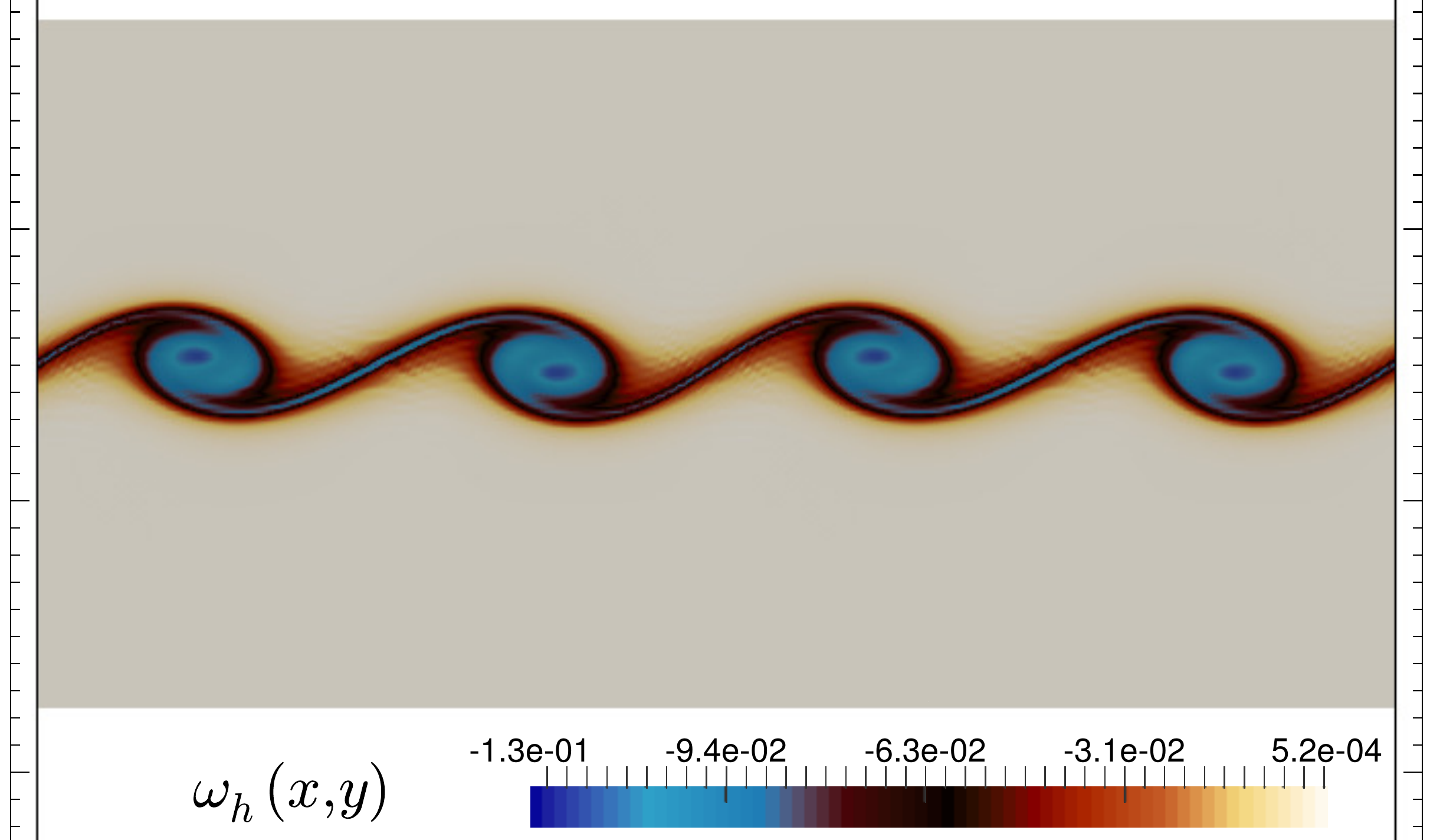}}
\subfigure[]{\includegraphics[width=0.325\textwidth]{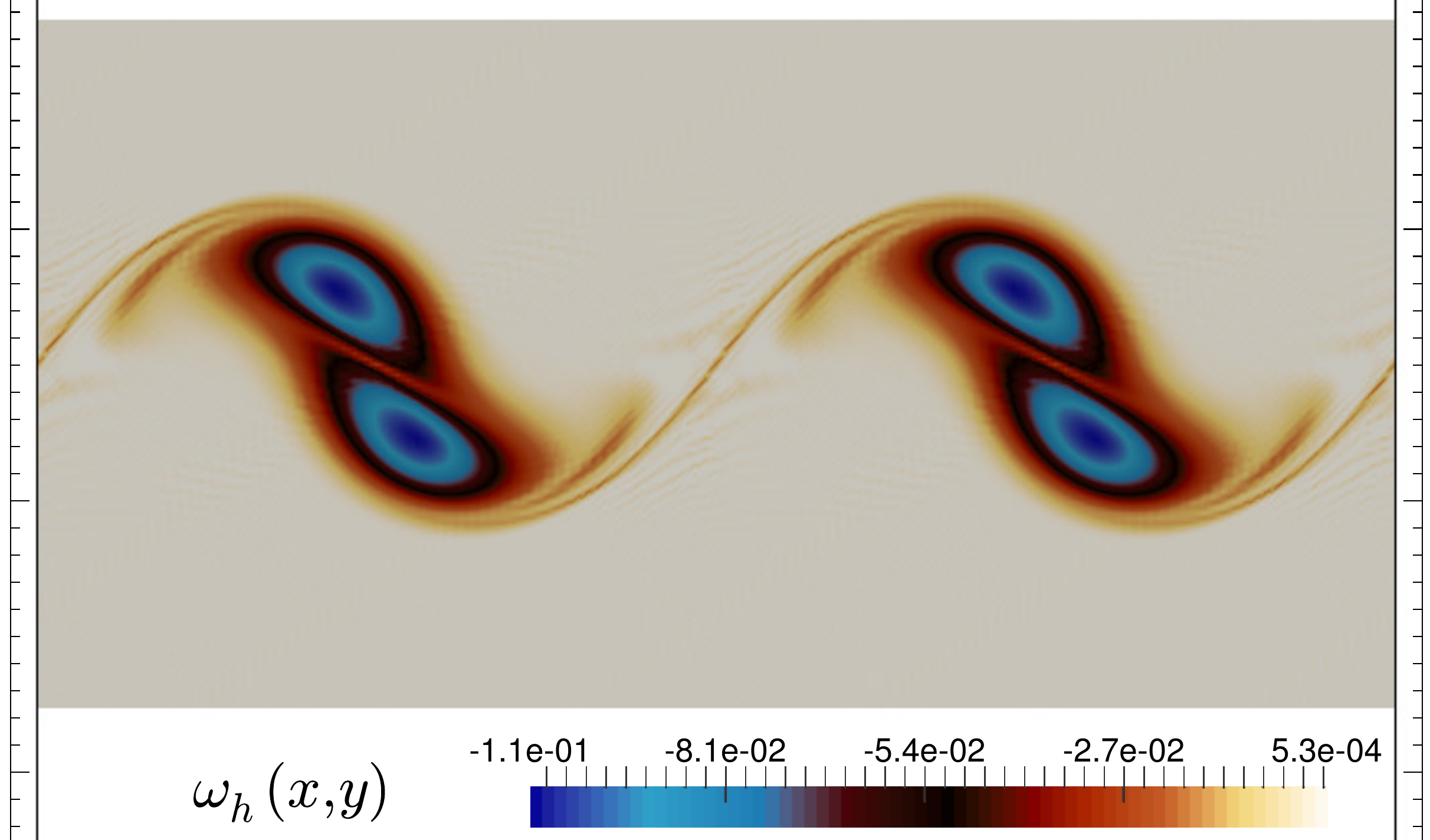}}
\subfigure[]{\includegraphics[width=0.325\textwidth]{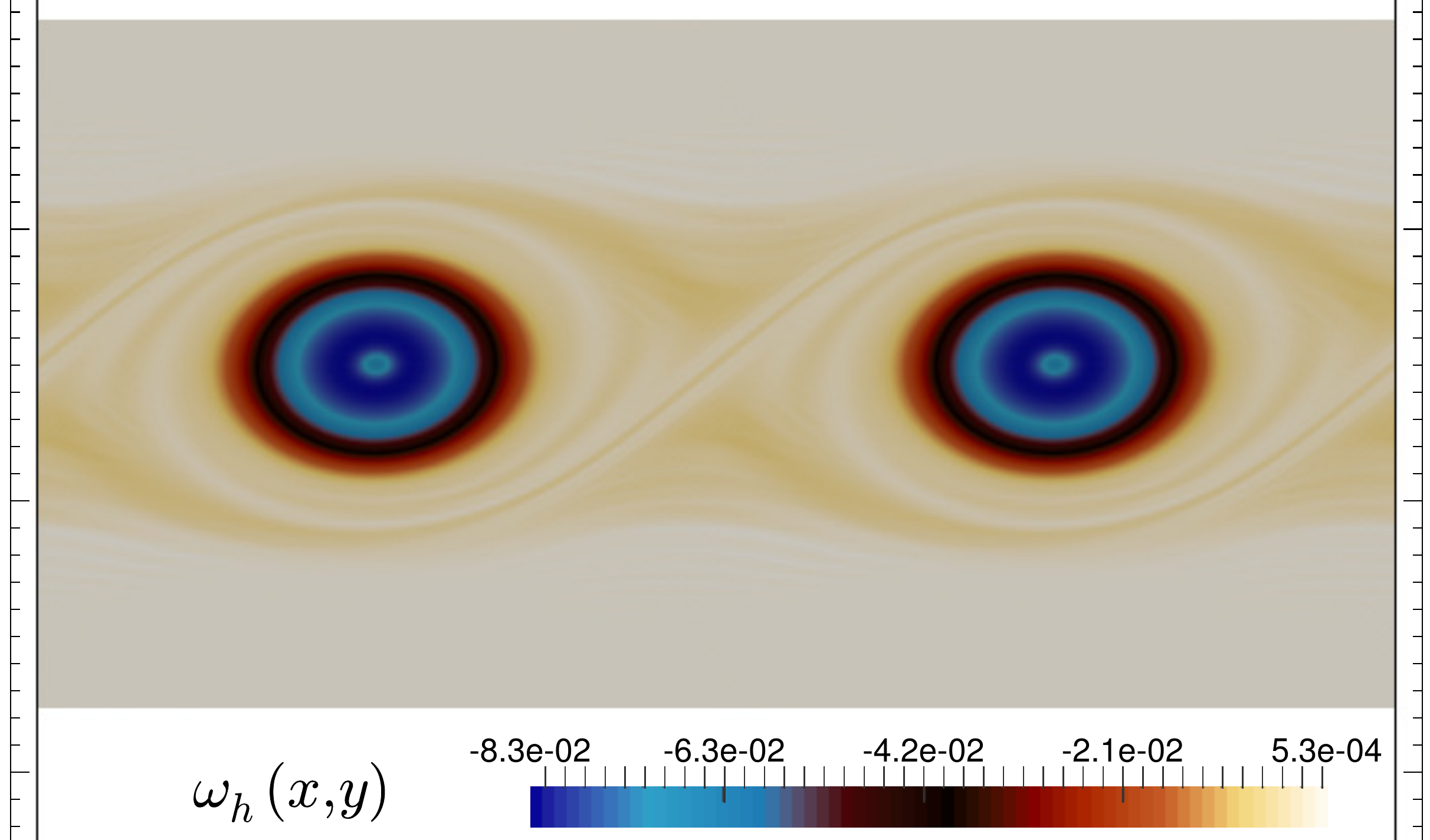}}\\
\subfigure[]{\includegraphics[width=0.325\textwidth]{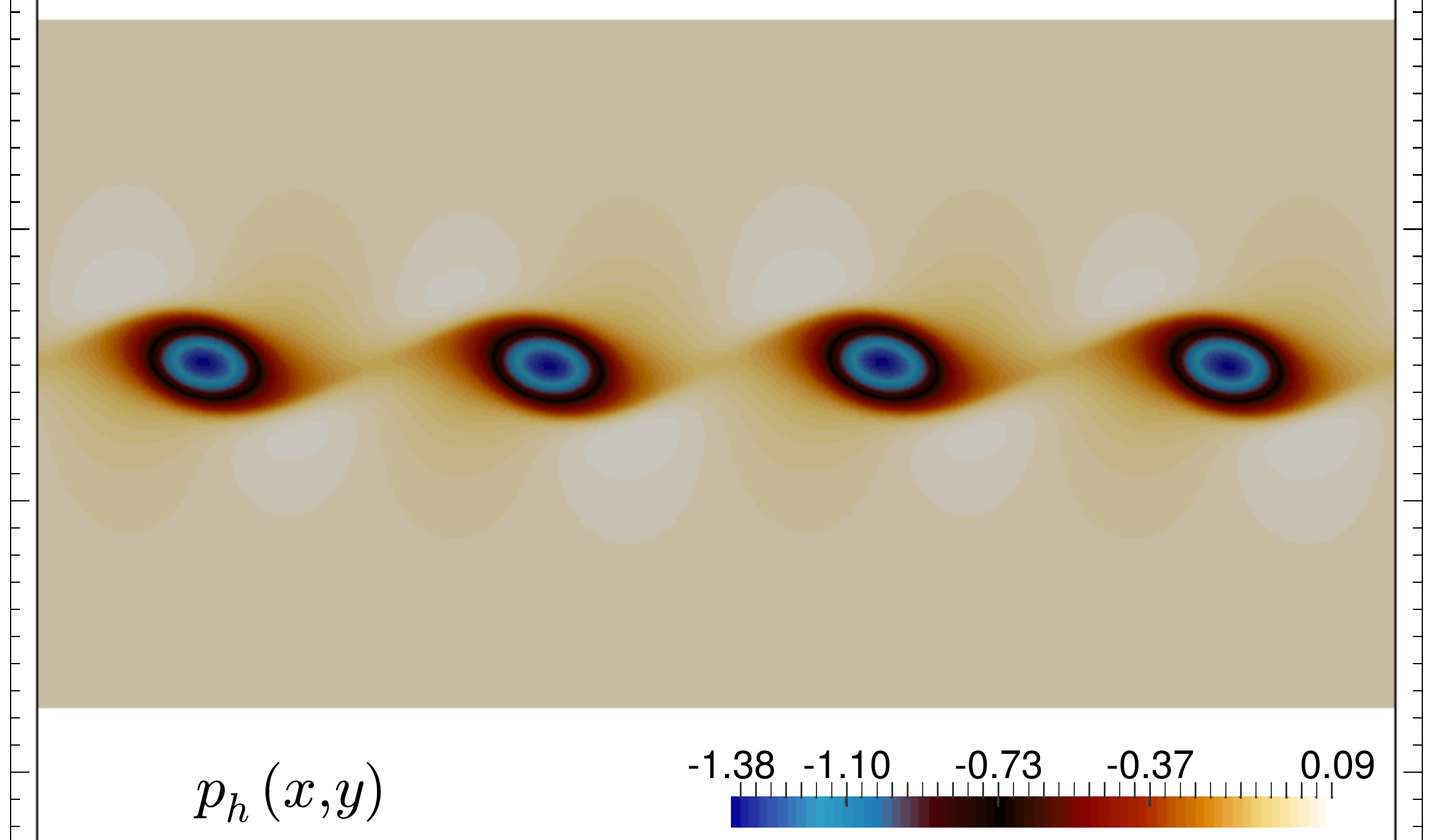}}
\subfigure[]{\includegraphics[width=0.325\textwidth]{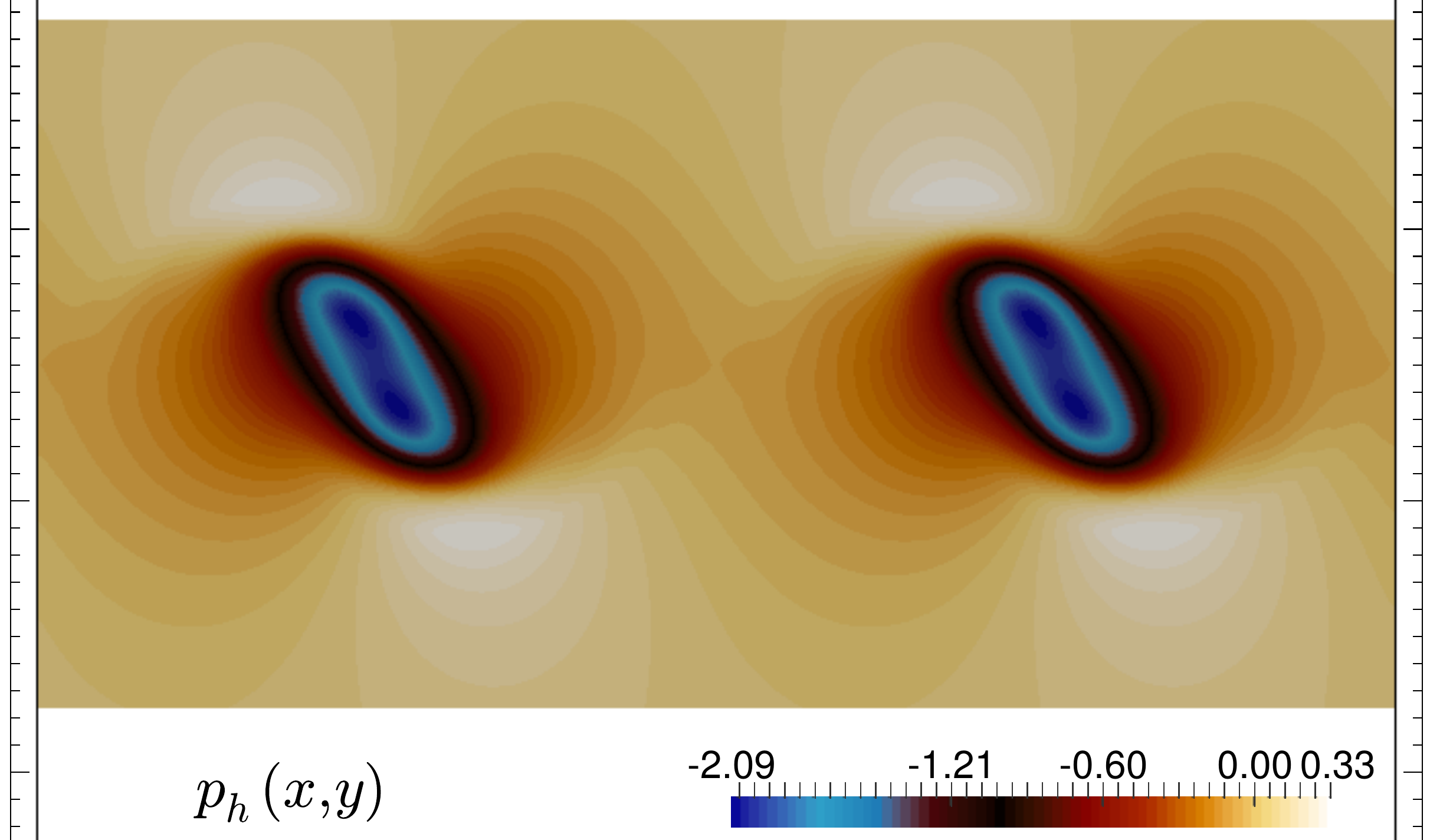}}
\subfigure[]{\includegraphics[width=0.325\textwidth]{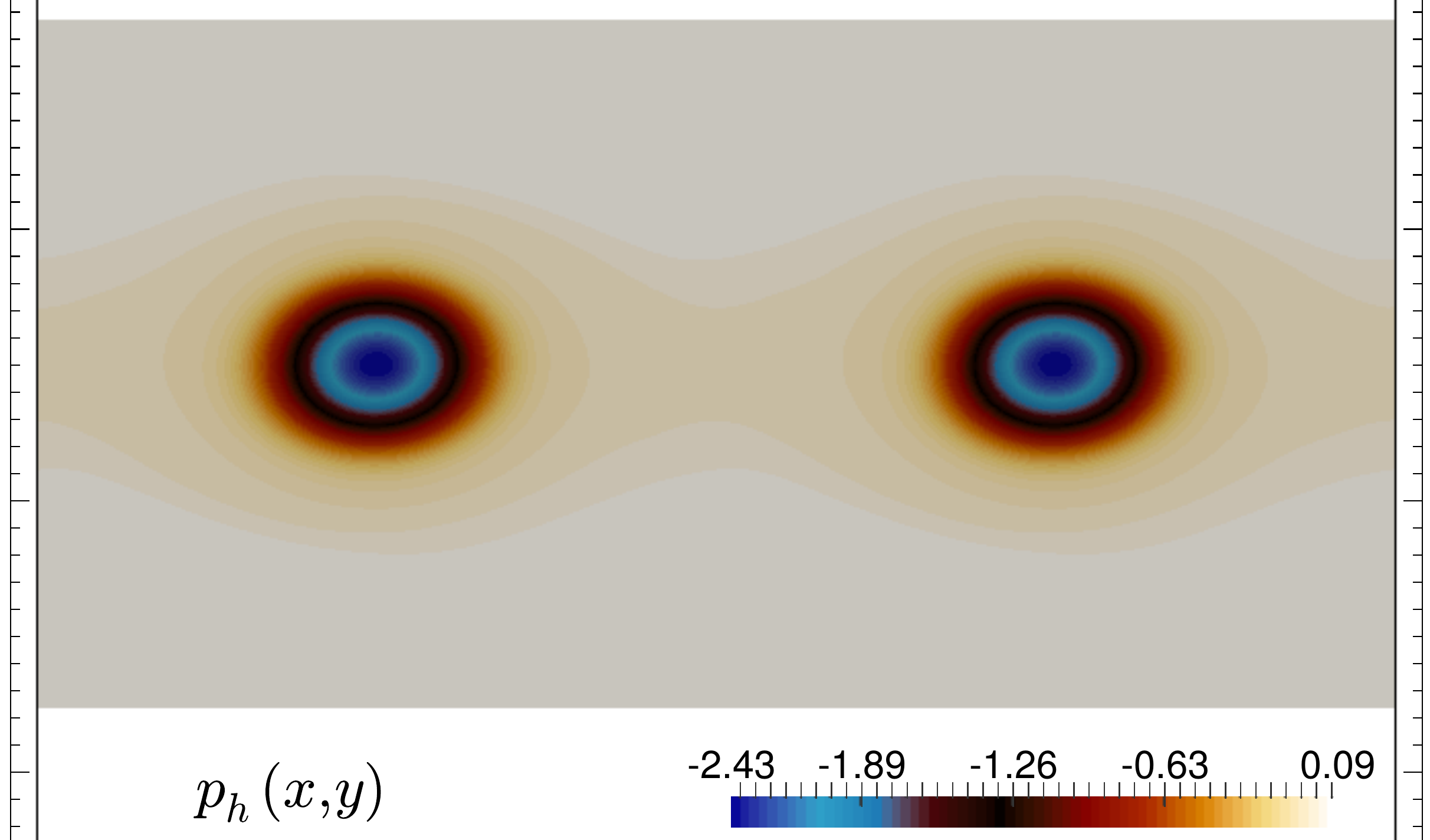}}
\end{center}

\vspace{-0.5cm}
\caption{Test~4. Transients of the Kelvin-Helmholtz instability problem computed 
with the first-order DG scheme \eqref{Mixed DG scheme}. Velocity magnitude (a-c), 
scalar vorticity (d-f), and Bernoulli pressure (g-i). }\label{fig:ex04}
\end{figure}

\paragraph{Test 4: Kelvin-Helmholtz mixing layer.} We close this section with a 
benchmark test related to the well-known vortex formation mechanisms known as the 
Kelvin-Helmholtz instability problem. The setup of the test follows the specifications 
in \cite{schroeder18} (see also \cite{burman17}), the transient Oseen equations are 
solved on the unit square $\Omega=(0,1)^2$ and the bottom and top walls constitute 
$\Gamma$, where we impose a 
free-slip velocity condition and zero vorticity. The left and right walls are regarded as a periodic boundary. 
The initial velocity is 
$$\bu = \begin{pmatrix}
u_{\infty}\tanh((2y-1)/\delta_0) - c_n u_{\infty}[\cos(w^a_{\infty}x)+\cos(w_{\infty}^bx)]\frac{(2y-1)}{\delta_0^2} \exp(-\frac{(y-1/2)^2}{\delta_0^2})\\
c_n u_{\infty} \exp(-\frac{(y-1/2)^2}{\delta_0^2})[w_{\infty}^a\sin(w_{\infty}^a x)+w_{\infty}^b\sin(w_{\infty}^bx)]\end{pmatrix},$$
with perturbation scaling $c_n = 0.001$, reference velocity $u_\infty = 1$, $w_{\infty}^a = 8\pi$, $w_{\infty}^b = 20\pi$, $\delta_0 = 1/28$. The 
characteristic time is $\bar{t}=\delta_0/u_{\infty}$, the Reynolds number is Re$=10000$, and the kinematic viscosity is $\nu = \delta_0u_{\infty}/\text{Re}$. 
We use a structured mesh of 128 segments per side, representing 131072 triangular elements, and we solve the problem 
using our first-order DG scheme, setting again the stabilisation constants to $a_{11}=c_{11}=\sigma= 1/\Delta t$ and $d_{11}=\nu$, 
where the timestep is taken as $\Delta t = \bar{t}/20$. The specification of this problem implies 
that the solutions will be quite sensitive to the initial 
perturbations present in the velocity, which will amplify and consequently vortices will appear. We proceed 
to compute numerical solutions until the dimensionless time $t = 7$, and present in Figure~\ref{fig:ex04} sample 
solutions at three different simulation times. For visualisation purposes we zoom into the region $0.25 \leq y \leq 0.75$, where 
all flow patterns are concentrated. 
In addition, a qualitative comparison against benchmark 
data from \cite{schroeder18} is presented in terms of the temporal evolution of the enstrophy $E(t)$ (here we rescale $\bomega_h$ 
with $\sqrt{\nu}$ to match again the real vorticity). We also record the evolution of the palinstrophy $P(t)$, 
a quantity that encodes the dissipation process.  These quantities are defined, respectively, as 
$$E(t):=\frac{1}{2\nu}  \| \bomega_h(t) \|^2_{0,\Omega},\qquad 
P(t):=\frac{1}{2\nu}  \| \nabla\bomega_h(t) \|^2_{0,\Omega},$$
and we remark that for the palinstrophy we use the discrete gradient associated with the DG discretisation.  
We show these quantities in Figure~\ref{fig:ex04comp}, where also include 
results from \cite{schroeder18} that correspond to coarse and fine mesh solutions of the Navier-Stokes equations 
using a high order scheme based on Brezzi-Douglas-Marini elements.

\begin{figure}[h!]
\begin{center}
\subfigure[]{\includegraphics[width=0.4\textwidth]{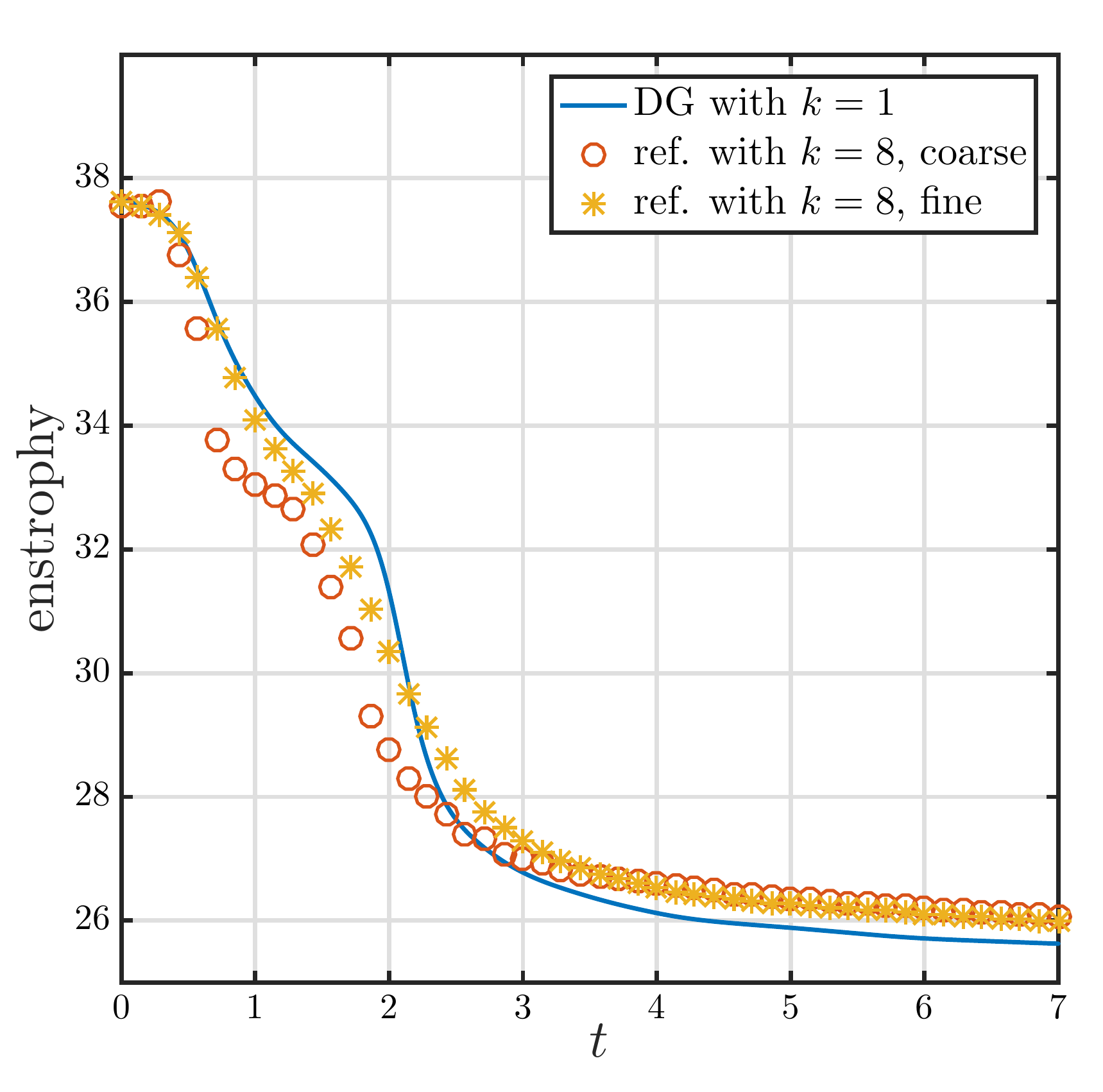}}
\subfigure[]{\includegraphics[width=0.4\textwidth]{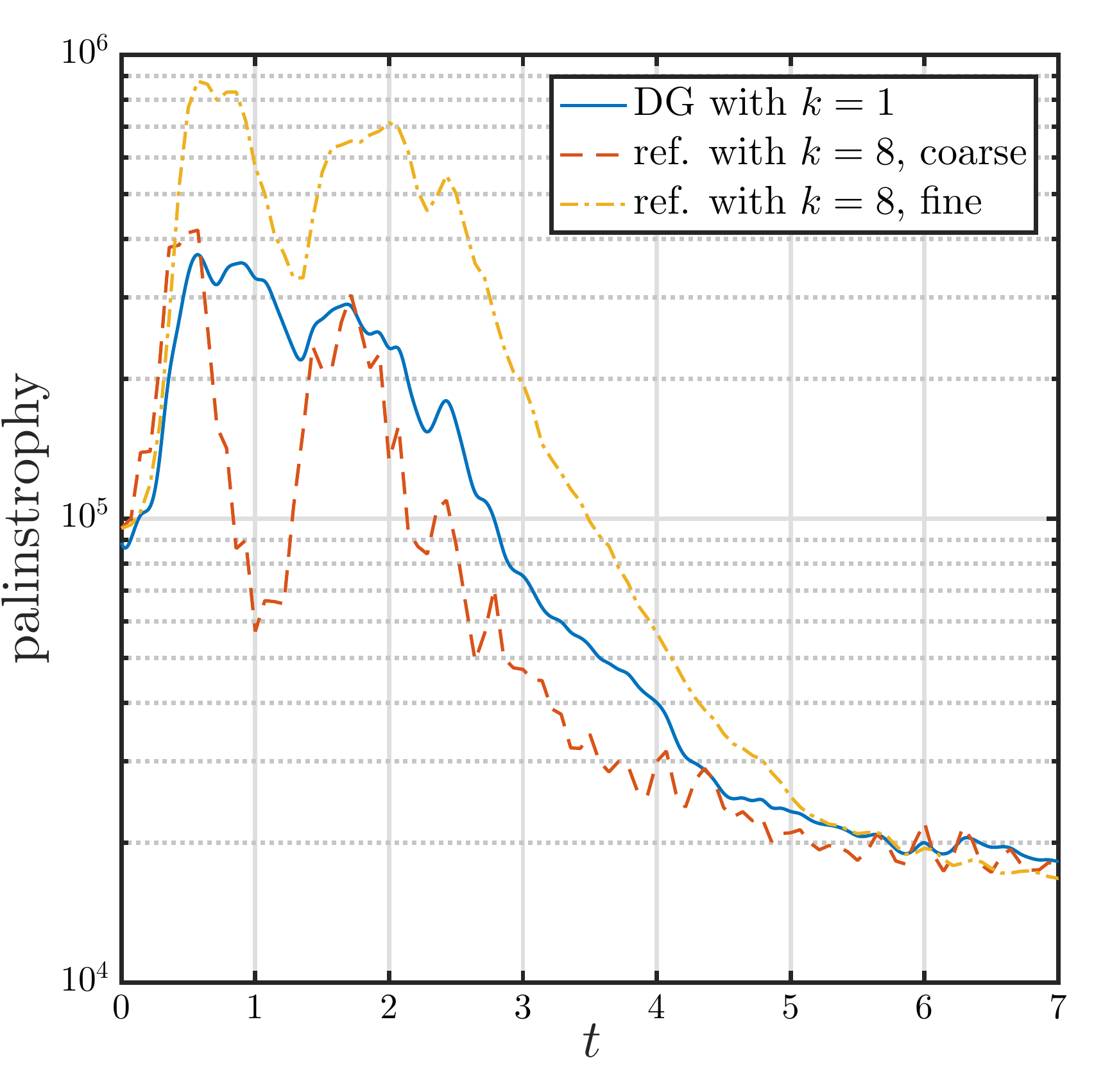}}\\[-1ex]

\caption{Test~4. Time evolution of enstrophy (a) and palinstrophy (b) in the Kelvin-Helmholtz 
mixing layer formation. Reference values correspond to computations from \cite{schroeder18} 
using BDM elements of order 8 on a coarse and on a fine mesh. }\label{fig:ex04comp}
\end{center}
\end{figure}

\small\paragraph{Acknowledgements.} This work has been supported by 
CNRS though the PEPS programme; by 
CONICYT - Chile through FONDECYT project 11160706;
by DIUBB, Universidad del B\'io-B\'io through projects 165608--3/R and 171508 GI/VC;
by DIDULS, Universidad de La Serena through  project PR17151;
and by the EPSRC through the research grant EP/R00207X/1.

\end{document}